\documentclass[12pt]{amsart}

\usepackage{ishai}

\title{Explicit Motivic Mixed Elliptic Chabauty-Kim}

\author{David Corwin}

\thanks{} 

\date{\today}

\usepackage{fullpage}
\usepackage{colonequals}
\usepackage{microtype}

\begin{document}

\maketitle

\raggedbottom
\SelectTips{cm}{11}

\begin{abstract}The main point of the paper is to take the explicit motivic Chabauty-Kim method developed in papers of Dan-Cohen--Wewers and Dan-Cohen and the author and make it work for non-rational curves. In particular, we calculate the abstract form of an element of the Chabauty-Kim ideal for $\Zb[1/\ell]$-points on a punctured elliptic curve, and lay some groundwork for certain kinds of higher genus curves. For this purpose, we develop an ``explicit Tannakian Chabauty-Kim method'' using $\Ql$-Tannakian categories of Galois representations in place of $\Qb$-linear motives. In future work, we intend to use this method to explicitly apply the Chabauty-Kim method to a curve of positive genus in a situation where Quadratic Chabauty does not apply.
\end{abstract}

\tableofcontents

\section{Introduction}

\subsection{Extended Abstract}

The main point of the paper is to take the explicit motivic Chabauty-Kim method developed in \cite{MTMUE} and \cite{PolGonI} and make it work for non-rational curves. In particular, we calculate the abstract form of an element of the Chabauty-Kim ideal for $\Zb[1/\ell]$-points on a punctured elliptic curve (Theorem \ref{thm:CK_ideal_element_form1} and \S \ref{sec:elimination}) and lay some groundwork for certain kinds of higher genus curves (\S \ref{sec:decomp_higher_genus}). Some important themes include:
\begin{itemize}
    \item Use of the Bloch-Kato conjectures and Poitou-Tate duality for explicitly bounding dimensions of Selmer groups (\S \ref{sec:BK_selmer}, \S \ref{sec:global_ranks}, \S \ref{sec:motivic_decomp})
    \item Use of $\Ql$-linear categories of Galois representations in place of $\Qb$-linear motives (\S \ref{sec:mixed_elliptic})
    \item Methods for dealing with primes of bad reduction via results of Betts--Dogra in \cite{BettsDogra20} (\S \ref{sec:local_bad}, \S \ref{sec:local_comp_bad})
    \item A setup for Tannakian Selmer varieties (\S \ref{sec:tannakian_selmer}) and a description in terms of cocycles (Theorem \ref{thm:cohom}) generalizing \cite[Proposition 5.2]{MTMUE}
    \item A $p$-adic period map for mixed elliptic motivic periods (\S \ref{sec:p-adic_periods})
\end{itemize}

\subsection{The Problem of Effective Faltings}

Let $X$ be a smooth proper curve of genus $g \ge 2$ over a number field $k$. The theorem of Faltings states that $X(k)$ is finite. A major open question is to find an algorithm for determining the finite set $X(k)$ given $X/k$.

More generally, the combination of the theorems of Faltings and Siegel imply that whenever $X$ is a smooth curve with negative (geometric) Euler characteristic, and $S$ is a finite set of places of $k$, we have $\Xc(\Oc_{k,S})$
%\footnote{Throughout the introduction, when we write $X(\Oc_{k,S})$, we mean $\Xc(\Oc_{k,S})$ for an appropriate integral model $\Xc$ of $X$. We will worry about the subtleties of choosing an integral model only later in \S  \ref{sec:selmer_varieties}.}
finite, for an $\Oc_{k,S}$-model\footnote{We define an $\Oc_{k,S}$-model as a finite type, separated, faithfully flat scheme $\Xc$ over $\Oc_{k,S}$ with an isomorphism $\Xc_k \to X$.} $\Xc$ of $X$. This formalism includes the case of rational points, as $\Xc(\Oc_{k,S})=X(k)$ whenever $X$ is proper.\footnote{One might then prefer to replace $\Xc$ by $X$ and $\Oc_{k,S}$, but as hinted at in \S  \ref{sec:S-integral_elliptic} and described in more detail in \S  \ref{sec:bad_outside_S}, it is best to set $S=\emptyset$ and thus consider $\Oc_k$-points in the proper case.} We are thus interested in the general question of determining the set $\Xc(\Oc_{k,S})$ for $X/k$ and a finite set $S$ of places of $k$.

In practice, one may often conjecturally find the set $\Xc(\Oc_{k,S})$ by searching over points of bounded height. This produces a finite set of elements of $\Xc(\Oc_{k,S})$, and one hopes, after a dilligent enough search that this is all of $\Xc(\Oc_{k,S})$. The challenge is in proving that one has found all of $\Xc(\Oc_{k,S})$.

\subsection{The Chabauty-Skolem Method}\label{sec:classical_chab_cole_skol}

Before Faltings' proof in 1983, the primary method for proving finiteness of $\Xc(\Oc_{k,S})$ was via the method of Chabauty-Skolem (\cite{Chabauty41}). In the 1980's, Chabauty's method was upgraded to an effective method by Coleman (\cite{ColemanChabauty}), using his theory known as ``Coleman integration.''
%\emph{For simplicity, we restrict to the case that $X$ is proper for the rest of \S \ref{sec:classical_chab_cole_skol}.}
More specifically, using the generalized Jacobian\footnote{This is characterized by the fact that $J$ is a semi-abelian variety and that the closed embedding $X \hookrightarrow J$ is an isomorphism on first homology.} $X \hookrightarrow J$, a basepoint $b \in \Xc(\Oc_{k,S})$,\footnote{More generally, an integral tangential basepoint is allowed when $X$ is not proper.} and a finite place $\pf$ of $k$ not in $S$ with $k_{\pf} \cong \Qp$, one constructs a diagram
\[
\xymatrix{
\Xc(\Oc_{k,S}) \ar[r] \ar[d] & \Xc(\Oc_{\pf}) \ar[d] \ar[dr]^{\int} \\
\Jc(\Oc_{k,S}) \ar[r]_-{\mathrm{loc}} & \Jc(\Oc_{\pf}) \ar[r]_-{\log} & \Lie J_{\Qp}
}
\]
for an appropriate integral model $\Jc$ of $J$ when $X$ is not proper. 

By definition, $\Lie J_{k_{\pf}}$ is the tangent space to $J_{k_{\pf}}$ at the identity. When $J$ is proper, it is the linear dual of $H^0(J_{k_{\pf}},\Omega^1)$, and more generally, it is dual to the subspace $H^0(J_{k_{\pf}},\Omega^1)^J$ of translation invariant differential $1$-forms. The map $\int$ sends $z \in \Xc(\Oc_{\pf})$ to the functional sending $\omega \in H^0(J_{k_{\pf}},\Omega^1)^J$ to the Coleman integral
\[
\int_b^z \omega.
\]

If $r \coloneqq \rank_{\Qb} J(k) < g$, then the image of $J(k)$ in $\Lie J_{k_{\pf}}$ lies in a vector subspace of positive codimension. Therefore, the annihilator $\Ic_J$ of $\loc(\Jc(\Oc_{k,S}))$ in $(\Lie J_{k_{\pf}})^{\vee} = H^0(J_{k_{\pf}},\Omega^1)^J$ is nonzero, so its pullback $\int^{\#}(\Ic_J)$ is a nonzero set of functions on $\Xc(\Oc_{\pf})$ that vanish on $\Xc(\Oc_{k,S})$. By the theory of Coleman, each nonzero $f \in \int^{\#}(\Ic_J)$ has finitely many zeroes, and the theory of Newton polygons allows one to $p$-adically estimate the set $\Xc(\Oc_{\pf})_J$ of common zeroes of all $f \in \int^{\#}(\Ic_J)$. More details may be found in \cite{McCallumPoonen}.

If $r \ge g$, one is usually out of luck with Chabauty's method. Moreover, even if $r < g$, $\Xc(\Oc_{\pf})_J\setminus \Xc(\Oc_{k,S})$ might be nonempty, and the fact that computations of zeroes are $p$-adic approximations means that one cannot then use $\int^{\#}(\Ic_J)$ to determine $\Xc(\Oc_{\pf})$ (though it can be successful in conjunction with other methods; see \cite[\S 5.3]{PoonenComputing}).

\subsection{Non-Abelian Chabauty's Method}

The non-abelian Chabauty's method of Minhyong Kim (\cite{kim05}, \cite{kim09}), also known as Chabauty-Kim, allows one to remove this restriction. For an integer $n$\footnote{Often called the \emph{depth}, although this conflicts with the notion of \emph{depth} in the theory of multiple zeta values.} and basepoint $b \in \Xc(\Oc_{k,S})$, Kim constructs a diagram

\[
\xymatrix{
\Xc(\Oc_{k,S}) \ar[r] \ar[d]^{\kappa} & \Xc(\Oc_{\pf}) \ar[d]_{\kappa_{\pf}} \ar[dr]^{\int} \\
\operatorname{Sel}(\Xc)_n \ar[r]_-{\mathrm{loc}_n} & \operatorname{Sel}(\Xc/\Oc_{\pf})_n \ar[r]_-{\log_{\mathrm{BK}}}^{\sim} & U_n/F^0 U_n
}
\]

The set $\operatorname{Sel}(\Xc)_n$ is the non-abelian cohomology set $H^1_{f,S}(G_k;U_n)$,\footnote{Technically, one needs to modify $H^1_{f,S}(G_k;U_n)$ by either expanding $S$ to include all places of bad reduction of $X$ or by taking a finite union of twists as described in \S \ref{sec:selmer_varieties}.} where $f$ refers to the Selmer conditions of Bloch-Kato, and $U_n$ is the $n$th descending central series quotient of the $\Qp$-unipotent geometric fundamental group of $X$ (based at $b$). More details may be found in \cite[\S 4]{ChabautytoKim}.

For $n=1$, this is essentially the same as the diagram in classical Chabauty's method. More precisely, $\operatorname{Sel}(\Xc)_1$ is the $p$-adic Selmer group of $J$, and we have an embedding
\[
\kappa_J \colon \Jc(\Oc_{k,S}) \otimes \Qp \hookrightarrow \operatorname{Sel}(\Xc)_1
\]
that is conjecturally (by finiteness of $\Sha(J)$) an isomorphism, and verifiably so in practice.

We define the \emph{Chabauty-Kim locus}:
\[
\Xc(\Oc_{\pf})_n \colonequals \kappa_{\pf}^{-1}(\operatorname{Im}(\mathrm{loc}_n)) = \int^{-1}(\operatorname{Im}(\log_{\mathrm{BK}} \circ \mathrm{loc}_n)).
\]

The \emph{Chabauty-Kim ideal}
\[
\Ic_{CK,n}(\Xc)
\]
of regular functions vanishing on the image of $\loc_n$ pulls back to a set $\kappa_{\pf}^{\#}(\Ic_{CK,n})$ of functions on $\Xc(\Oc_{\pf})$ vanishing on $\Xc(\Oc_{k,S})$ with $\Xc(\Oc_{\pf})_n$ as its set of common zeroes. As long as $\kappa_J$ is an isomorphism, $\Ic_{CK,1}$ is the ideal generated by $\Ic_J$, and $\Xc(\Oc_{\pf})_1 = \Xc(\Oc_{\pf})_J$. For $n \ge 1$, \[\Xc(\Oc_{\pf})_{n+1} \subseteq \Xc(\Oc_{\pf})_n.\]

When
\begin{equation}\label{eqn:selmer_inequality}
\dim_{\Qp} \operatorname{Sel}(\Xc)_n < \dim_{\Qp} \operatorname{Sel}(\Xc/\Oc_{\pf})_n,
\end{equation}
the set $\Xc(\Oc_{\pf})_n$ is finite, a consequence of the fact that $\kappa_{\pf}$ has Zariski dense image (\cite[Theorem 1]{kim09}).

Kim shows (\cite[Theorem 2]{kim09}) that this inequality holds for sufficiently large $n$ if a part of the Bloch-Kato Conjecture (see Conjecture \ref{conj:BK} below) holds.

The following appears as \cite[Conjecture 3.1]{nabsd} for $S=\emptyset$ and as a remark about what one ``might conjecture'' in \cite[\S 8]{nabsd}:

\begin{conj}[Kim's Conjecture]\label{conj:kim}
For $k=\Qb$, a regular minimal model\footnote{Let $g$ denote the genus of the smooth projective closure $\overline{X}$ of $X$. Then $\Xc$ is a \emph{regular minimal model} if it is the complement of a reduced horizontal divisor in the regular minimal model $\overline{\Xc}$ of $\overline{X}$ over $\Oc_{k,S}$ (resp. in $\Pb^1/\Oc_{k,S}$) when $g \ge 1$ (resp. when $g=0$).} $\Xc$ and $n$ sufficiently large, we have
\[
\Xc(\Oc_{\pf})_n = \Xc(\Oc_{k,S}).
\]
\end{conj}

\begin{rem}
It is mentioned in \cite[Remark 3.2]{nabsd} that there should be a suitable generalization to all number fields $k$. One might expect such a conjecture to follow from the recent ideas of \cite{dogra2020unlikely}, although no such conjecture is contained therein.
\end{rem}

The intuition behind this conjecture is that a random $p$-adic analytic function should not vanish at a given point unless it has a very good reason to do so.

This conjecture implies that if we can compute $\Xc(\Oc_{\pf})_n$ up to some $p$-adic precision, then there is an effective version of Faltings' Theorem. More precisely, if we have a subset $F$ of $\Xc(\Oc_{k,S})$, then to check that $F=\Xc(\Oc_{k,S})$, we need only find some $n$ for which $|\Xc(\Oc_{\pf})_n| = |F|$. In particular, we need only the part of the theory of Newton polygons that determines the number of zeroes of a $p$-adic power series in a residue disc.

The set $\Xc(\Oc_{\pf})_n$ is defined as the common zero set of the pullbacks under $\int$ of the set of functions vanishing on the image of $\loc_n$. Therefore, modulo Conjecture \ref{conj:kim}, effective Faltings over $\Qb$ is reduced to the problem of computing, up to some $p$-adic precision, the set of functions on $U_n/F^0 U_n$ that vanish on the image of $\loc_n$.

\subsection{Quadratic Chabauty}

The most successful method to-date for computing with non-abelian Chabauty is the Quadratic Chabauty method of Balakrishnan et al (\cite{BalakrishnanIntegral}, \cite{BalakrishnanDograI}). This method essentially computes part of $\loc_2$ using the observation of Kim that a certain coordinate of $\operatorname{Sel}(\Xc)_2$ corresponds to the $p$-adic height pairing on $J$, while a similar coordinate of $\operatorname{Sel}(\Xc/\Oc_{\pf})_2$ corresponds to the local component at $\pf$ of the $p$-adic height pairing. (More precisely, \cite{BalakrishnanIntegral} covers the case of integral points on affine curves using the Coleman-Gross $p$-adic height pairing, while \cite{BalakrishnanDograI} covers the case of rational points on proper curves using a whole slew of pairings defined relative to a certain kind of divisor on $X \times X$.)

In particular, let $\rho \colonequals \rank_{\Qb} \operatorname{NS}(J)$, with $X$ proper. Then \cite[Lemma 3.2]{BalakrishnanDograI} shows that $\Xc(\Oc_{\pf})_2$ is finite whenever \[r < g + \rho - 1,\] and moreover, \cite[\S 8]{BalakrishnanDograI} gives a method to $p$-adically approximate a set containing $\Xc(\Oc_{\pf})_2$ in this case. In \cite{BalakrishnanDograII}, the authors relax the condition (partly dependent on Conjecture \ref{conj:BK}), but still restrict to the case $n=2$.

\subsection{Explicit Motivic Non-Abelian Chabauty}

The only computed cases of Chabauty-Kim going beyond $n=2$ are for $S$-integral points on $\Xc=\Poneminusthreepoints$. More specifically, the case of $S=\emptyset, \{2\}$ with $n=4$ appeared in \cite{MTMUE}, while $S=\{3\}$, $n=4$ appeared in \cite{PolGonI}. Both cases were also done in \cite{BrownIntegral}.

To compute $\Xc(\Oc_{\pf})_n$, one needs to understand $\mathrm{Sel}(\Xc)_n$ and the map $\loc_n$ concretely. As mentioned above, (the set of $\Qp$-points of) $\mathrm{Sel}(\Xc)_n$ is (modulo technicalities described in \S \ref{sec:selmer_varieties}) $$H^1_{f,S}(G_k;U_n),$$ the set of cohomology classes of $G_k$ with coefficients in $U_n$ that are crystalline at primes over $p$ and unramified outside $S \cup \{p\}$. Both the group $G_k$ and the local conditions are hard to understand explicitly.

The key insight is that one need only understand the \emph{category} of continuous $p$-adic representations of $G_k$ that appear in $U_n$ and its torsors. This category is Tannakian, and $\mathrm{Sel}(\Xc)_n$ may be described as group cohomology of the Tannakian fundamental group with coefficients in $U_n$.

In the case of $\Xc=\Poneminusthreepoints$, the relevant category is the category $$\Rep_{\Qp}^{\MT}(\Oc_{k,S})$$ of mixed Tate geometric Galois representations with good reduction over $\Oc_{k,S}$. Its Tannakian fundamental group \[\pi_1^{\MT}(\Oc_{k,S})\] is isomorphic to an extension of $\Gm$ by a pro-unipotent group. The pro-unipotent group may be determined by computing the Bloch-Kato Selmer groups $H^1_f(G_k;\Qp(n))$ for each $n$, and these are known (\cite{Soule79}). In this way, the set
$$H^1_{f,S}(G_k;U_n),$$
becomes simply the group cohomology set
\[
H^1(\pi_1^{\MT}(\Oc_{k,S});U_n),
\] with no further local conditions other than those encoded in the group $\pi_1^{\MT}(\Oc_{k,S})$ itself.

\begin{rem}\label{rem:Qb_linear}
In fact, $\Rep_{\Qp}^{\MT}(\Oc_{k,S})$ has a $\Qb$-form
$$\MT(\Oc_{k,S}, \QQ),$$
the category of mixed Tate motives over $\Oc_{k,S}$ with coefficients in $\Qb$. This latter category was defined in \cite{DelGon05} and first applied to Selmer varieties in \cite{HadianDuke} and \cite{MTMUE}. The group $\pi_1^{\MT}(\Oc_{k,S})$ is $\pi_1(\MT(\Oc_{k,S}, \QQ))$, while the analogue of what we do in this paper is its tensorization with $\Qp$.

As such a $\Qb$-linear category is only conjectural in the mixed elliptic case (see \cite{Patashnick13}), we work with categories of Galois representations. While it is in some ways nicer to work over $\Qb$, it is not necessary, as the end result of the Chabauty-Kim method is still $p$-adic.
\end{rem}

\subsection{This Work: Explicit Tannakian Non-Abelian Chabauty}

The main goal of this paper is to extend the methods of \cite{MTMUE} and \cite{PolGonI} (see also \cite{MTMUEII} and \cite{PolGonII}) beyond the mixed-Tate case. We develop some general foundations that we expect to apply to all curves, and we do some more explicit computations in the \emph{mixed elliptic case}. This is the case in which the Jacobian $J$ of the smooth compactification $\overline{X}$ of $X$ is isogenous to a power of an elliptic curve.

We set up some explicit examples when $X = E' = E \setminus \{O\}$ is a punctured elliptic curve, and $\Oc_{k,S}=\Zb[1/2]$. This is arguably the simplest non-rational example in which the Chabauty-Kim locus is infinite for $n=2$. We plan to compute these examples when the necessary code for computing Coleman integrals is fully available.

More precisely, we replace $\Rep_{\Qp}^{\MT}(\Oc_{k,S})$ with a category
\[
\Rep_{\Qp}^{\rm f,S}(G_k,E)
\]
of $S$-integral mixed elliptic Galois representations; i.e., iterated extensions of Galois representations appearing in tensor powers of the Tate module \[h_1(E) \colonequals H_1^{\et}(E_{\overline{k}};\Qp)\] of $E$. The fundamental group
\[
\pi_1^{\ME}(\Oc_{k,S}, E) \colonequals \pi_1(\Rep_{\Qp}^{\rm f, S}(G_k,E))
\]
of this Tannakian category is now an extension of $\mathrm{GL}_2$ by a pro-unipotent group $U(\Oc_{k,S},E)$, the latter of which is determined by Bloch-Kato Selmer groups
\[
H^1_{f,S}(G_k;M_{a,b}),
\]
where
\[M_{a,b} \colonequals \mathrm{Sym}^a(h_1(E))(b),\]
runs over all irreducible representations of $\mathrm{GL}_2$; i.e., over $a,b \in \Zb$ with $a \ge 0$.

We explain how to compute the dimensions of $H^1_{f,S}(G_k;M_{a,b})$ while assuming the Bloch-Kato conjectures when $k$ is $\Qb$ or a quadratic imaginary field (\S \ref{sec:global_ranks}). We show how to use this to explicitly bound the dimension of the level $n$ Selmer variety of a curve of mixed elliptic type (\S \ref{sec:motivic_decomp}).

For an elliptic curve over $\Qb$ with ordinary reduction, the necessary cases of the Bloch-Kato conjecture with $a=1$ may be verified using \cite[Theorem 18.4]{Kato04} and computations of $p$-adic L-functions. More generally, other relevant cases should be accessible using \cite{Allen16} (for $a=2$) and work of Loeffler--Zerbes (for $a=3$). We review this in \S \ref{sec:verifying_conjectures}.

We explain the basic setup of Selmer varieties in the Tannakian formalism in \S \ref{sec:tannakian_selmer}-\ref{sec:selmer_localization}. Most notably, we prove an explicit description of cohomology sets, generalizing \cite[Proposition 5.2]{MTMUE}, in Theorem \ref{thm:cohom}. This is necessary in order to carry out explicit motivic Chabauty-Kim in general.

We give an example computation of an element of $\Ic_{CK,n}(\Xc)$ for $n=3$ and $\Xc$, $\Oc_{k,S}=\Zb[1/\ell]$, and $\Xc$ the punctured minimal Weierstrass model of an elliptic curve, using the Tannakian formalism (\S \ref{sec:geometric_step}). From it, we deduce the following:

\begin{thm}[Theorem \ref{thm:CK_ideal_element_form}]\label{thm:CK_ideal_element_form1}
Let $\Ec$ be the minimal Weierstrass model of an elliptic curve $E$ over $\Qb$ with $p$-Selmer rank $1$, let $\alpha$ denote a choice of component of the N\'eron model of $E$ at each place of $\Qb$, and let $S=\{\ell\}$ for some prime $\ell \neq p$. Then assuming Conjecture \ref{conj:BK} for $h^1(E)$, there is a function of the form
\[
c_1 J_4 + c_2 J_3 + c_3 J_1 J_2 + c_4 J_1^3 + c_5 J_1
\]
vanishing on the subset $\Ec'(\Zb[1/S])_{\alpha}$ of $\Ec'(\Zb[1/S])$ reducing to $\alpha$ at each bad prime of $\Ec$, in which:
\begin{itemize}
    \item The $c_i \in \Qp$ arising as periods of elements of $\Oc(W)$ (c.f. \S \ref{sec:selmer_localization}), not all of which are zero, and
    \item The $J_i$ are explicit iterated integrals on $E'$ defined in \S \ref{sec:local_selmer_coords}.
\end{itemize}

\end{thm}

\begin{rem}\label{rem:can_verify}
As noted in \S \ref{sec:verifying_conjectures}, one can often computationally verify Conjecture \ref{conj:BK} for $h^1(E)$ using $p$-adic $L$-functions. Therefore, for individual elliptic curves, one may apply Theorem \ref{thm:CK_ideal_element_form1} without relying on any conjectures.
\end{rem}

\begin{rem}\label{rem:function_doesnt_vanish}
While Theorem \ref{thm:CK_ideal_element_form1} does not formally assert that the function is nonzero, it would follow from the existence of a $\Qb$-linear category of mixed elliptic motives and a $p$-adic Kontsevich-Zagier period conjecture (generalizing \cite[Conjecture 4]{YamashitaBounds}) for this category that the function is nonzero. Furthermore, one may in practice verify that the $c_i$ are not all zero by computing Coleman integrals.
\end{rem}

We analyze some explicit examples, corresponding to the curves `102a1' and `128a2' (\S \ref{sec:explicit_examples}). We describe the appropriate integral models and list known $\Zb[1/2]$-points. Most notably, we explain how to deal with primes of bad reduction in \S \ref{sec:local_comp_bad}, based on the theory of \cite{BettsDogra20} as described in \S \ref{sec:bad_outside_S}.

\subsubsection{Paper Outline}

In \S  \ref{sec:BK}, we review some background on Galois representations and Bloch-Kato Selmer groups. This includes the statement of the part of the Bloch-Kato Conjecture we need (Conjecture \ref{conj:BK}). We also introduce certain Grothendieck groups ($K_0$) of Galois representations to be used in \S \ref{sec:motivic_decomp}.

In \S  \ref{sec:selmer_varieties}, we describe the different kinds of Selmer varieties we use, paying careful attention to integral models and local conditions at primes of bad reduction. We review the results of \cite{BettsDogra20} as they apply to our current work.

In \S  \ref{sec:mixed_elliptic}, we define $\Rep_{\Qp}^{\rm f,S}(G_k,E)$ and related objects. In \S  \ref{sec:global_ranks}, we compute the dimension of $H^1_{f,S}(G_k;M_{a,b})$ in various cases. In \S  \ref{sec:local_ranks}, we compute the ranks of the corresponding local Selmer groups $H^1_{f}(G_{\pf};M_{a,b})$, then summarize the differences between local and global ranks.

In \S  \ref{sec:motivic_decomp}, we explain a general procedure for determining the semisimplification of $U_n$ when it is mixed elliptic. We also carry this out up to $n=5$ for $X=E'$ and up to $n=3$ for projective curves of genera $2$ and $4$ whose motive is mixed elliptic. We mention some explicit genus $2$ curves to which our method should apply.

In \S \ref{sec:tannakian_selmer}, we define a Tannakian fundamental group $\pi_1^{\ME}(\Oc_{k,S}, E)$ and write Selmer varieties in terms of it. We also explain the relevance of an analogue of \cite[Proposition 5.2]{MTMUE}, proven in the appendix (Theorem \ref{thm:cohom}).

In \S \ref{sec:selmer_localization}, we explain how to understand the localization map $\loc_{\Pi}$ more explicitly. For this, we describe a $p$-adic period map and a universal cocycle evaluation map.

In \S \ref{sec:level_3}, we specialize to $\Pi=U_3(E')$. We describe a basis for $\Oc(\Pi)$ in \S \ref{sec:fund_grp_coords}. In \S \ref{sec:local_selmer_coords}, we explicitly describe the local Selmer variety and unipotent Kummer map for $X=E'$ and $n=3$. In particular, we review results of \cite{Beacom}. In \S \ref{sec:U3_semisimple}, we explain why $\Pi$ is semisimple as a motive, which simplifies many calculations.

In \S \ref{sec:geometric_step}, we carry out some fundamental computations in shuffle algebras in order to complete the proof of Theorem \ref{thm:CK_ideal_element_form1}.

In \S  \ref{sec:explicit_examples}, we outline some examples we plan to compute explicitly in future work.

\subsection{Future Work}\label{sec:future_work}
Our eventual goal is to extend these methods to all higher genus curves, with a view toward effective Faltings. Nonetheless, the case of a punctured elliptic curve presents many of the features (and complications) of the general case. Some of these include:
\begin{enumerate}
    \item The need to use Conjecture \ref{conj:BK}
    \item The nontriviality of $F^0 U_n$
    \item A higher-dimensional reductive group in place of $\Gm$
    \item The lack of a $\Qb$-linear Tannakian category of motives
\end{enumerate}

While our goal in this paper is to focus on an example, we also set up much of the framework necessary to extend to the general case. In particular, see Remarks \ref{rem:general_J_1}, \ref{rem:general_J_2}, and \ref{rem:general_J_3}. Furthermore, in the mixed elliptic case, our computations could be used to compute Chabauty-Kim in all levels.

The biggest obstacle in applying this method to all hyperbolic curves $X$ is that there are not always going to be enough points. We give a brief overview of our plans to overcome this obstacle. See also Remark \ref{rem:explicit_periods}.

The issue in fact already appeared in \cite{PolGonI}, in which $\Xc(\Oc_{k,S})=\emptyset$. The key is to choose a finite extension $k'/k$ and/or a set $S'$ containing the places of $k'$ above those in $S$ for which $\Xc(\Oc_{k',S'})$ is large enough. We then use elements of $\Xc(\Oc_{k',S'})$ to explicitly understand elements of the coordinate ring
\[\Oc(\pi_1(\Rep_{\Qp}^{\rm f, S'}(G_{k'},E))).\]
One then uses an understanding of the inclusion
\[
\Rep_{\Qp}^{\rm f, S}(G_k,E) \subseteq \Rep_{\Qp}^{\rm f, S'}(G_{k'},E)
\]
to work things out over $\Oc_{k,S}$. If $X(k)$ is infinite, we often keep $k=k'$, while if $X(k)$ is finite, expect to pass to a finite extension field of $k$, possibly without changing $S$.

We mention three approaches to using elements of $\Xc(\Oc_{k',S'})$ to determine $\Xc(\Oc_{k,S})$:

\begin{itemize}
    \item Adapt the virtual cocycles of \cite[\S 10]{BrownIntegral} to a more general context
    \item Develop a theory of elliptic motivic polylogarithms and use it as in \cite{PolGonI}
    \item Use the ``Special Elements'' and coproduct formulas of \cite{Patashnick13}
\end{itemize}

One major obstacle in defining elliptic motivic polylogarithms is the lack of a canonical de Rham path between any two points in an elliptic curve (and more generally, any higher genus curve). This amounts to the lack of a global fiber functor as in the case of $\Poneminusthreepoints$. Some ideas in this direction are proposed in \cite[\S 11]{BrownIntegral}. In the unpublished note \cite{BrownKimHodge}, I correct some misstatements in loc.cit. and other articles on the subject and analyze different possible approaches to defining elliptic motivic polylogarithms.

In one ongoing project, Ishai Dan-Cohen, Stefan Wewers, and I have shown how to define a global fiber functor for vector bundles with unipotent connection on a projective curve $X$ depending on a decomposition \[X = U_1 \cup U_2\] into affine Zariski opens and a subspace \[T \subseteq H^0(U_1 \cap U_2; \Oc_X)\] mapping isomorphically onto $H^1(X;\Oc_X)$. Using Deligne's canonical extension, this provides a fiber functor for the category of vector bundles with unipotent connection on any open subcurve of $X$. For an elliptic curve $E$ given by Weierstrass model, one may take $T$ to be the subspace spanned by $y/x$.

In another approach, one may define elliptic motivic polylogarithms using the proposition in \S \ref{sec:cohom_cocycles}. A choice of element $\omega \in \Oc(\Pi)$ along with $z \in X(k)$ (and a basepoint $b \in X(k)$ lurking in the background) would thus define a motivic iterated integral
\[
\int_b^z \omega \in \Oc(U_E)
\]
as the pullback of $\omega$ along the cocycle corresponding to $\kappa(z)$.

Either the virtual cocycles approach or the more direct approach of \cite{PolGonI} will require a version of Goncharov's coproduct formula (\cite[Theorem 1.2]{GonGal}) for elements of $\Oc(U_E)$. We expect this to be straightforward, following from a general formula for coproducts of Tannakian matrix coefficients.

A coproduct formula and a conjectural basis of $\Oc(U_E)$ have already been written down in \cite{Patashnick13} in terms of the bar construction on cycle complexes. Ongoing work of the author of this paper and Owen Patashnick seeks to use the coproduct formoula of loc.cit.

\subsection{Notation}

When we use the term ``motive'', we are thinking of a system of realizations (as in \cite[Definition 5.5]{BlochKato}), but working primarily with the $p$-adic Galois representation realization. Thus when we say ``pure motive,'' we mean ``semisimple $p$-adic Galois representation.'' We use ``$\{p\}$'' to denote the set of places $k$ above a place $p$ of $\Qb$; to justify this notation when $p$ is not inert in $k$, one may think of it as the subscheme of $\Oc_k$ defined by $\{p=0\}$. We use $H^1_f$ and $H^1_g$ as in \cite{BlochKato} and \cite{FPR91}, recalled also in \S \ref{sec:BK_selmer}. We write $H^1_{f,S}$ for the subset of $H^1_g$ unramified/crystalline at all $v \notin S$.

For a $p$-adic representation $V$ of $G_k$, we write
\[
h^i(G_k;V) \colonequals \dim H^i(G_k;V)
\]
and
\[
h^i_{\bullet}(G_k;V) \colonequals \dim H^i_\bullet(G_k;V)
\]
for $k$ a local or global field and $\bullet \in \{f,g\}$ or $k$ a global field and $\bullet = f,S$.

We always work over a number field $k$ and with a chosen rational prime $p$. We write $G_k$ for the absolute Galois group of $k$. If $v$ is a place of $k$, then $k_v$ denotes the completion of $k$ at $v$, $\Oc_{k_v}$ its integer ring, $G_v$ its absolute Galois group, and $I_v$ the inertia subgroup. We write $\Oc_{k,S}$ for the subset of $k$ that is integral at all $v \notin S$. For a set $T$ of places, we write $G_{k,T}$ for the Galois group of the maximal extension of $k$ unramified outside $T \cup \{\infty\}$.

If $Y$ is a variety over $k$ and $i$ a non-negative integer, we let $h^i(Y)$ denote the continuous $p$-adic representation of $G_k$ given by $H^i_{\mathrm{\acute{e}}\mathrm{t}}(Y_{\overline{k}};\Qp)$, and similarly for $h_i(Y)$.

For a smooth geometrically irreducible variety $Y$ over $k$, we let
\[
U(Y) \coloneqq \pi_1^{\et,\un}(Y_{\overline{k}})_{\Qp},
\]
the $\Qp$ pro-unipotent completion of the geometric \'etale fundamental group of $Y$. It has a continuous algebraic action of $G_k$, and $U(Y)^{ab} \cong h_1(Y)$. We always take it relative to basepoint $b \in X(k)$, but we suppress this basepoint in the notation.

We use ``$*$'' to denote the point, in the context of objects of a pointed category.

\subsection{Acknowledgements}

Some of these ideas began while the author was graduate student with Bjorn Poonen and later as a postdoc with Martin Olsson, both of whom helped with some of these ideas. The author was supported by NSF grants DMS-1069236 and DMS-160194, Simons Foundation grant \#402472, and NSF RTG Grant \#1646385 during various parts of the development of these ideas. This work is connected to a larger collaboration with Ishai Dan-Cohen and Stefan Wewers, and the author would like to thank both of them for beneficial conversations. He would also in particular like to thank Ishai Dan-Cohen for pointing out errors in an earlier draft. The author would also like to thank the Stockholm conference on elliptic motives in May 2019, as well as Francis Brown and Nils Matthes for discussions during that conference. He thanks Alex Betts for many helpful conversations and email replies about the results of \cite{BettsDogra20}. He thanks David Loeffler for answering questions about Iwasawa theory and Rob Pollack for explaining how to compute $p$-adic L-functions of elliptic curves.

Much of this paper was writen while the author was supported by the National Science Foundation under Grant No. DMS-1928930 while the author participated in a program hosted by the Mathematical Sciences Research Institute in Berkeley, California, during the Fall 2020 semester. The author would like to thank Barry Mazur for many helpful questions and comments on this article during that program.

%Had questions answered by Beilinson, Goncharov, Hain? But maybe not relevant to this paper.

\section{Bloch-Kato Groups and Conjectures}\label{sec:BK}

In this section, we go over background on Galois representations and the Bloch-Kato Selmer groups and conjectures.

\subsection{Categories of Galois Representations}\label{sec:galois_rep_cats}

Let $k$ be a number field and $p$ a prime number, and let 
\[
\Rep_{\Qp}^{\rm g}(G_k)
\]
denote the category of $p$-adic representations of $G_k$ that are unramified almost everywhere and de Rham at every place of $k$ above $v$. We let
\[
\Rep_{\Qp}^{\rm sg}(G_k)
\]
denote the subcategory of representations $V$ that are \emph{strongly geometric}, meaning $V$ has a finite increasing filtration $W^{\bullet} V$, known as the \emph{motivic weight filtration}, for which
\[
\operatorname{Gr}^W_n V
\]
is semisimple and pure of weight $n$ at almost all unramified places in the sense of \cite{WeilII}. For an integer $n$, denote by $\Rep_{\Qp}^{\rm sg}(G_k)_{w\le n}$ (resp. $\Rep_{\Qp}^{\rm sg}(G_k)_{w \ge n}$, $\Rep_{\Qp}^{\rm sg}(G_k)_{w=n}$) the subcategories of objects whose nonzero weight-graded pieces all have weights less than or equal to $n$ (resp. greater than or equal to $n$, equal to $n$).

The conjecture of Fontaine-Mazur states that every irreducible object of $\Rep_{\Qp}^{\rm g}(G_k)$ is a subquotient of an \'etale cohomology group of some variety over $k$. A mixed version of Fontaine-Mazur\footnote{According to B. Mazur, while this was not part of their original conjecture, it may have been `in the air' at the time. Note that \cite[Observation 2]{kim09} refers to this mixed version rather than the original version.} states that every object of $\Rep_{\Qp}^{\rm g}(G_k)$ is such a subquotient. By resolution of singularities and the Weil conjectures, any such subquotient has a motivic weight filtration for which the graded pieces are pure of the appropriate weight. A corollary of this mixed Fontaine-Mazur along with the Grothendieck-Serre semisimplicity conjecture\footnote{This is the conjecture that the \'etale cohomology of a smooth projective variety over a finitely generated field of characteristic $0$ is semisimple as a Galois representation, and it follows from the Tate conjecture by \cite{Moonen19}.} is thus:

\begin{conj}\label{conj:g_sg}
The categories $\Rep_{\Qp}^{\rm sg}(G_k)$ and $\Rep_{\Qp}^{\rm g}(G_k)$ are equal.\footnote{We really mean equal, not equivalent or isomorphic, as the former is defined as an explicit subcategory of the latter.}
\end{conj}

\subsection{Grothendieck Groups}\label{sec:groth_grp}

In \S \ref{sec:motivic_decomp}, we will use the Grothendieck ring
\[
K_0(\Rep_{\Qp}^{\rm sg}(G_k)),
\]
which has a grading with the $n$th graded piece given by the group\footnote{Not a ring for $n \neq 0$!} $K_0(\Rep_{\Qp}^{\rm sg}(G_k)_{w=n})$. The subring $K_0(\Rep_{\Qp}^{\rm sg}(G_k)_{w \le 0})$ is negatively graded, and we denote by
\[
\widehat{K}_0(\Rep_{\Qp}^{\rm sg}(G_k)_{w \le 0})
\]
its completion with respect to the ideal $K_0(\Rep_{\Qp}^{\rm sg}(G_k)_{w \le -1})$. If $V$ is an object of $\Rep_{\Qp}^{\rm sg}(G_k)_{w \le -1}$, then $[\Sym{V}] = \sum_{n=0}^{\infty} [\Sym^n{V}]$ and $[TV] = \sum_{n=0}^{\infty} [V^{\otimes n}]$ are in $\widehat{K}_0(\Rep_{\Qp}^{\rm sg}(G_k)_{w \le 0})$. We let
\[
\pr_n
\]
denote projection onto the $n$th component. We thus have $\pr_n([V]) = [\Gr^W_n{V}]$.

\subsection{Bloch-Kato Selmer Groups}\label{sec:BK_selmer}

Let $v$ be a place of $k$, $G_v$ a decomposition group of $v$ in $G_k$, and $I_v$ the inertia subgroup. We recall the local and global Selmer groups of \cite{BlochKato}.

For $v \nmid p$ and a $p$-adic representation $V$ of $G_v$, we set
\[
H^1_g(G_v;V) \colonequals H^1(G_v;V)
\]
\[
H^1_f(G_v;V) \colonequals H^1_{ur}(G_k;V) \colonequals \ker(H^1(G_v;V) \to H^1(I_v;V)).
\]

For $v \mid p$, let $B_{cris}$ and $B_{dR}$ denote the crystalline and de Rham period rings associated to $k_v$, respectively. We set
\[
H^1_g(G_v;V) \colonequals \ker(H^1(G_v;V) \to H^1(G_v;V \otimes_{\Qp} B_{dR}))
\]
\[
H^1_f(G_v;V)  \colonequals \ker(H^1(G_v;V) \to H^1(G_v;V \otimes_{\Qp} B_{cris})).
\]

Let $V$ be a $p$-adic representation of $G_k$. For a place $v$ of $k$, we let $\res_v$ denote the restriction map $H^1(G_k;V) \to H^1(G_v;V)$. We then set
\[
H^1_g(G_k;V) \colonequals \{\alpha \in H^1(G_k;V) \, \mid \, \res_v(\alpha) \in H^1_g(G_v;V) \, \forall \, v\}.
\]
\[
H^1_f(G_k;V) \colonequals \{\alpha \in H^1(G_k;V) \, \mid \, \res_v(\alpha) \in H^1_f(G_v;V) \, \forall \, v\}.
\]

As in \cite[II.1.3.1]{FPR91}, for a set $S$ of places of $k$, we set
\[
H^1_{f,S}(G_k;V) \colonequals \{\alpha \in H^1_g(G_k;V) \, \mid \, \res_v(\alpha) \in H^1_f(G_v;V) \, \forall \, v \notin S\}.
\]

Conjecture \ref{conj:g_sg} implies the following conjecture of Bloch-Kato:

\begin{conj}[Bloch-Kato]\label{conj:BK}
For an irreducible geometric $p$-adic representation $V$ of $G_k$ of non-negative weight, the group
\[
H^1_g(G_k;V)
\]
vanishes.
\end{conj}

\begin{rem}
There are three philosophical reasons behind this conjecture:
\begin{enumerate}
    \item The mixed version of Fontaine-Mazur predicts that any such extension class comes from geometry and therefore has a weight filtration splitting it.
    \item The dimension of the group should be the order of vanishing of an $L$-function in the region of absolute (and hence nonzero) convergence.
    \item The group corresponds to an algebraic $K$-theory group in negative degree.
\end{enumerate}

%One philosophy behind this conjecture is as follows. There should be a mixed version of the Fontaine-Mazur conjecture in which all $p$-adic Galois representations (not necessarily) that are de Rham at all places dividing $p$ and almost everywhere unramified are expected to come from geometry. Thus an element $\alpha \in H^1_f(G_k;V)$ corresponds to an extension

%By Deligne's Weil II, any representation coming from geometry has a weight filtration, which implies that all cohomology classes satisfying the Bloch-Kato Selmer g
\end{rem}

We also need the following consequence of Poitou-Tate duality (c.f. \cite[Theorem 2.2]{BellaicheNotes} or \cite[II.2.2.2]{FPR91}):

\begin{fact}
For a geometric Galois representation $V$, we have
\begin{equation}\label{eqn:Poitou-Tate}
h^1_f(G_k;V) = h^0_f(G_k;V) + h^1_f(G_k;V^{\vee}(1)) - h^0_f(G_k;V^{\vee}(1))  + \sum_{v \mid p} \dim_{\Qp} (\operatorname{D}_{\dR}{V}/\operatorname{D}^+_{\dR}{V}) - \sum_{v \mid \infty} h^0(G_v;V)
.\end{equation}
\end{fact}

In terms of the category $\Rep_{\Qp}^{\rm g}(G_k)$, we have
\[
H^1_g(G_k;V) = \Ext^1_{\Rep_{\Qp}^{\rm g}(G_k)}(\Qp(0),V).
\]

\subsection{Verifying Cases of the Conjectures}\label{sec:verifying_conjectures}

This section is a brief interlude explaining that many cases of interest of Conjecture \ref{conj:BK} may be verified in specific cases. These verifications require that $E$ be modular, which is known in general when $k$ is totally real.

For a punctured elliptic curve in level $3$ (the topic of \S \ref{sec:geometric_step}), the only conjecture we need is that
\[
H^1_f(G_k;h^1(E)) = 0.
\]
Kato's Euler system argument (\cite[Theorem 17.4]{Kato04}) combined with a control theorem (\cite[\S 8.10]{SelmerComplexes}) may be used in the ordinary case to reduce this statement to non-vanishing of $L_p(E,2)$. One may then approximate this $L$-value using Mazur-Stickelberger elements defined in terms of Manin symbols (\cite[\S 3.5]{WuthrichOverview14}). We have verified that $L_p(E,2) \neq 0$ when $E=102a1$ and $p=5$ and when $E=128a2$ and $p=3,5$ using the following code in SageMath:

\begin{verbatim}
def mu(a,p,n,phi):
    return phi(a*p^(-n))*alpha^(-n) - (alpha^(-n-1))*phi(a*p^(-n+1))

def mazur_stickelberger_sum(E,p,prec):
    if E.is_supersingular(p):
        error('curve is supersingular')
    phi = E.modular_symbol(-1)
    K = Qp(p,prec); Fp = GF(p)
    Ep = E.base_extend(Fp); f = Ep.frobenius_polynomial()
    R.<x> = PolynomialRing(K); fp = R(f)
    rs = fp.roots(); alpha = rs[0][0]
    if abs(alpha) < 1:
        alpha = rs[1][0]
    sum = K(0)
    for a in range(p^prec):
        if not p.divides(a):
            sum += K(a*mu*(a,p,prec,phi))
    return sum
\end{verbatim}

Computations for higher genus curves, including those outlined in \S \ref{sec:decom_genus2}, may require vanishing of
\[
H^1_f(G_k;M_{2,-1})
\]
or
\[
H^1_f(G_k;M_{3,-2}).
\]

The former follows in most cases by \cite{Allen16}, which uses modularity lifting results. The latter may be approached by adapting arguments of Loeffler--Zerbes from \cite{LoefflerZerbesSymmetricCube} that in turn use the general results of \cite{LoefflerZerbesGSp4}.

\subsection{\texorpdfstring{$S$}{S}-integral Versions}\label{sec:S-integral}

Let $\Rep_{\Qp}^{\rm f, S}(G_k)$ denote the full subcategory of $\Rep_{\Qp}^{\rm g}(G_k)$ of representations that are unramified outside $S \cup \{p\}$ and crystalline at $v$ if $v \notin S$ has residue characteristic $p$. If $V \in \Rep_{\Qp}^{\rm f, S}(G_k)$, then
\[
H^1_{f,S}(G_k;V) = \Ext^1_{\Rep_{\Qp}^{\rm f, S}(G_k)}(\Qp(0),V).
\]
More generally, we might want to consider $H^1_{f,S}$ even when $V$ is not in  $\Rep_{\Qp}^{\rm f, S}(G_k)$. For this we make a modification in Remark \ref{rem:elliptic_ext_S}.

Let us briefly analyze the extent to which $H^1_{f,S}(G_k;V)$ depends on $S$.

\begin{prop}\label{prop:change_in_S}

Let $S' = S \cup \{v\}$, with $V$ a $p$-adic representation of $V$, and suppose $v \nmid p$. Then
\[
h^1_{f,S'}(G_k;V) = h^1_{f,S}(G_k;V) + h^0(G_v;V^{\vee}(1))\]

In particular, $\dim H^1_{f,S}(G_k;V) = \dim H^1_{f,S'}(G_k;V)$ if $\restr{V}{G_v}$ has no quotient isomorphic to $\Qp(1)$.
\end{prop}

\begin{proof}
By \cite[Proposition III.3.3.1(b)]{FPR91}, we have an exact sequence
\[
0 \to H^1_{f,S}(G_k;V) \to \dim H^1_{f,S'}(G_k;V) \to H^1_{g/f}(G_v;V) \to 0,\]
where $H^1_{g/f}(G_v;V) \colonequals H^1_{g}(G_v;V)/H^1_{f}(G_v;V)$.

By Tate-Poitou duality and because $v \notin \{p\}$, we have an isomorphism
\[
H^1_{g/f}(G_v;V) = H^1(G_v;V)/H^1_f(G_v;V) \cong H^1_f(G_v;V^{\vee}(1))^{\vee}.
\]
The inequality then follows by $\dim H^1_f(G_v;V^{\vee}(1)) = \dim H^0(G_v;V^{\vee}(1))$.
\end{proof}

%If $V \in \Rep_{\Qp}^{\rm sg}(G_k)$ is unramified at $v$, then we expect $h^0(G_v;V^{\vee}(1)) = 0$ as long $V^{\vee}(1)$ has trivial weight $0$ graded piece, or equivalently, as long as $V$ has trivial weight $-2$ graded piece. This is expected moreover if $V$ is potentially unramified at $v$. If not, it is possible that $h^0(G_v;V^{\vee}(1)) \neq 0$ even if $V$ does not (globally) have a nontrivial weight $-2$ graded piece, as we now discuss.

\subsection{Local Conditions and Weight-Monodromy}\label{sec:WM}

Supposing still that $v \nmid p$, we recall the theory of weights for possibly-ramified $p$-adic representations of $G_v$. By Grothendieck's Monodromy Theorem (\cite[Appendix]{SerreTate68}), after restricting to an open subgroup $G_v'$ of $G_v$, the action of inertia is unipotent, so the irreducible pieces in the Jordan-H\"older series of $\restr{V}{I_{v}'}$ are unramified. We may therefore talk about the Frobenius weight of such a piece, and the \emph{(Frobenius) weights}\footnote{We write ``Frobenius weight'' in full when we must distinguish from motivic weight.} of $\restr{V}{G_v}$ is the set of weights that appear in these subquotients.

Let $V \in \Rep_{\Qp}^{\rm sg}(G_k)$, so that there is a filtration $W_{\bullet}V$ by motivic weight. By our definition,
\[
\operatorname{Gr}^W_n V
\]
is then unramified and pure of weight $n$ at almost all places $v$. Let us recall what happens if we ask for a description at \emph{all} $v$. After restricting to an open subgroup of $I_v$ that acts unipotently, the inertia action may be described by a nilpotent operator $N \colon V \to V$. By Jacobson-Morosov (\cite[Proposition 1.6.1]{WeilII}), there is a unique filtration $\Fil^N_{\bullet} V$ for which $N(\Fil_i^N V) \subseteq \Fil^N_{i-2} V$ and $N^i \colon \Gr^N_i V \to \Gr^N_{-i} V$ is an isomorphism for all $i \ge 0$ (c.f. \cite[Conjecture 1.13]{perfectoid}). If we assume that $V$ comes from geometry (as is implied by the Fontaine-Mazur conjecture), then the weight-monodromy conjecture of Deligne (\cite{DeligneHodgeI}) states that:

\begin{conj}
$\Gr^N_j \Gr^W_i  V$ is pure of weight $i+j$ as a representation of $V$.
\end{conj}

This conjecture is the Weight-Monodromy Theorem of Grothendieck (\cite[Appendix]{SerreTate68}) for a Galois representation coming from an abelian variety. As a corollary, it holds for any representation coming from first homology, or even more generally, from the $p$-adic unipotent fundamental group, of any variety. In particular, it holds for all the representations we care about.

If $V$ is potentially unramified at $v$, then $N=0$, so that the Frobenius weights of $\restr{V}{G_v}$ are the motivic weights of $V$. For example, the set of weights of $\Qp(n)$ is $\{-2n\}$.

\begin{prop}\label{prop:WM_h0}
If $0$ is not one of the Frobenius weights of $\restr{V}{G_v}$, then $h^0(G_v;V) = 0$. More generally, if $\restr{V}{G_v}$ satisfies the weight-monodromy conjecture at $V$, and \[
\Gr^N_j \Gr^W_i  V
\]
is trivial whenever $i+j=0$ and $i \ge 0$.
\end{prop}

\begin{rem}
The condition is satisfied if $\Gr_W^i V = 0$ for $i \ge 0$, i.e., if $V \in \Rep_{\Qp}^{\rm sg}(G_k)_{w\le -1}$. This corresponds to the condition ``$(\mathrm{WM}_{<0})$'' of \cite[\S 2.1]{BettsDogra20}.
\end{rem}

\begin{proof}
It suffices to prove this for $V$ of pure motivic weight. By the definition of $\Fil_i^N V$, the kernel of $N$ is contained in $\Fil_1^N V$. Therefore, if $H^0(G_v;V)$ is nontrivial, it projects nontrivially onto $\Gr^N_j \Gr_W^i  V$ for some $j \le 0$. This piece then has a nontrivial Galois-fixed part, so we have $i+j=0$. Then $i \ge 0$, hence $\Gr^N_j \Gr^W_i  V$ is trivial by hypothesis, a contradiction.
\end{proof}

We now consider what happens when we dualize and twist. The monodromy operator $N$ of $V^{\vee}$ is the negative dual of the monodromy operator of $V$; in particular, $\Fil^N_{\bullet} V^{\vee}$ is dual to $\Fil^N_{\bullet} V$. Finally, $W_{\bullet} V^{\vee}$ is dual to $W_{\bullet} V$.  In particular, we find that
\[
\Gr^N_j \Gr^W_i  V
\]
is dual to
\[
\Gr^N_{-j} \Gr^W_{-i}  V^{\vee}
\]
as a $G_v/I_v$-representation. Since duality negates the Frobenius weights, so that the weight-monodromy conjecture holds for $V$ iff it does for $V^{\vee}$.

A Tate twist shifts the motivic weight filtration down by $2$ and does not affect $\Fil^N_{\bullet}$. It also shifts the Frobenius weights down by $2$ and thus preserves the truth of the weight-monodromy conjecture. We have 
\[
\dim \Gr^N_j \Gr^W_i  V = \dim \Gr^N_{-j} \Gr^W_{-i-2}  V^{\vee}(1).
\]

As a corollary of Proposition \ref{prop:change_in_S} and Proposition \ref{prop:WM_h0}, we get

\begin{cor}\label{cor:change_in_S_WM}
In the notation of Proposition \ref{prop:change_in_S}, if $-2$ is not one of the Frobenius weights of $\restr{V}{G_v}$, then
\[
\dim H^1_{f,S}(G_k;V) = \dim H^1_{f,S'}(G_k;V).
\]
This is true more generally if $\restr{V}{G_v}$ satisfies the weight-monodromy conjecture at $V$, and \[
\Gr^N_j \Gr^W_i  V
\]
is trivial whenever $i+j=-2$ and $i \le -2$.
\end{cor}

\begin{rem}\label{rem:WM_abel_var}
The Weight-Monodromy Theorem for abelian varieties states more precisely that the weights of the Tate module are contained in $\{0,-1,-2\}$. The same is therefore true for the first homology of any smooth variety.
\end{rem}

%**Can I do something more general? Like that everything of weight less than or equal to three works?

\section{Selmer Varieties}\label{sec:selmer_varieties}

We use this section to precisely specify our notation and definitions and to review the differing notions of Selmer variety or scheme in the literature. Almost none of the material here is new.

Let $k$ be a number field, $X/k$ a curve for which $h_1(X)$ lies in $\Rep_{\Qp}^{\rm ss}(G_k,E)$, $\Pi$ a finite-dimensional Galois-equivariant quotient of $U=U(X)$ based at $b \in X(k)$, and $v$ a place of $k$. Then $\Lie{\Pi}$ is a (Lie algebra) object in
\[
\Rep_{\Qp}^{\rm f, S}(G_k,E),
\]
while its universal enveloping algebra $\Uc \Pi$ and coordinate ring $\Oc(\Pi)$ are Pro- and Ind-objects, respectively, of that category.

There are local and global unipotent Kummer maps\footnote{Also called \emph{unipotent Albanese maps} in some parts of the literature, such as \cite{kim09} and \cite{KimTamagawa}.} fitting into a diagram
\[
\xymatrix{
X(k) \ar[r] \ar[d]^-{\kappa} & X(k_v) \ar[d]^-{\kappa_{v}}\\
H^1(G_k;\Pi) \ar[r]_-{\mathrm{loc}_v} & H^1(G_v;\Pi).
}
\]

\subsection{Selmer Schemes as Schemes}

Let $G=G_v$ or $G=G_{k,T}$ for a finite set $T$ of places of $k$, and let $\Pi$ be a unipotent group over $\Qp$ with continuous action of $G$.

Kim proves in \cite{kim09} that
\[
H^1(G;\Pi)
\]
is naturally the set of $\Qp$-points of a scheme over $\Qp$. It is characterized by setting it to be the affine space underlying the $\Qp$-vector space $H^1(G;\Pi)$  when $\Pi$ is abelian and requiring that
\[
* \to H^1(G;\Pi') \to H^1(G;\Pi) \to H^1(G;\Pi'') \to *
\]
is a short exact sequence of pointed schemes when $1 \to \Pi' \to \Pi \to \Pi'' \to 0$ is a short exact sequence of unipotent groups with $G$-action.

\begin{fact}\label{fact:local_are_algebraic}
It is furthermore shown in \cite[\S 2]{kim09} that
\begin{itemize}
    \item The map $\loc_v \colon H^1(G_{k,T};\Pi) \to H^1(G_v;\Pi)$ is a map of algebraic varieties, and
    \item The subspace $H^1_f(G_v;\Pi)$ (Definition \ref{defn:non_ab_local_conditions}) is an algebraic subvariety of $H^1(G_v;\Pi)$
\end{itemize}

\end{fact}

We will not be particularly careful about the difference between a Selmer scheme and its set of $\Qp$-points, and the careful reader may check that it does not pose any problems. By Fact \ref{fact:local_are_algebraic}, all the subsets we define below correspond to algebraic subvarieties.

\subsection{Good Reduction and Local Conditions}

%For the rest of this article, let $S$ denote a set of places of $k$ disjoint from $\{p\}$

\begin{defn} We define good reduction as follows:

\begin{itemize}
    \item For a place $v$ of $k$, we say that $X$ \emph{has good reduction at $v$} if there is a model $\Xc$ of $X$ over $\Oc_v$ sitting inside a smooth proper curve $\overline{\Xc}$ over $\Oc_{v}$ with \'etale boundary divisor.\footnote{As an example of why the condition on the divisor is necessary, $\Pb^1\setminus\{0,1,2,\infty\}$ has bad reduction at $2$ even though it has a smooth model over $\mathbb{Z}_2$.} We say more precisely that the \emph{model} $\Xc$ has good reduction at $v$.
    
    \item We say that $X$ has \emph{potentially good reduction} at $v$ if there is a finite extension $l_v/k_v$ for which $X_{l_v}$ has good reduction at $v$. We say more precisely that the model $\Xc$ has potentially good reduction at $v$ if $\Xc_{\Oc_{l_v}}$ is dominated\footnote{In this context, ``dominate'' means that the good model is the complement in a blowup of an integral compactification of the strict transform of the boundary. In particular, the good model does not necessarily map to $\Xc_{\Oc_{l_v}}$ as a scheme over $\Oc_{l_v}$.} by a good model of $X_{l_v}$.
\end{itemize}
\end{defn}

\begin{defn}\label{defn:non_ab_local_conditions}
We set
\[
H^1_f(G_v;\Pi) \colonequals \Ker(H^1(G_v;\Pi) \to H^1(I_v;\Pi)
\]
for $v \notin \{p\}$ and
\[
H^1_f(G_v;\Pi) \colonequals \Ker(H^1(G_v;\Pi) \to H^1(G_v;\Pi \otimes_{\Qp} B_{cris})
\]
for $v \in \{p\}$.
\end{defn}

If $\Xc$ has good reduction at $v$, the action of $G_v$ on $\Pi$ is unramified.

For $\pf \in \{p\}$ at which $\Xc$ has good reduction, we set
\[
\operatorname{Sel}(\Xc/\Oc_{\pf})_{\Pi} \colonequals H^1_f(G_{\pf};\Pi).
\]

\begin{fact}\label{fact:local_conditions}
If $v \notin \{p\}$, then
\begin{enumerate}
    \item (\cite[Corollary 2.1.9]{BettsDogra20})
    \[
H^1_f(G_k;V) = *
\]
\item (\cite[Proposition 1.2]{BettsDogra20}) If $\Xc$ has potentially good reduction at $v$, then
    \[
    \kappa_v(\Xc(\Oc_v)) = H^1_f(G_k;V) = *
    \]
   % \item If $\Xc$ is a regular semistable model of $X$ at $v$, then
    %\[
    %\kappa_v(\Xc(\Oc_v)) \subseteq H^1(G_k;V)
    %\]
    %is a finite set corresponding to the set of components of the special fiber of $\Xc$.\footnote{Note that an integral point on a regular semistable model reduces to a unique component of the special fiber by \cite[Theorem 1.1]{LiuTong16}.}
\end{enumerate}

If $v \in \{p\}$,
\begin{enumerate}
\setcounter{enumi}{2}
    \item (\cite[Theorem 1]{kim09}) and $\Xc$ has good reduction at $v$, then
    \[
    \kappa_v(\Xc(\Oc_v))^{\mathrm Zar} = H^1_f(G_k;V)
    \]
\end{enumerate}
\end{fact}

\subsection{The Chabauty-Kim Diagram}

Let $S$ denote a finite set of places of $k$ disjoint from $\{p\}$ and $\pf \in \{p\}$ with $k_{\pf} \cong \Qp$. For a model $\Xc$ of $X$ over $\Oc_{k,S}$ with good reduction at $\pf$ and $b \in \Xc(\Oc_{k,S})$, we will describe a Selmer variety $\operatorname{Sel}(\Xc)_{\Pi}$ fitting into a diagram

\[
\xymatrix{
\Xc(\Oc_{k,S}) \ar[r] \ar[d]^{\kappa} & \Xc(\Oc_{\pf}) \ar[d]_-{\kappa_{\pf}} \\
\operatorname{Sel}(\Xc)_{\Pi} \ar[r]_-{\mathrm{loc}_{\Pi}} & \operatorname{Sel}(\Xc/\Oc_{\pf})_{\Pi}
},
\]
where $\loc_{\Pi}$ always refers to $\loc_{\pf}$ for the chosen place $\pf \in \{p\}$ with respect to the fundamental group quotient $\Pi$, and
\[
\operatorname{Sel}(\Xc/\Oc_{\pf})_{\Pi} \colonequals H^1_f(G_{\pf};\Pi).
\]

We set the \emph{Chabauty-Kim locus}:
\[
\Xc(\Oc_{\pf})_{\Pi} \colonequals \kappa_{\pf}^{-1}(\operatorname{Im}(\mathrm{loc}_{\Pi})) = \int^{-1}(\operatorname{Im}(\log_{\mathrm{BK}} \circ \mathrm{loc}_{\Pi})).
\]

The \emph{Chabauty-Kim ideal}
\[
\Ic_{CK,\Pi}(\Xc)
\]
of regular functions vanishing on the image of $\loc_{\Pi}$ pulls back to a set $\kappa_{\pf}^{\#}(\Ic_{CK,\Pi})$ of functions on $\Xc(\Oc_{\pf})$ vanishing on $\Xc(\Oc_{k,S})$ with $\Xc(\Oc_{\pf})_{\Pi}$ as its set of common zeroes.

For $\Pi = U_n$, we write $\operatorname{Sel}(\Xc)_{n}$, $\operatorname{Sel}(\Xc/\Oc_{\pf})_{n}$, $\Ic_{CK,n}(\Xc)$, and $\Xc(\Oc_{\pf})_{n}$.

As described in \cite[p.96]{kim09}, when
\begin{equation}\label{eqn:selmer_inequality2}
\dim \operatorname{Sel}(\Xc)_{\Pi}
<
\dim \operatorname{Sel}(\Xc/\Oc_{\pf})_{\Pi},
\end{equation}
we may conclude that $\loc_{\Pi}$ is non-dominant and therefore that
\[
\Ic_{CK,\Pi}(\Xc)
\]
is nonzero, hence by \cite[Theorem 1]{kim09}
\[
\Xc(\Oc_{\pf})_{\Pi}
\]
is finite. The statement of \cite[Theorem 2]{kim09} is that this happens for $k=\Qb$ and $n$ sufficiently large if Conjecture \ref{conj:BK} is true.

In \S \ref{sec:global_ranks}-\ref{sec:motivic_decomp}, we will show how to check (\ref{eqn:selmer_inequality2}) when $X$ is mixed elliptic (Definition \ref{defn:mixed_elliptic}).

\subsection{Global Selmer Varieties}

Let $T_0$ denote the set of places at which $X$ does not have potentially good reduction, and let $T' = S \cup \{p\} \cup T_0$.  Let $T_1$ denote the set of places at which $X$ has bad but potentially good reduction, and let $T = T' \cup T_1$. Then $X$ has good reduction outside $T$, so the action of $G_k$ on $\Pi$ factors through $G_{k,T}$, and our Selmer varieties will be subvarieties of
\[
H^1(G_{k,T};\Pi).
\]

We first discuss the case in which $T_0 \subseteq S$, which is much simpler and already applies to one of our examples, given by the elliptic curve ``128a2''. The general case is in \S \ref{sec:bad_outside_S}.

\subsubsection{Good Reduction Outside \texorpdfstring{$S$}{S}}\label{sec:good_outside_S}

\begin{defn}\label{defn:selmer_kim09}
We suppose $\Xc$ has potentially good reduction at all $v \in \Spec{\Oc_{k,S}}$. As in \cite{kim09}, we define
\[
\operatorname{Sel}(\Xc)_{\Pi} \colonequals H^1_{f,S}(G_k;\Pi) \colonequals \{ \alpha \in H^1(G_{k,T};\Pi) \, \mid \, \loc_v(\alpha) \in H^1_f(G_v;\Pi) \, \forall\,  v \notin S\}.
\]
\end{defn}

This works, for example, if $X$ is the elliptic curve ``128a2'', and $\Oc_{k,S} = \Zb[1/2]$.
%%% ** Should I change this? Depends on how I divide the different examples

If $T_0 \nsubseteq S$, we may expand $S$ to $S' = S \cup T_0$.\footnote{If one is worried about having a single model $\Xc$ over $\Oc_{k,S'}$ with potentially good reduction, or even good reduction, one may, by spreading out, expand $S'$ to ensure this is the case.} We may also modify $p$ and $\pf$ to ensure that $S' \cap \{p\} = \emptyset$. One might then hope to apply the Chabauty-Kim method to compute $\Xc(\Oc_{k,S'})$, then find the subset $\Xc(\Oc_{k,S}) \subseteq \Xc(\Oc_{k,S'})$ by hand. This is the approach of \cite{kim09},\footnote{In \cite{kim09}, it is further assumed that $\Xc$ has good reduction at all $v \in \Spec{\Oc_{k,S}}$, but this does not appear necessary.} which is enough to show that Conjecture \ref{conj:BK} implies finiteness of $\Xc(\Oc_{k,S})$ when $k=\Qb$.

\subsubsection{Bad Reduction Outside \texorpdfstring{$S$}{S}}\label{sec:bad_outside_S}

Nonetheless, it is often more practical to work with a Selmer scheme $\operatorname{Sel}(\Xc)_{\Pi}$ defined even when $\Xc/\Oc_{k,S}$ has permanent bad reduction at some $v \in \Spec{\Oc_{k,S}}$. This is because we often have
\[
\dim \operatorname{Sel}(\Xc)_{\Pi} < \dim \operatorname{Sel}(\Xc_{\Oc_{k,S'}})_{\Pi},
\]
which means in practice that one may need to pass to a larger $\Pi$ to get finiteness for the Chabauty-Kim locus from $\operatorname{Sel}(\Xc_{\Oc_{k,S'}})_{\Pi}$ than from $\operatorname{Sel}(\Xc)_{\Pi}$. This is especially true when $X$ is proper, for which one hopes to set $S = \emptyset$.

We now recall the definition of Selmer variety from \cite{nabsd}:

\begin{defn}\label{defn:selmer_nabsd}
Given a model $\Xc$ of $X$ over $\Oc_{k,S}$, we define $\Sel_{S,\Pi}(\Xc)$ as the subset of $\alpha \in H^1(G_k;\Pi)$ for which
\begin{itemize}
    \item $\alpha \in \kappa_v(\Xc(\Oc_v))^{\rm Zar}$ for all $v \notin S$,
\end{itemize}
where superscript $\mathrm{Zar}$ denotes Zariski closure.
\end{defn}

Then, almost by definition, we have a map
\[
\kappa \colon \Xc(\Oc_{k,S}) \to \Sel_{S,\Pi}(\Xc).
\]

It follows from Fact \ref{fact:local_conditions} that Definition \ref{defn:selmer_nabsd} agrees with Definition \ref{defn:selmer_kim09} when $\Xc$ has potentially good reduction at all $v \in \Spec{\Oc_{k,S}}$ and good reduction at all $v \in \{p\}$.

We further assume that $\Xc$ has good reduction at all places in $\{p\}$, not just at $\pf$.

In general, $\Sel_{S,\Pi}(\Xc)$ may differ from $H^1_{f,S}(G_k;\Pi)$ when $T_0 \setminus S \neq \emptyset$. This is because of the fact that
\[
\kappa_v(\Xc(\Oc_{v}))
\]
may be nontrivial for $v \in T_0 \setminus S$.

Nonetheless, when $v \notin \{p\}$, the image is finite by \cite[Corollary 0.2]{KimTamagawa}. We therefore assume that $\Xc$ has good reduction at all $v \in \{p\}$ (previously we assumed this only for $v=\pf$), i.e. that $(T_0 \cup T_1) \cap \{p\} = \emptyset$. This may be arranged by an appropriate choice of $p$.

We now explain, under the assumptions above (that $\{p\}$ is disjoint from $S \cup T_0 \cup T_1$), why the Selmer variety
\[
\Sel_{S,\Pi}(\Xc)
\]
is a disjoint union of copies of
\[
H^1_{f,S}(G_k;\Pi)
\]
indexed by the (finite) image of the map
\[
\Sel_{S,\Pi}(\Xc) \to \prod_{v \in T_0 \setminus S} \kappa_v(\Xc(\Oc_{v})) = \prod_{v \in T_0 \setminus S} \kappa_v(\Xc(\Oc_{v}))^{\rm Zar}.
\]

As in \cite{BalakrishnanDograI} and \cite{BalakrishnanDograII}, let $\alpha_1,\cdots,\alpha_N$ denote a set of representatives in $\Sel_{S,\Pi}(\Xc)$ for this image. For $i=1,\cdots,N$, let
\[
\Sel_{S,\Pi}(\Xc)_{\alpha_i} = \{\alpha \in \Sel_{S,\Pi}(\Xc) \, \mid \, \loc_{v}(\alpha) = \loc_v(\alpha_i) \, \forall \, v \in T_0 \setminus S\},
\]
so that
\[
\Sel_{S,\Pi}(\Xc) = \bigsqcup_{i=1}^N \Sel_{S,\Pi}(\Xc)_{\alpha_i}.
\]

Then by \cite[Lemma 2.6]{BalakrishnanDograI} (c.f. also \cite[Lemma 2.1]{BalakrishnanDograII} and \cite[Lemma 3.1]{dogra2020unlikely}), we have a natural isomorphism
\begin{equation}\label{eqn:selmer_variety}
\Sel_{S,\Pi}(\Xc)_{\alpha_i} \cong H^1_{f,S}(G_k;\Pi^{\alpha_i}),
\end{equation}
where $\Pi^{\alpha_i}$ denotes the twist of $\Pi$ by the cocycle $\alpha_i$.

In particular, $\Sel_{S,\Pi}(\Xc)_{\alpha_i}$ has the same dimension as $H^1_{f,S}(G_k;\Pi^{\alpha_i})$, so its dimension may be computed by the methods of \S \ref{sec:global_ranks}-\ref{sec:motivic_decomp}.

\subsection{Local Geometry at Bad Places}\label{sec:local_bad}

Given $z \in \Xc(\Oc_{k,S})$, we may use \cite[Theorem 1.3.1]{BettsDogra20} to determine the $i$ for which $\kappa(z) \in \Sel_{S,\Pi}(\Xc)_{\alpha_i}$. For $v \in T_0 \setminus S$, let $l_v/k_v$ be a finite extension over which $\overline{X}$ has semistable reduction, and let $\overline{\Xc}^{r-ss}$ denote a regular semistable integral model of $\overline{X}_{l_v}$. Let $\Gamma_v=\Gamma_v(\overline{\Xc}^{r-ss})$ denote the reduction graph of $\overline{X}$, so that \[E(\Gamma_v)\] is the set of irreducible components of the special fiber of $\overline{\Xc}^{r-ss}$.

Then there is a natural map $\overline{\Xc}^{r-ss}(\Oc_{l_v}) \to E(\Gamma_v)$, hence a map
\[
\mathrm{red}_v \colon X(k) \to X(l_v) \to \overline{X}(l_v) = \overline{\Xc}^{r-ss}(\Oc_{l_v}) \to E(\Gamma_v)
\]

\begin{fact}[{\cite[Theorem 1.3.1]{BettsDogra20}}]\label{fact:thm131}
Let $\Xc^{r-ss}$ be an integral model of $X_{l_v}$ equal to the complement in $\overline{\Xc}^{r-ss}$ of a horizontal divisor $\Dc$. Let $x,y \in X(k_v) \cap \Xc^{r-ss}(\Oc_{l_v}) \subseteq X(l_v)$. Then
\begin{enumerate}
    \item\label{item:red_then_kappa} If $\red_v(x)=\red_v(y)$, then $\kappa_v(x)=\kappa_v(y)$
    
    \item\label{item:kappa_then_red} If $\Pi$ dominates $U_3(X)$, and $\kappa_v(x) = \kappa_v(y)$, then $\red_v(x) = \red_v(y)$.
\end{enumerate}
\end{fact}

\begin{rem}\label{rem:boundary_divisor_etale}
\cite[Theorem 1.3.1]{BettsDogra20} also requires the boundary divisor $\Dc$ to be \'etale over $\Oc_{l_v}$. One may arrange this by blowing up points on the boundary, but this does not change $\restr{\red_v}{\Xc^{r-ss}(\Oc_{l_v})}$ and is therefore not strictly necessary.
\end{rem}

\begin{rem}\label{rem:dependence_on_rss_model}
If $x$ or $y$ is not in $\Xc^{r-ss}(\Oc_{l_v})$, then the truth of ``$\red_v(x)=\red_v(y)$'' might depend on the chosen integral model $\overline{\Xc}^{r-ss}$. An example, suggested to us by A. Betts, is provided by $\Xc=\Xc^{r-ss} = \Poneminusthreepoints$, $v=v_3$, $x=2$, and $y=3$. Then $\red_v(x)=\red_v(y)$ because $|E(\Gamma_v(\overline{\Xc}^{r-ss}))|=1$, but $\kappa_v(x) \neq \kappa_v(y)$, at least for $\Pi$ dominating $U_3$.
\end{rem}

%% The following might be FALSE.

\begin{rem}\label{rem:non_integral_reduction}
Suppose $x,y \in X(l_v) \setminus \Xc^{r-ss}(\Oc_{l_v})$, and $\red_v(x)=\red_v(y)$. Then we still have $\kappa_v(x)=\kappa_v(y)$ for all $\Pi$ factoring through the quotient map
\[
U(X) \to U(\overline{X})
\]
but not for general quotients $\Pi$ of $U(X)$.

Relatedly, if $\kappa_v(x) = \kappa_v(y)$ for any $\Pi$ dominating $U_3(\overline{X})$, then we have $\red_v(x)=\red_v(y)$, regardless of whether $x,y$ are integral.
\end{rem}

In practice, we will choose a model $\Xc$ for $X$ over $\Oc_v$ (or even $\Oc_{k,S}$) for which it is clear that all elements of $\Xc(\Oc_v)$ extend to elements of $\Xc^{r-ss}(\Oc_{l_v})$.

In general, one may do this as follows. We suppose that $\Xc$ is given with a compactification $\Xc \subseteq \overline{\Xc}$. First, choose $l_v/k_v$ for which $X$ has semistable reduction. Then we may apply Lipman's resolution of singularities to $\overline{\Xc}_{\Oc_{l_v}}$ to obtain a regular semistable model $\overline{\Xc}^{r-ss}$. At each stage, we choose a model of $X_{l_v}$ in our model of $\overline{X}_{l_v}$ by taking the strict transform (not the preimage) of the boundary divisor. 

In one example, $\Xc$ will be the minimal Weierstrass model of a punctured elliptic curve with semistable reduction at all $v \in T_0 \setminus S = \{3,17\}$, so that $l_v=k_v$. The model is already regular at $17$, but at $3$ we obtain a regular model $\overline{\Xc}^{r-ss}$ by blowing up at the unique singular point. Then the smooth locus of $\overline{\Xc}^{r-ss}$ is the N\'eron model of $E=\overline{X}$, and the reduction of a point is its image in the N\'eron component group.

%**
%But in practical terms, how do you deal with the fact that there might be no global cohomology class for some of them??? I guess the point is that that component of the Selmer variety is either $H^1_{f,S}$ or empty, meaning that it at least has a function of the desired form.

%Will we be able to write the twisting construction into the explicit Tannakian theory?

%Alternative idea: change the basepoint each time. The question is how to deal with tuples of components that don't have a basepoint over them (maybe something with the gerbe?).

%Cute idea: preimage is a gerbe under the fundamental group

%**

\section{Mixed Elliptic Case}\label{sec:mixed_elliptic}

\begin{defn}\label{defn:mixed_elliptic}
We say that a curve $X/k$ is \emph{mixed elliptic} if the Jacobian $J_{\overline{X}}$ of the smooth compactification $\overline{X}$ of $X$ is isogenous to a power of an elliptic curve.
\end{defn}

When $X$ is mixed elliptic for an elliptic curve $E$, the Galois representations associated to $U(X)$ and its torsors will lie in a certain subcategory $\Rep_{\Qp}^{\rm g}(G_k, E)$ (Definition \ref{defn:elliptic_motives}) of $\Rep_{\Qp}^{\rm g}(G_k)$. We spend the rest of \S \ref{sec:mixed_elliptic} discussing this subcategory and its variants.

Fix a non-CM elliptic curve $E$ over a number field $k$. Since we work only with $p$-adic realizations, we introduce the notation
\[
h_1(E) \coloneqq H_1^{\et}(E_{\overline{k}},\Qp) = H^1_{\et}(E_{\overline{k}},\Qp(1)) = T_p(E) \otimes_{\Zl} \Qp.
\]

%We let $E' = E \setminus \{O\}$, where $O$ denotes the origin (identity element) of $E$.

We let \[M_{a,b} \coloneqq \mathrm{Sym}^a(h_1(E))(b).\] As $a$ ranges over non-negative integers and $b$ ranges over all integers, these give the set $\Irr{\glt}$ of irreducible objects of the Tannakian subcategory $\Rep_{\Qp}^{\rm ss}(G_k, E)$ of $\Rep_{\Qp}^{\rm g}(G_k)$ generated by $h_1(E)$. 

It follows from Serre's Open Image Theorem (\cite{SerreOpenImageI}) that:
\begin{fact}\label{fact:gl2}
The category $\Rep_{\Qp}^{\rm ss}(G_k, E)$ is a semisimple Tannakian category equivalent to $\Rep_{\Qp}(\mathrm{GL}_2)$. Furthermore, for a finite extension $l/k$, the functor $\Rep_{\Qp}^{\rm ss}(G_k, E) \to \Rep_{\Qp}^{\rm ss}(G_l, E)$ induced by restriction along $G_l \hookrightarrow G_k$ is an equivalence.
\end{fact}

The weight of $M_{a,b}$ as a Galois representation is
\[
w = -a - 2b.
\]

For $V=M_{a,b}$, we have
\[
V^{\vee}(1) \cong M_{a,-a-b+1},
\]
which has weight $-w-2$.

Motivated by the desire to check cases of the inequality (\ref{eqn:selmer_inequality2}), we would like to determine the conjectured values of
\[
d_{a,b} = h^1_f(G_k;M_{a,b}) \coloneqq \dim_{\Qp} H^1_f(G_k;M_{a,b})
\]
for $M_{a,b} \in \Irr{\glt}$ that appear as subquotients of $U(X)$ for $X$ mixed elliptic. In \S \ref{sec:global_ranks}, we will describe $h^1_f(G_k;M_{a,b})$ when $M_{a,b} \in \Irr^{\aeff}(\glt)$, the latter defined as follows:

\begin{defn}\label{defn:anti-effective}
We say that $M_{a,b}$ is \emph{anti-effective}\footnote{Motives arising from the cohomology of varieties are often called \emph{effective}. As this is the subring arising from the homology of varieties, we have chosen to call them \emph{anti-effective.}} if one of the following equivalent conditions holds:
\begin{enumerate}
    \item $M_{a,b}$ is a subquotient of $h_1(E)^{\otimes w}$
    \item $M_{a,b}$ is a subquotient of $\Qp$-unipotent \'etale fundamental group $U(X)$ for $X$ hyperbolic and mixed elliptic
    \item $M_{a,b}$ is a subquotient of $h_i(Y)$ for an algebraic variety $Y$
    \item $b \ge 0$
\end{enumerate}
We denote the set of such $M_{a,b}$ by $\Irr^{\aeff}(\glt)$.
\end{defn}

\begin{defn}\label{defn:elliptic_motives}
Let \[\Rep_{\Qp}^{\rm sg}(G_k, E)\] denote the thick subcategory of $\Rep_{\Qp}^{\rm sg}(G_k)$ generated by $\Rep_{\Qp}^{\rm ss}(G_k, E)$
\end{defn}
We then have
\[H^1_g(G_k;M_{a,b}) = \Ext^1_{\Rep_{\Qp}^{\rm sg}(G_k, E)}(\Qp(0),M_{a,b})
\]
if $w=-a-2b < 0$.

\begin{rem}\label{rem:g_sg}
We could instead use $\Rep_{\Qp}^{\rm g}(G_k, E)$, defined as the corresponding thick subcategory of $\Rep_{\Qp}^{\rm g}(G_k)$. This is conjecturally the same as $\Rep_{\Qp}^{\rm sg}(G_k, E)$ by Conjecture \ref{conj:g_sg}. Then we would know $H^1_g(G_k;M_{a,b}) = \Ext^1_{\Rep_{\Qp}^{\rm g}(G_k, E)}(\Qp(0),M_{a,b})$ for all $a,b$ without assuming Conjecture \ref{conj:BK}. Nonetheless, we find it more convenient in \S \ref{sec:tannakian_selmer} to use the category $\Rep_{\Qp}^{\rm sg}(G_k, E)$, as we will make use of the weight filtration. Doing so presents no trouble because the fundamental group of a curve has strictly negative weights, and all path torsors arise from geometry and therefore have a weight filtration.
\end{rem}

\begin{rem}The case $w=-1$ corresponds to $h_1(E) = M_{1,0}$, in which case $d_{a,b}$ is the dimension of the $p$-adic Selmer group, which the BSD conjecture predicts to be equal to the Mordell-Weil rank of $E$.\end{rem}

\begin{rem}\label{rem:mixed_motives}
Conjecturally, there is a category $\mathcal{MM}(k,\Qb)$ of mixed motives over $k$ with $\Qb$-coefficients. It should have a realization functor
\[
\mathcal{MM}(k,\Qb) \xrightarrow{\mathrm{real}_p} \Rep_{\Qp}^{\rm sg}(G_k)
\]
for which the induced functor $\mathcal{MM}(k,\Qp) \coloneqq \mathcal{MM}(k,\Qb) \otimes_{\Qb} \Qp \to \Rep_{\Qp}^{\rm sg}(G_k)$\footnote{The tensor product here is base extension from $\Qb$-linear categories to $\Qp$-linear categories; in particular, changes the class of objects as well as the $\Hom$-sets.} is an equivalence. This equivalence amounts to the Fontaine-Mazur, Tate, and Bloch-Kato conjectures. The latter means that the map induced by the realization functor
\[
\Ext^1_{\mathcal{MM}(k,\Qb)}(\Qb(0),M_{a,b}) \otimes_{\Qb} \Qp \to H^1_g(G_k;M_{a,b})
\]
is an isomorphism.
\end{rem}

\begin{rem}\label{rem:elliptic_motives}
The subcategory of $\mathcal{MM}(k,\Qb)$ corresponding under $\mathrm{real}_p$ to $\Rep_{\Qp}^{\rm sg}(G_k, E)$ is known as the category of \emph{mixed elliptic motives}. A candidate for this category was constructed in \cite{Patashnick13}.
\end{rem}

\begin{rem}
Nonconjecturally, as mentioned in \cite{GonMEM}, we may define \[\Ext^1_{\mathcal{MM}(k,\Qb)}(\Qb(0),M_{a,b})\]
as
\[K_{a+2b-1}(E^{(a)})^{(a+b)}_{\mathrm{sgn}} \otimes \Qb,\]
where $E^{(a)}$ denotes the kernel of the addition map $E^{a+1} \to E$, and $\mathrm{sgn}$ denotes the part on which $S_{a+1}$ acts by the sign representation. Note also that
\[
K_{a+2b-1}(E^{(a)})^{(a+b)} \otimes \Qb \cong \CH^{a+b}(E^{(a)},a+2b-1) \otimes \Qb \cong H^{a+1}(E^{(a)},\Qb(a+b))
\]
as $S_{a+1}$-representations

In the notation of loc.cit., our $M_{a,b}$ is $Sym^a\mathcal{H}(b)$.
\end{rem}

\subsection{\texorpdfstring{$S$}{S}-Integral Elliptic Motives}\label{sec:S-integral_elliptic}

When $X$ is proper, the question of determining $X(k)$ is the same as that of determining $\Xc(\Oc_k)$, so we will use $H^1_f$ to define Selmer varieties in \S \ref{sec:selmer_varieties}. If $X$ is not proper, we may consider $\Oc_{k,S}$-points, and we must then replace $H^1_f$ by $H^1_{f,S}$, so we set:
\[
d_{a,b}^S \colonequals h^1_{f,S}(G_k;M_{a,b}) = \dim H^1_{f,S}(G_k;M_{a,b}).
\]

It is well-known that $d_{1,0}^S$ is independent of $S$, while $d_{0,1}^S$ is very much dependent on $S$, as it is the rank of $\Gm(\Oc_{k,S})$.

Let $S' = S \cup \{v\}$ with $v \notin S \cup \{p\}$. As mentioned in \S \ref{sec:S-integral}, we have
\[
H^1_{f,S}(G_k;M_{a,b}) = H^1_{f,S'}(G_k;M_{a,b})
\]
whenever $\Qp(1)$ is not a quotient of $\restr{M_{a,b}}{G_v}$, which happens in particular if $-2$ is not one of its weights.

By Remark \ref{rem:WM_abel_var}, the Frobenius weights of $\restr{M_{1,0}}{G_v}$ are $\{-1\}$ if $E$ has potentially good reduction at $v$ and $\{0,-2\}$ otherwise. As a result, $\restr{M_{a,b}}{G_v}$ has Frobenius weights $\{-a-2b\}$ if $E$ has potentially good reduction at $v$ and weights $\{-2a-2b,-2a-2b+2,\cdots,-2b-2,-2b\}$ otherwise. We thus find by Corollary \ref{cor:change_in_S_WM} that $d_{a,b}^S = d_{a,b}^{S'}$ whenever:
\begin{itemize}
    \item $E$ has potentially good reduction at $v$, and $-a-2b \neq -2$. If $b \ge 0$, this happens as long as $(a,b) \neq (0,1), (2,0)$.

    \item $b \ge 2$
\end{itemize}

For $(a,b)=(1,0)$, its motivic weight is pure $-1$, so we have $d_{a,b}^S = d_{a,b}^{S'}$ in that case as well.

In particular, for $(a,b)=(1,0)$ or $b \ge 2$, we have
\[
d_{a,b}^S = d_{a,b}
\]
as long as $S \cap \{p\} = \emptyset$. If $S$ moreover contains no primes of bad reduction of $E$, this is true more generally for any $(a,b) \neq (0,1),(2,0)$ with $b \ge 0$.

In the case $(a,b)=(1,1)$, we have $d_{1,1}^S = d_{1,1}^{S'}$ if and only if:
\begin{itemize}
    \item $E$ does not have split multiplicative reduction at $v$.
\end{itemize}

Otherwise, $M_{1,1}^{\vee}(1)=h^1(E)$ has a nontrivial fixed part, so that $d_{1,1}^{S'} = d_{1,1}^{S}+1$ by Proposition \ref{prop:change_in_S}.

\begin{rem}\label{rem:elliptic_ext_S}

If $E$ has good reduction over $\Oc_{k,S}$, we set $\Rep_{\Qp}^{\rm f, S}(G_k, E)$ to be the intersection of $\Rep_{\Qp}^{\rm sg}(G_k, E)$ with $\Rep_{\Qp}^{\rm f, S}(G_k)$. More generally, for arbitrary $E$ and $\Oc_{k,S}$, we define
$\Rep_{\Qp}^{\rm f, S}(G_k, E)$ to be the subcategory of $V \in \Rep_{\Qp}^{\rm sg}(G_k,E)$ such that for every place $v \notin S$
    \begin{itemize}
        \item If $v \nmid p$, the weight filtration of $V$ splits as a representation of $I_{k_v}$.
        \item If $v \mid p$, the weight filtration of $V \otimes B_{cris}$ splits as a representation of $G_v$.
    \end{itemize}
Then
\[
H^1_{f,S}(G_k;M_{a,b}) = \Ext^1_{\Rep_{\Qp}^{\rm f, S}(G_k, E)}(\Qp(0),M_{a,b})
\]
for $w=-a-2b < 0$.
\end{rem}

\begin{rem}\label{rem:general_J_1}
We may analogously define categories $\Rep_{\Qp}^{\rm ss}(G_k, A)$, $\Rep_{\Qp}^{\rm sg}(G_k, A)$, and $\Rep_{\Qp}^{\rm f,S}(G_k, A)$ for any abelian variety $A$ in place of $E$. Then $\Rep_{\Qp}^{\rm ss}(G_k, A)$ will be equivalent to the category of representations of the Mumford-Tate group of $A$. To apply Chabauty-Kim to a curve $X$, one may take $A$ to be the Jacobian of $\overline{X}$. We plan to work with these in the future, the only obstacle being the messiness of the representation theory of reductive groups larger than $\glt$.
\end{rem}

\section{Ranks of Global Selmer Groups}\label{sec:global_ranks}

We find formulas for $d_{a,b}$ when $k=\Qb$, $w \le -2$, and $(a,b) \neq (0,1)$. Let $V=M_{a,b}$. Since $V \not\equiv \Qp(1)$, we have $h^0(G_k;V)=0$, and since $w \le -2$, we have $h^1_f(G_k;V^{\vee}(1)) = h^0(G_k;V^{\vee}(1)) = 0$.  Thus (\ref{eqn:Poitou-Tate}) becomes
\[
h^1_f(G_k;V) =  \sum_{v \mid p} \dim_{\Qp} (\operatorname{D}_{\dR}{V}/\operatorname{D}^+_{\dR}{V}) - \sum_{v \mid \infty} h^0(G_v;V)
\]

\subsection{The Case \texorpdfstring{$k=\Qb$}{k=Q}}\label{sec:global_ranks_Q}
There is one place above $p$ and one place above $\infty$, so we get
\[
h^1_f(G_k;V) = 
\dim \operatorname{D}_{\dR}{V}/\operatorname{D}^+_{\dR}{V} - h^0(G_{k_\infty},V).
\]

We first consider $\dim \operatorname{D}_{\dR}{V}/\operatorname{D}^+_{\dR}{V}$. If $b \ge 1$, then all Hodge weights are negative, so we get
\[
\dim \operatorname{D}_{\dR}{V}/\operatorname{D}^+_{\dR}{V} = \dim V = a+1.
\]
If $b=0$, then $\operatorname{D}^+_{\dR}{V}$ is one-dimensional, so we get
\[
\dim \operatorname{D}_{\dR}{V}/\operatorname{D}^+_{\dR}{V} = a.
\]

We next consider the term $h^0(G_{k_\infty},V)$. Note that $V$ is isomorphic to the regular representation of $G_{k_{\infty}} \cong C_2$. If $a$ is odd, then $\Sym^a(V) \cong V^{\oplus \frac{a+1}{2}}$ as a $C_2$-representation. Thus, we have $h^0(G_{k_\infty},V) = \frac{a+1}{2}$. If $a$ is even, then $\Sym^a(V)$ has $\frac{a}{2}+1$ copies of the trivial representation and $\frac{a}{2}$ copies of the sign representation. Therefore, if $b$ is even, then $h^0(G_{k_\infty},V) = \frac{a}{2}+1$, and if $b$ is odd, then $h^0(G_{k_\infty},V) = \frac{a}{2}$. In summary, we have

\[
    h^0(G_{k_{\infty}},V) =  \left\{\begin{array}{lr}
        \frac{a+1}{2}, &\; a\text{ odd}\\
        \frac{a}{2}+1, &\; a\text{ even},\, b\text{ even} \\
        \frac{a}{2}, &\; a\text{ even},\, b\text{ odd} 
        \end{array}\right\}.
\]

In summary, we have
\[
d_{a,b} = \dim \operatorname{D}_{\dR}{V}/\operatorname{D}^+_{\dR}{V} - h^0(G_{k_\infty},V) =
\left\{\begin{array}{lr}
        \frac{a-1}{2}, &\; a\text{ odd},\, b=0\\
        \frac{a}{2}-1, &\; a\text{ even},\, b=0\\
        \frac{a+1}{2}, &\; a\text{ odd},\, b \ge 1\\
        \frac{a}{2}, &\; a\text{ even},\, b\text{ even and }\ge 1\\
        \frac{a}{2}+1, &\; a\text{ even},\, b\text{ odd}
        \end{array}\right\}
\]

The relevant values are $d_{1,\ge 1},d_{2, \ge 1}, d_{3, \ge 0}, d_{4,\ge 0}, d_{5, \ge -1}, d_{6,\ge -1}, d_{7, \ge -2}, \cdots$. We record the values relevant to non-abelian Chabauty up to degree $5$:

\begin{eqnarray*}
d_{2,0} &=& 0\\
d_{3,0} &=& 1 \\
d_{1,1} &=& 1\\
d_{0,2} &=& 0\\
d_{2,1} &=& 2\\
d_{4,0} &=& 1\\
d_{1,2} &=& 1\\
d_{3,1} &=& 2\\
d_{5,0} &=& 2
\end{eqnarray*}

\subsection{The Case of \texorpdfstring{$k$}{k} Imaginary Quadratic}\label{sec:global_ranks_imag_quad}

We consider Selmer groups over $G_k$, where $k$ is an imaginary quadratic field, under the assumption that $E$ is defined over $\Qb$.

There is one place $v$ above $\infty$, and it is complex, so we simply have $h^0(G_{k_\infty},V) = \dim_{\Qp} V$. For $M_{a,b}$, this is $a+1$.

For $v \mid p$, there are two possibilities: 1) $v$ splits into two places 2) $v$ is inert or ramified. However, in either case, the contributions is the same; more specifically it is twice the contribution in the case $k=\Qb$. In the first case, this is because we sum over the two $v \mid p$, and in the latter case, it is because $\operatorname{D}_{\dR}$ outputs a vector space over $K_v$, whose dimension over $\Qp$ is twice is dimension over $K_v$. Thus the contribution is $2a$ for $b=0$ and $2a+2$ for $b \ge 0$.

We thus have
\[
    h^1_f(G_k;V) =  \left\{\begin{array}{lr}
        a-1, &\; b = 0\\

        a+1, &\; b \ge 1
        \end{array}\right\}.
\]

In weight $w \ge -5$, we list the values:

\begin{eqnarray*}
d_{2,0} &=& 1\\
d_{3,0} &=& 2 \\
d_{1,1} &=& 2\\
d_{0,2} &=& 1\\
d_{2,1} &=& 3\\
d_{4,0} &=& 3\\
d_{1,2} &=& 2\\
d_{3,1} &=& 4\\
d_{5,0} &=& 4
\end{eqnarray*}

\subsection{\texorpdfstring{$S$}{S}-Integral Global Selmer Groups}

We suppose $S \cap \{p\} = \emptyset$. By the discussion in \S \ref{sec:S-integral_elliptic}, we have

\begin{eqnarray*}
d_{1,0}^S &=& d_{1,0}\\
d_{0,2}^S &=& d_{0,2}\\
d_{1,2}^S &=& d_{1,2}
\end{eqnarray*}

For $k=\Qb$, we have
\begin{eqnarray*}
d_{0,1}^S = |S|
\end{eqnarray*}

Finally, let
\[
\delta_S(E)
\]
denote the number of places in $S$ at which $E$ has split multiplicative reduction. Then
\[
d_{1,1}^S = 1 + \delta_S(E).
\]

\section{Ranks of Local Selmer Groups}\label{sec:local_ranks}

In order to check the inequality (\ref{eqn:selmer_inequality2}), we need to compute dimensions of local Selmer varieties. To do so, we must compute the ranks of local Selmer groups.

Let $V=M_{a,b}$ have negative weight, and fix $\pf \in \{p\}$ of good reduction for $E$. Then \cite[Proposition 2.8]{BellaicheNotes} (or \cite[I.3.3.11]{FPR91}) gives the formula
\[
l_{a,b} \coloneqq h^1_f(G_{\pf};V) = \dim_{\Qp} (\operatorname{D}_{\dR}{V}/\operatorname{D}^+_{\dR}{V}).
\]

Let us suppose that $k_{\pf} \cong \Qp$. In our examples, we take either $k=\Qb$ or $\pf$ split in an imaginary quadratic field $k$.

As noted previously, we have
\[
H^1_f(G_{\pf};V) = \dim_{\Qp} (\operatorname{D}_{\dR}{V}/\operatorname{D}^+_{\dR}{V}) =
 \left\{\begin{array}{lr}
        a, &\; b = 0\\

        a+1, &\; b \ge 1
        \end{array}\right\}.
\]

We record these values in weight $w \ge -5$:

\begin{eqnarray*}
l_{1,0} &=& 1\\
l_{0,1} &=& 1\\
l_{2,0} &=& 2\\
l_{3,0} &=& 3\\
l_{1,1} &=& 2\\
l_{0,2} &=& 1\\
l_{2,1} &=& 3\\
l_{4,0} &=& 4\\
l_{1,2} &=& 2\\
l_{3,1} &=& 4\\
l_{5,0} &=& 5
\end{eqnarray*}

\subsection{Difference Between Local and Global}\label{sec:difference_lg}

The inequality (\ref{eqn:selmer_inequality2}) holds iff the difference between the dimensions of the local and global Selmer varieties is positive. We're therefore particularly interested in the difference \[c_{a,b} \coloneqq l_{a,b}-d_{a,b}.\]

If $\Pi$ denotes a finite-dimensional Galois-equivariant quotient of $U(X)$, then Chabauty-Kim for $\Pi$ produces a finite set iff the sum of $c_{a,b}$ over all weight-graded pieces of $\Pi$ (counted with multiplicity) is positive.

We let $r$ denote the rank of $E(k)$. We assume that this equals the $p$-Selmer rank $d_{1,0}$, an assumption that may be verified computationally. We also note that $d_{0,1}=0$ for $k=\Qb$ and $k$ imaginary quadratic. We explain the modification necessary for $S$-integral points below.

For $k=\Qb$, we have

\begin{eqnarray*}
c_{1,0} &=& 1-r\\
c_{0,1} &=& 1\\
c_{2,0} &=& 2\\
c_{3,0} &=& 2\\
c_{1,1} &=& 1\\
c_{0,2} &=& 1\\
c_{2,1} &=& 1\\
c_{4,0} &=& 3\\
c_{1,2} &=& 1\\
c_{3,1} &=& 2\\
c_{5,0} &=& 3
\end{eqnarray*}

For $k$ imaginary quadratic and $v$ split in $k$, we have
\begin{eqnarray*}
c_{1,0} &=& 1-r\\
c_{0,1} &=& 1\\
c_{2,0} &=& 1\\
c_{3,0} &=& 1\\
c_{1,1} &=& 0\\
c_{0,2} &=& 0\\
c_{2,1} &=& 0\\
c_{4,0} &=& 1\\
c_{1,2} &=& 0\\
c_{3,1} &=& 0\\
c_{5,0} &=& 1.
\end{eqnarray*}

For a set $S$ of places of $k$, we set $c_{a,b}^S \colonequals l_{a,b}-d^S_{a,b}$. Then
\begin{eqnarray*}
c_{1,0}^S &=& c_{1,0}\\
c_{0,2}^S &=& c_{0,2}\\
c_{1,2}^S &=& c_{1,2}
\end{eqnarray*}

For $k=\Qb$, we have
\begin{align*}
c_{0,1}^S = 1-|S|\\
c_{1,1}^S = 1-\delta_S(E)
\end{align*}

\section{Motivic Decomposition of Fundamental Groups}\label{sec:motivic_decomp}

Let $X$ be a smooth curve over a number field $k$. To simplify notation, we fix an implicit basepoint $b \in X(k)$ and  and set
\[
U = U(X) \colonequals \pi_1^{\et,\un}(X_{\overline{k}},b)_{\Qp}
\]
and
\begin{align*}
    U^1 = U\\
    U^{n+1} = [U,U^n]\\
    U_n = U/U^{n+1}\\
    U[n] = U^n/U^{n+1},
\end{align*}
where commutator denotes the closure of the group-theoretic commutator. Notice the short exact sequence
\[
0 \to U[n] \to U_n \to U_{n-1} \to 0.
\]

We let $\Pi$ be a finite-dimensional Galois-equivariant quotient of $U$. In particular, $\Pi$ factors through $U_n$ for some $n$. The main goal of this section is to explain how to use the results of \S \ref{sec:global_ranks}-\ref{sec:local_ranks} to bound $\dim H^1_f(G_k,\Pi)$ and compute $\dim H^1_f(G_{\pf};\Pi)$. The key to doing this is understanding a certain class of $\Pi$ in the ring $K_0(\Rep_{\Qp}^{\rm sg}(G_k))$, explained in \S \ref{sec:k0_classes_unipotent}. Doing this will produce an algorithm for checking when the inequality (\ref{eqn:selmer_inequality2}) is satisfied.\footnote{More precisely, it allows us to either verify that the inequality is satisfied or that Conjecture \ref{conj:FPR3} implies that the inequality is not satisfied.}

Finally, for the purposes of comparing with Quadratic Chabauty, we define
\[
U_Q = U_Q(X)
\]
to be the quotient of $U_2$ by the maximal subspace of $U[2]$ with no subrepresentation isomorphic to $\Qp(1)$.

\begin{rem}\label{rem:indep_of_b}
All computations in this section are completely independent of $b$, essentially because the graded pieces $U[n]$ are homological in nature.\end{rem}

\subsection{\texorpdfstring{$K_0$}{K0} Classes of Unipotent Groups with Galois Action}\label{sec:k0_classes_unipotent}

We define linear functions
\[
d,l, d^S \colon K_0(\Rep_{\Qp}^{\rm sg}(G_k)) \to \Zb
\]
by $d([M_{a,b}]) = d_{a,b}$ (resp. $l([M_{a,b}]) = l_{a,b}$, $d^S([M_{a,b}]) = d^S_{a,b}$). We set
\[
c \colonequals l - d
\]
and $c^S \colonequals l - d^S$.

To such subquotient $\Pi$ of $U$, we associate a class $[\Pi]$ in the ring
\[
K_0(\Rep_{\Qp}^{\rm sg}(G_k))
\]
defined by requiring 
\[
[\Pi_2] = [\Pi_1]+[\Pi_3]
\]
when there is a short exact sequence
\begin{equation}\label{eqn:pi_ses}
0 \to \Pi_1 \to \Pi_2 \to \Pi_3 \to 0
\end{equation}
and that if $\Pi$ is abelian, then $[\Pi]$ is the class of the corresponding Galois representation.

We use the notation
\[
d(\Pi)
\]
(resp. $l(\Pi)$, $c(\Pi)$, $d^S(\Pi)$, $c^S(\Pi)$) to refer to $d([\Pi])$ (resp. $l([\Pi])$, $c([\Pi])$, $d^S([\Pi])$, $c^S([\Pi])$).

By the short exact sequence for Galois cohomology, we have
\begin{align}
\dim H^1_{f,S}(G_k,\Pi) \le d^{S}(\Pi)\\
\dim H^1_f(G_{\pf};\Pi) = l(\Pi),
\end{align}
the latter by the fact that the weights are negative and that crystalline $H^2$ vanishes.

\begin{rem}\label{rem:exact_dimension}
In fact, it would follow from Conjecture \ref{conj:FPR3} below that
\[
\dim H^1_{f,S}(G_k,\Pi) \le d^{S}(\Pi).
\]
\end{rem}

Our goal for the rest of this section is to explain how to compute the class of $\Pi$ in $K_0(\Rep_{\Qp}^{\rm sg}(G_k))$.

For simplicity, we focus on the case $\Pi=U_n$ for a positive integer $n$. The group $U_n$ may be identified via the Lie exponential with its Lie algebra $\operatorname{Lie}{U_n}$, which has the structure of a $p$-adic Galois representation. Let $W$ denote the weight filtration and $\operatorname{Gr}^W_{\bullet}$ the associated graded for the weight filtration.

We suppose $k$ is a positive integer less than or equal to $n$, until otherwise specified. The $-k$th weight-graded piece is
\[
\operatorname{Gr}^W_{-k} \operatorname{Lie}{U_n} = U[k].
\]
In particular, we have
\[
[\operatorname{Lie}{U_n}] = [U_n] = \sum_{k=1}^{n} [U[k]] \in K_0(\Rep_{\Qp}^{\rm sg}(G_k))
\]

\subsection{Decomposition of \texorpdfstring{$U[k]$}{U[k]} in terms of \texorpdfstring{$U_1$}{U1}}\label{sec:decomp_Uk}

We outline a general procedure for computing the class of $U_n$ in $K_0(\Rep_{\Qp}^{\rm sg}(G_k))$ in terms of $h_1(X)$.

As $U_n$ is a unipotent group, we have $\Oc(U_n) \cong \operatorname{Sym} \operatorname{Lie}{U_n}^{\vee}$ as $\Qp$-algebras, in a way that respects Galois action on associated graded. We thus have
\[
\operatorname{Gr}^W_{\bullet} \Oc(U_n)^{\vee} \cong \operatorname{Gr}^W_{\bullet} \operatorname{Sym} \operatorname{Lie}{U_n}.
\]

If we know the structure of $\Gr^W_{\bullet} \Oc(U_n)$ as a motive, this allows us to inductively compute the structure of $U^k/U^{k+1}$ as follows. We have
\begin{eqnarray*}
\Gr^W_{-k} \Oc(U_n)^{\vee} &=& \Gr^W_{-k} \operatorname{Sym} \operatorname{Lie}{U_n}\\
&=& \Gr^W_{-k} \operatorname{Sym} \operatorname{Gr}^W_{\ge -k} \operatorname{Lie}{U_n}\\
&=& \Gr^W_{-k} \operatorname{Sym} (\Gr^W_{-k} \operatorname{Lie}{U_n} \oplus  \operatorname{Gr}^W_{> -k} \operatorname{Lie}{U_n})\\
&=& \Gr^W_{-k} \operatorname{Lie}{U_n} \oplus \Gr^W_{-k} \operatorname{Sym} \operatorname{Gr}^W_{> -k} \operatorname{Lie}{U_n}\\
&=& U[k] \oplus \Gr^W_{-k} \operatorname{Sym}(\oplus_{i=1}^{k-1} U[i])
\end{eqnarray*}

Note that $n$ is irrelevant here, as long as $n \ge k$. We thus get
\[
[U^k/U^{k+1}] = [\Gr^W_{-k} \Oc(U_n)^{\vee}] - [ \Gr^W_{-k}\operatorname{Sym}(\oplus_{i=1}^{k-1} U[i])] = \pr_{-k}([\Oc(U_n)^{\vee}] - [\operatorname{Sym}(\oplus_{i=1}^{k-1} U[i])])
\]

The term $\pr_{-k} [\operatorname{Sym}(\oplus_{i=1}^{k-1} U[i])]$ decomposes according to nontrivial partitions of $k$. We represent a partition by a sequence $n_1,\cdots,n_k$ for which $\sum_{i=1}^k i n_i = k$, and we call it nontrivial if $n_k=0$. We then have
\begin{equation}\label{eqn:sym_decomp}
\pr_{-k} [\operatorname{Sym}(\oplus_{i=1}^{k-1} U[i])] = \sum_{\sum_{i=1}^{k-1} i n_i = k} \prod_{j=1}^{k-1} [\Sym^{n_j}(U[j])].
\end{equation}

%\subsection{Description of \texorpdfstring{$\Gr^W_{\bullet} \Oc(U_n)^{\vee}$}{GrW O(U_n)}}

If $X$ is affine, then $U$ is a free pro-unipotent group, and we have
\[
\Gr^W_{-k} \Oc(U_n)^{\vee} \cong h_1(X)^{\otimes k}.
\]

\subsubsection{Projective Case}\label{sec:decomp_projective}

Suppose $X$ is projective, and let $X'$ be the complement of a point in $X$. We set $U \colonequals U(X)$ and $U' \colonequals U(X')$. Then 
\[
\Lie{U'} \cong \FreeLie{h_1(X)}
\]
is free on $h_1(X)=h_1(X')$, while
\[
U
\]
is the quotient of $U'$ by an element of $U'^2 \setminus U'^3$ on which $\glt$ acts as $M_{0,1}$. More precisely, this element corresponds to the dual
\[
h_2(X) \to \wedge^2 h_1(X) \cong U(X')[2]
\]
of the intersection pairing $\wedge^2 h^1(X) \to h^2(X)$.

For $k=1,2,3$, this does nothing more than remove a copy of $h_2(X) \cong \Qp(1)$ from $U'[2]$ and remove a copy of $h_1(X)(1)$ from $U'[3]$. In the projective case, we prefer to first compute the associated graded of the Lie algebra as if it were affine and then mod out by the appropriate Lie ideal.

\subsection{Elliptic Motive Case}\label{sec:elliptic_case}

We suppose for the rest of \S \ref{sec:motivic_decomp} that $X$ is mixed elliptic.

In this case, every $[\Pi]$ is in the subring
\[
 K_0(\Rep_{\Qp}^{\rm ss}(G_k,E)) = K_0(\Rep_{\Qp}^{\rm g}(G_k, E)) \subseteq K_0(\Rep_{\Qp}^{\rm sg}(G_k)),
\]
which is naturally the free abelian group
\[
\bigoplus_{a \in \Zb_{\ge 0}, b \in \Zb} \Zb [M_{a,b}],
\]
with product determined by the rule 
\begin{equation}\label{eqn:tensor_product_rule}
    [M_{1,b_1}][M_{k,b_2}] = [M_{k+1,b_1+b_2}] + [M_{k-1,b_1+b_2+1}],
\end{equation}
where $[M_{-1,b}]=0$ by convention. More precisely, it lies in the anti-effective subring
\[
 K_0(\Rep_{\Qp}^{\rm ss}(G_k,E))^{\aeff} \colonequals \bigoplus_{a,b \ge 0} \Zb [M_{a,b}].
\]

We will need to use the decomposition of tensor powers of $h_1(E)$ in terms of the $M_{a,b}$. One may compute the following using (\ref{eqn:tensor_product_rule}):
\begin{align*}
[h_1(E)]^{2} &=& [M_{2,0}] + [M_{0,1}] \\
[h_1(E)]^{3} &=& [M_{3,0}] + 2 [M_{1,1}]\\
[h_1(E)]^{4} &=& [M_{4,0}]+3[M_{2,1}]+2[M_{0,2}]\\
[h_1(E)]^{5} &=& [M_{0,5}]+4[M_{3,1}]+5[M_{1,2}]
\end{align*}

\subsection{Explicit Decomposition for a Punctured Elliptic Curve}\label{sec:punctured_ell_curve}

Let $X=E' = E\setminus\{O\}$ for a non-CM elliptic curve $E$. In this case, we determine $[U[k]]$ for $k=1,2,3,4$. The cases $k=1,2,3$ will be used for the examples of \S \ref{sec:level_3}-\ref{sec:newton_polygons}.

\subsubsection{Level 1}\label{sec:punctured_ell_curve_1}

We have $U_1 = U[1] \cong h_1(X)$. Thus $[U_1]=[h_1(X)]=[M_{1,0}]$.

\subsubsection{Level 2}\label{sec:punctured_ell_curve_2}

We next have $[U[2]]=[h_1(X)^{\otimes 2}] - [\Sym^2 h_1(X)] = [\wedge^2 h_1(X)] = [M_{0,1}]$.

\subsubsection{Level 3}\label{sec:punctured_ell_curve_3}

We get $\pr_{-3}[\Sym(U[1] \oplus U[2])] = [\Sym^3{U[1]}] + [U[1]][U[2]] = [M_{3,0}]+[M_{1,1}]$. We also have $[h_1(E)^{\otimes 3}] = [h_1(E)]^3 = [M_{3,0}] + 2[M_{1,1}]$. Thus
\[
[U[3]] = [M_{3,0}] + 2[M_{1,1}] - ([M_{3,0}]+[M_{1,1}]) = [M_{1,1}]
\]

\subsubsection{Level 4}\label{sec:punctured_ell_curve_4}

We get \begin{eqnarray*}
\pr_{-4}[\Sym(U_1 \oplus U[2] \oplus U[3])] &=& [\Sym^4{U_1}] + [U_1][U[3]] + [\Sym^2 U_1][U[2]] + [\Sym^2 U[2]]\\
&=& [M_{4,0}] + [M_{1,0}][M_{1,1}] + [M_{2,0}][M_{0,1}] + [M_{0,1}]^2\\
&=& [M_{4,0}] + [M_{2,1}] + [M_{0,2}] + [M_{2,1}] + [M_{0,2}]\\
&=& [M_{4,0}] + 2[M_{2,1}] + 2[M_{0,2}]
\end{eqnarray*}
We thus get
\[
[U^4/U^5] = [h_1(E)]^{4}  - ([M_{4,0}] + 2[M_{2,1}] + 2[M_{0,2}]) = [M_{2,1}].
\]

\subsubsection{Dimensions}\label{sec:punctured_ell_curve_dim}

We compute $c^S(U_2)$ and $c^S(U_3)$ for $k=\Qb$. We have
\begin{eqnarray*}
c^S(U_2) &=& c^S(U[1]) + c^S(U[2])\\
&=& c^S_{1,0} + c^S_{0,1}\\
&=& 1-r + (1-|S|)\\
&=& 2 - r - |S|.
\end{eqnarray*}

In particular, Quadratic Chabauty may fail when $r+|S| \ge 2$ (note that $U_2=U_Q$ for $X=E'$). We consider the case $r=|S|=1$ in the examples of \S \ref{sec:level_3}-\ref{sec:newton_polygons}.

On the other hand, we have
\begin{eqnarray*}
c^S(U_3) &=& c^S(U_2) + c^S(U[3])\\
&=& 2 - r - |S| + c^S_{1,1}\\
&=& 2 - r - |S| + 1 - \delta_S(E)\\
&=& 3 - r - |S| - \delta_S(E).
\end{eqnarray*}

In particular, if $r=|S|=1$, and $\delta_S(E) = 0$, then the Chabauty-Kim method will give finiteness in level $3$.

\subsection{Formulas for Projective Curves in Levels \texorpdfstring{$\le 3$}{=<3}}\label{sec:general_computations}

We suppose that $X=\overline{X}$ is projective, so that
\[
h_1(X) \cong h_1(\overline{X}) \cong h_1(E)^{g},
\]
where $g$ is the genus of $X$. We recall from \S \ref{sec:decomp_projective} that $X'$ is $X$ punctured at one point, and $U'=U(X')$.

We have
\[
[U'[1]] = [U[1]] = [h_1(X)] = g[M_{1,0}].
\]

We will similarly find formulas for $[U[2]]$ and $[U[3]]$ in terms of $g$. More specifically, we first find $[U'[2]]$ and $[U'[3]]$ and then apply \S \ref{sec:decomp_projective}. Recall that
\[
\Gr^W_{-k} \Oc(U'_n)^{\vee} \cong h_1(X)^{\otimes k}.
\]
for $k \le n$, so we will need to analyze the $S_k$-action on $h_1(X)^{\otimes k}$. For this, we set
\[
V \colonequals H^0(G_k,h_1(X) \otimes h^1(E)),
\]
so that $h_1(X) = h_1(E) \otimes V$. Then $V$ is naturally the rank $g$ trivial object of $\Rep_{\Qp}^{\rm sg}(G_k)$.

For a positive integer $k$, we have
\[
h_1(X)^{\otimes k} \cong h_1(E)^{\otimes k} \otimes V^{\otimes k}
\]
as $S_k$-representations, where $S_k$ acts individually on each of the three tensor powers.

\subsubsection{Affine Case in Level \texorpdfstring{$2$}{2}}\label{sec:affine_ME_2}

We compute $[U'[2]]$. We have
\begin{eqnarray*}
\Sym^2 h_1(X) \oplus \wedge^2 h_1(X) &\cong& h_1(X)^{\otimes 2}\\
&\cong&
h_1(E)^{\otimes 2} \otimes V^{\otimes 2}\\
&\cong&
(\Sym^2 h_1(E) \oplus \wedge^2 h_1(E)) \otimes (\Sym^2 V \oplus \wedge^2 V)\\
&\cong&
\Sym^2 h_1(E) \otimes \Sym^2 V \oplus \Sym^2 h_1(E) \otimes \wedge^2 V\\
& & \oplus
\wedge^2 h_1(E)) \otimes \Sym^2 V \oplus \wedge^2 h_1(E)) \otimes \wedge^2 V
\end{eqnarray*}

Comparing $S_2$-actions on both sides, we get
\[
\Sym^2 h_1(X) \cong \Sym^2 h_1(X) \otimes \Sym^2 V \oplus \wedge^2 h_1(E) \otimes \wedge^2 V
\cong
M_{2,0} \otimes \Sym^2 V \oplus M_{0,1} \otimes \wedge^2 V
\]
and
\[
\wedge^2 h_1(X) = \Sym^2 h_1(X) \otimes \wedge^2 V \oplus \wedge^2 h_1(E) \otimes \Sym^2 V
\cong
M_{2,0} \otimes \wedge^2 V \oplus M_{0,1} \otimes \Sym^2 V
\]

In particular, we have
\[
[U'[2]] = [\wedge^2 V][M_{2,0}] + [\Sym^2 V][M_{0,1}]
=
\left(\frac{g(g-1)}{2}\right)[M_{2,0}]
+
\left(\frac{g(g+1)}{2}\right)[M_{0,1}].
\]

\subsubsection{Affine Case in Level \texorpdfstring{$3$}{3}}\label{sec:affine_ME_3}

We now turn to $[U'[3]]$.

Let $A$ denote the trivial representation of $S_3$, $B$ the sign representation, and $C$ the standard two-dimensional representation.

We have 
\[
[h_1(X)^{\otimes 3}] = [h_1(E)]^3[V]^3 = g^3[h_1(E)]^3 = g^3[M_{3,0}]+2g^3[M_{1,1}].
\]

We get \[\pr_{-3}[\Sym(U'_1 \oplus U'[2])] = [\Sym^3{U'_1}] + [U'_1][U'[2]].\]

Note that $\Sym^3{U'_1} = \Sym^3{h_1(X)} \cong (h_1(E)^{\otimes 3} \otimes V^{\otimes 3})^{S_3}$. We must therefore decompose each of $h_1(X)^{\otimes 3}$ and $V^{\otimes 3}$ as a representation of $S_3 \times \glt$.

We know that $h_1(E)^{\otimes 3}$ decomposes as $M_{3,0} \oplus M_{1,1} \oplus M_{1,1}$. Since the $A$-isotypical piece is $\Sym^3 h_1(E) = M_{3,0}$, and $\wedge^3 h_1(E)=0$, we get that $S_3$ acts on the $M_{1,1} \oplus M_{1,1}$-piece of $h_1(E)^{\otimes 3}$ via two copies of $C$. Thus in total, we find that
\[
h_1(E)^{\otimes 3} \cong M_{3,0} \otimes A \oplus M_{1,1} \otimes C
\]
as a representation of $\glt \times S_3$.

For $V$, we have $\dim{V^{\otimes 3}} = g^3$, $a \colonequals \dim{\Sym^3 V} = \binom{g+2}{3} = \frac{g(g+1)(g+2)}{6}$, and $b \colonequals \dim{\wedge^3 V} = \binom{g}{3} = \frac{g(g-1)(g-2)}{6}$. We set $c \colonequals g^3 - a - b = \frac{g^3-g}{3}$.

Therefore, as a representation of $S_3$, we get
\[
V^{\otimes 3} \cong A^{\oplus a} \oplus B^{\oplus b} \oplus C^{\oplus c}.
\]

We therefore find
\begin{eqnarray*}
h_1(X)^{\otimes 3} &\cong&  h_1(E)^{\otimes 3} \otimes V^{\otimes 3}\\
&\cong& (M_{3,0} \otimes A \oplus M_{1,1} \otimes C) \otimes (A^{\oplus a} \oplus B^{\oplus b} \oplus C^{\oplus c})\\
&\cong& (M_{3,0} \otimes A)^{\oplus a} \oplus (M_{3,0} \otimes B)^{\oplus b} \oplus (M_{3,0} \otimes C)^{\oplus c} \oplus (M_{1,1} \otimes C)^{\oplus a} \oplus (M_{1,1} \otimes C)^{\oplus b}\\
& & \oplus (M_{1,1} \otimes C \otimes C)^{\oplus c}\\
&\cong& M_{3,0}^{\oplus a} \otimes A \oplus M_{3,0}^{\oplus b} \otimes B \oplus (M_{3,0}^{\oplus c} \oplus M_{1,1}^{\oplus a+b}) \otimes C \oplus M_{1,1}^{\oplus c} \otimes (C \otimes C)
\end{eqnarray*}

Using the decomposition $C \otimes C \cong A \oplus B \oplus C$, we find that 
\[
\Sym^3{h_1(X)} = (h_1(E)^{\otimes 3} \otimes V^{\otimes 3})^{S_3} \cong M_{3,0}^{\oplus a} \oplus M_{1,1}^{\oplus c},
\]
and thus
\[
[\Sym^3{h_1(X)}] = a[M_{3,0}] + c [M_{1,1}]
=
\left(\frac{g(g+1)(g+2)}{6}\right)[M_{3,0}] + \left(\frac{g^3-g}{3}\right)[M_{1,1}].
\]

Finally, we have
\begin{eqnarray*}
[U'_1][U'[2]] &=& (g[M_{1,0}])\left(\left(\frac{g(g-1)}{2}\right)[M_{2,0}]
+
\left(\frac{g(g+1)}{2}\right)[M_{0,1}]\right)\\
&=& \left(\frac{g^2(g-1)}{2}\right)([M_{3,0}]+[M_{1,1}])+\left(\frac{g^2(g+1)}{2}\right)[M_{1,1}]\\
&=& \left(\frac{g^2(g-1)}{2}\right)[M_{3,0}] + g^3[M_{1,1}].
\end{eqnarray*}

Then
\begin{eqnarray*}
[U'[3]] &=& [h_1(X)^{\otimes 3}] - \pr_{-3}[\Sym(U'_1 \oplus U'[2])]\\
&=& [h_1(X)^{\otimes 3}] - [\Sym^3{h_1(X)}] - [U'_1][U'[2]]\\
&=& 
g^3[M_{3,0}]+2g^3[M_{1,1}]
-
\left[
\left(\frac{g(g+1)(g+2)}{6}\right)[M_{3,0}] + \left(\frac{g^3-g}{3}\right)[M_{1,1}]
\right]\\
& &
-
\left[
\left(\frac{g^2(g-1)}{2}\right)[M_{3,0}] + g^3[M_{1,1}]
\right]\\
&=&
\left(g^3
-
\frac{g(g+1)(g+2)}{6}
-
\frac{g^2(g-1)}{2}
\right)[M_{3,0}]
+
\left(2g^3
-
\frac{g^3-g}{3}
-
g^3
\right)[M_{1,1}]\\
&=& \left(\frac{g^3-g}{3}\right)[M_{3,0}] + \left(\frac{2g^3+g}{3}\right)[M_{1,1}].
\end{eqnarray*}

\begin{rem}
I wonder if Schur-Weyl duality could simplify some of the calculations above.
\end{rem}

\subsubsection{Projective Case}\label{sec:projective_mixed_elliptic}

As described in \S \ref{sec:decomp_projective}, we may compute $[U[k]]$ in terms of $[U'[k]]$, where the latter is computed as in \S \ref{sec:affine_ME_2}-\ref{sec:affine_ME_3}. We note in particular that
\begin{eqnarray*}
[U[2]] &=& [U'[2]] - [M_{0,1}]\\
&=&
\left(\frac{g(g-1)}{2}\right)[M_{2,0}]
+
\left(\frac{g(g+1)}{2}\right)[M_{0,1}]
-
[M_{0,1}]\\
&=&
\left(\frac{g(g-1)}{2}\right)[M_{2,0}]
+
\left(\frac{g(g+1)}{2} - 1\right)[M_{0,1}].
\end{eqnarray*}
and
\begin{eqnarray*}
[U[3]] &=& [U'[3]] - g[M_{1,1}]\\
&=& \left(\frac{g^3-g}{3}\right)[M_{3,0}] + \left(\frac{2g^3+g}{3}\right)[M_{1,1}]
-
g[M_{1,1}]\\
&=&
\left(\frac{g^3-g}{3}\right)[M_{3,0}] + \left(\frac{2g^3-2g}{3}\right)[M_{1,1}].
\end{eqnarray*}

\subsection{Explicit Computations for Higher Genus Projective Curves}\label{sec:decomp_higher_genus}

We now discuss the computation of $[U[k]]$ for $g=2$ and $g=4$ with attention to particular examples. While we do not need them for the examples of \S \ref{sec:level_3}-\ref{sec:newton_polygons}, we do it both to demonstrate in practice the methods of \S \ref{sec:decomp_Uk} and with a view toward future applications.

\subsubsection{Genus \texorpdfstring{$2$}{Curves}}\label{sec:decom_genus2}

As described in \S \ref{sec:projective_mixed_elliptic}, we have
\[
[U[1]] = g[M_{1,0}] = 2[M_{1,0}],
\]
\[
[U[2]
=
\left(\frac{g(g-1)}{2}\right)[M_{2,0}]
+
\left(\frac{g(g+1)}{2} - 1\right)[M_{0,1}]
=
[M_{2,0}] + 2[M_{0,1}],
\]
and
\[
[U[3]]
=
\left(\frac{g^3-g}{3}\right)[M_{3,0}] + \left(\frac{2g^3-2g}{3}\right)[M_{1,1}]
=
2[M_{3,0}]
+
4[M_{1,1}].
\]

For $k=\Qb$ or $k$ imaginary quadratic, we have
\begin{eqnarray*}
c(U_Q) &=& c(U[1]) + c(U_Q[2])\\
&=& 2c_{1,0} + 2c_{0,1}\\
&=& 2(1-r) + 2\\
&=& 4 - 2r.
\end{eqnarray*}

In particular, Quadratic Chabauty applies whenever $r \le 1$, or equivalently $\rank_{\Qb} J_X(k) \le 2$. Every mixed elliptic genus $2$ curve over $\Qb$ on LMFDB satisfies this condition over $k=\Qb$.

Let's see what happens for $k$ imaginary quadratic and the full level $2$ quotient. We have
\begin{eqnarray*}
c(U_2) &=& c(U[1]) + c(U[2])\\
&=& 2c_{1,0} + 2c_{0,1} + c_{2,0}\\
&=& 2(1-r) + 2 + 1\\
&=& 5 - 2r.
\end{eqnarray*}

It follows that if $r = \rank_{\Qb} E(k) = 2$, then the inequality (\ref{eqn:selmer_inequality2}) holds for $\Pi=U_2$. (In fact, the same is true for $k=\Qb$, as $c(U_2)= 6 - 2r$, but we do not have an example of this.)

For the genus $2$ curve with LMFDB label 38416.a.614656.1, given by
\[
y^2 = x^6 - 3x^5 - x^4 + 7x^3 - x^2 - 3x + 1,
\]
we have $E$ the curve with Cremona label ``196a2''. In this case, $r=2$ for $k=\Qb(\sqrt{-d})$ and $d=1,2,5,6,10,14,17,21$.

Its twist by $-1$ is the genus $2$ curve with LMFDB label 614656.a.614656.1, given by
\[
y^2 = -x^6 - 3x^5 + x^4 + 7x^3 + x^2 - 3x - 1,
\]
and $E$ has Cremona label ``784i1''. In this case, $r=2$ for $k=\Qb(\sqrt{-d})$ and $d=1,5,17,21,33,37,53$.

For completeness, we compute $c(U_3)$, both when $k=\Qb$ and $k$ is imaginary quadratic. For $k=\Qb$, we have
\begin{eqnarray*}
c(U_3) &=& c(U_2) + c(U[3])\\
&=& 6 - 2r + 2c_{3,0}+4c_{1,1}\\
&=& 6 - 2r + 2(2) + 4(1)\\
&=& 14 - 2r
\end{eqnarray*}
while for $k$ imaginary quadratic,
\begin{eqnarray*}
c(U_3) &=& c(U_2) + c(U[3])\\
&=& 5 - 2r + 2c_{3,0}+4c_{1,1}\\
&=& 5 - 2r + 2(1) + 4(0)\\
&=& 7 - 2r.
\end{eqnarray*}

Notice that if $r=3$, then (\ref{eqn:selmer_inequality2}) holds in level $3$ but not level $2$.

%%% The following two are nice. QC still works over Q, but not necessarily over an imaginary quadratic field (either because the rank goes up or because other Selmer groups are larger over that field).

% https://www.lmfdb.org/Genus2Curve/Q/38416/a/614656/1

% https://www.lmfdb.org/Genus2Curve/Q/614656/a/614656/1

%% for a real quadratic example maybe try 12544.d.25088.1 or 193600.e.968000.1 or 202500.a.405000.1

%% 589824.b.589824.1 has two non-isogenous elliptic curves, but is it possible they are isogenous over an extension??
%% Similar question for 25600.a.102400.1, 236196.a.472392.1, 36864.b.36864.1, 25600.a.102400.1

%%% Search for type J(E_2) with analytic rank 1 for things that split over a quartic field

\subsubsection{Explicit Decomposition for Bring's Curve}\label{sec:brings_curve}

We now demonstrate the decomposition for $g=4$. We were encouraged to do this by B. Mazur, as it applies to Bring's curve, given by the three homogeneous equations
\begin{align*}
\sum_{i=0}^4 x_i^k = 0 \hspace{1in} k = 1,2,3
\end{align*}
in $\Pb^4$. When $X$ is Bring's curve, there is an elliptic curve $E$ given by Cremona label `50a1' for which
\[
h_1(X) \cong h_1(E)^{\oplus 4}.
\]

As described in \S \ref{sec:projective_mixed_elliptic}, we have
\[
[U[1]] = g[M_{1,0}] = 4[M_{1,0}],
\]
\[
[U[2]
=
\left(\frac{g(g-1)}{2}\right)[M_{2,0}]
+
\left(\frac{g(g+1)}{2} - 1\right)[M_{0,1}]
=
6[M_{2,0}] + 9[M_{0,1}],
\]
and
\[
[U[3]]
=
\left(\frac{g^3-g}{3}\right)[M_{3,0}] + \left(\frac{2g^3-2g}{3}\right)[M_{1,1}]
=
20[M_{3,0}]
+
40[M_{1,1}].
\]

For $k$ imaginary quadratic, we have
\begin{eqnarray*}
c(U_Q) &=& c(U[1]) + c(U_Q[2])\\
&=& 4c_{1,0} + 9c_{0,1}\\
&=& 4(1-r) + 9\\
&=& 13 - 4r.
\end{eqnarray*}

In particular, Quadratic Chabauty applies whenever $r \le 3$, or equivalently $\rank_{\Qb} J_X(k) \le 12$.

\section{Tannakian Selmer Varieties}\label{sec:tannakian_selmer}

The goal of this section is to explicitly understand
\[
H^1_{f,S}(G_k;\Pi)
\]
via the Tannakian categories of \S \ref{sec:elliptic_case}. We refer back to \S \ref{sec:selmer_varieties} for the relationship between $H^1_{f,S}(G_k;\Pi)$ and $\Sel_{S,\Pi}(\Xc)$.

We have by Remark \ref{rem:elliptic_ext_S} that $H^1_{f,S}(G_k;M_{a,b}) = \Ext^1_{\Rep_{\Qp}^{\rm f, S}(G_k, E)}(\Qp(0),M_{a,b})$ and more generally that 
\[
H^1_{f,S}(G_k;\Pi) =
H^1(\Rep_{\Qp}^{\rm f, S}(G_k,E);\Pi),\footnote{Note that this equality does not require Conjecture \ref{conj:g_sg}, because $\Pi$ has negative weights. Furthermore, the semisimplicity condition in the definition of $\Rep_{\Qp}^{\rm sg}(G_k)$ follows from \cite{Faltings83}. C.f. Remark \ref{rem:g_sg}.}
\]
where $H^1(\Rep_{\Qp}^{\rm f, S}(G_k,E);\Pi)$ denotes the set of torsors under $\Pi$ in the Tannakian category $\Rep_{\Qp}^{\rm f, S}(G_k,E)$. In order to express this set as the group cohomology of a pro-algebraic group, we introduce fiber functors.

We recall there is a tensor functor $\Rep_{\Qp}^{\rm f, S}(G_k,E) \to \Rep_{\Qp}^{\rm ss}(G_k,E)$ sending a representation $V$ to its associated graded $\Gr^W_{\bullet} V$ for the weight filtration.

\begin{defn}\label{defn:tannakian_galois}

We define \[\pi_1^{\ME}(\Oc_{k,S}, E)\] to be the Tannakian fundamental group of the category
\[
\Rep_{\Qp}^{\rm f, S}(G_k, E),
\]
defined in Remark \ref{rem:elliptic_ext_S}, with respect to the de Rham fiber functor $V \mapsto V^{\dR} \colonequals \operatorname{D}_{\dR}{V}$.

We define \[\pi_1^{\ME}(\Oc_{k,S}, E)^{\Gr}\] to be the Tannakian fundamental group of the same category with respect to the graded de Rham fiber functor \[V \mapsto V^{\Gr\dR} \colonequals \operatorname{D}_{\dR}{\Gr^W_{\bullet} V}.\]
\end{defn}

The two fiber functors are isomorphic (in a manner compatible with associated graded) by \cite[IV.2.2.2]{SaavedraRivanoBook} and \cite[Main Theorem 1.2]{ZieglerGrFil15}. We describe the issue more in \S \ref{sec:graded_ungraded} below.

Note that both fiber functors are canonically isomorphic when restricted to the subcategory $\Rep_{\Qp}^{\rm ss}(G_k,E) \subseteq \Rep_{\Qp}^{\rm f, S}(G_k,E)$. We denote its Tannakian Galois group by $\pi_1(\Rep_{\Qp}^{\rm ss}(G_k,E))$, and it is isomorphic to $\glt$.

For a unipotent group $\Pi$ in the Tannakian category $\Rep_{\Qp}^{\rm f, S}(G_k, E)$, we denote by $\Pi^{\dR}$ the unipotent group over $\Qp$ with $\pi_1^{\ME}(\Oc_{k,S}, E)$-action associated to the Lie algebra $(\Lie{\Pi})^{\dR}$, with its induced Lie algebra structure.

We have
\[
\Ext^1_{\Rep_{\Qp}^{\rm f, S}(G_k, E)}(\Qp(0),M_{a,b}) = H^1(\pi_1^{\ME}(\Oc_{k,S}, E);{M_{a,b}}^{\dR}) = H^1(\pi_1^{\ME}(\Oc_{k,S}, E)^{\Gr};{M_{a,b}}^{\Gr\dR}),
\]
and
\[
H^1(\Rep_{\Qp}^{\rm f, S}(G_k,E);\Pi)
=
H^1(\pi_1^{\ME}(\Oc_{k,S}, E);{\Pi}^{\dR}) = H^1(\pi_1^{\ME}(\Oc_{k,S}, E)^{\Gr};{\Pi}^{\Gr\dR}).
\]

%\emph{From now on, we supress the notation $\dR$, since it will be clear from context which version we are using in any given situation.}

\subsection{Structure of the Graded Tannakian Galois Group}\label{sec:structure_tann}

The inclusion
\[
\Rep_{\Qp}^{\rm ss}(G_k,E) \hookrightarrow \Rep_{\Qp}^{\rm f, S}(G_k,E)
\]
induces a map
\[
\pi_1^{\ME}(\Oc_{k,S},E)^{\Gr} \to \pi_1(\Rep_{\Qp}^{\rm ss}(G_k,E)) \cong \glt
\]
that is surjective, with kernel the unipotent radical $U(\Oc_{k,S},E)^{\Gr}$ of $\pi_1^{\ME}(\Oc_{k,S},E)^{\Gr}$. We denote the Lie algebra of $U(\Oc_{k,S},E)^{\Gr}$ by $\nf(\Oc_{k,S},E)^{\Gr}$, its universal envelopping algebra by $\Uc(\Oc_{k,S},E)^{\Gr}$, and its coordinate ring by $A(\Oc_{k,S},E)^{\Gr}$.

The functor $\Rep_{\Qp}^{\rm f, S}(G_k,E) \to \Rep_{\Qp}^{\rm ss}(G_k,E)$ defined by $V \mapsto \Gr^W_{\bullet} V$ induces a canonical section 
\begin{equation}\label{eqn:canonical_section}
s \colon \pi_1(\Rep_{\Qp}^{\rm ss}(G_k,E)) \to \pi_1^{\ME}(\Oc_{k,S},E)^{\Gr}.
\end{equation} We therefore have a canonical semidirect product decomposition
\[
\pi_1^{\ME}(\Oc_{k,S}, E)^{\Gr} \cong \mathrm{GL}_2 \ltimes U(\Oc_{k,S},E)^{\Gr},
\]
inducing actions of $\glt$ on $U(\Oc_{k,S},E)^{\Gr}$, $\Pi^{\Gr\dR}$, and their associated algebras.

%\begin{rem}\label{rem:weight_grading}
%There is a canonical map $\Gm \to \pi_1^{\ME}(\Oc_{k,S},E)/U(\Oc_{k,S},E) \cong \mathrm{GL}_2$ corresponding to the weight-grading on any object (resp. Pro-object, Ind-object) of $\Rep_{\Qp}^{\rm ss}(G_k,E)$, such as $\Lie{\Pi}$ (resp. $\Uc \Pi$, $\nf(\Oc_{k,S},E)$, or $\Uc(\Oc_{k,S},E)$, $\Oc(\Pi)$ or $A(\Oc_{k,S},E)$).
%\end{rem}

The theory of such extensions is described in \S \ref{sec:background_red_unip}-\ref{sec:cohom_cocycles}. In particular, by Theorem \ref{thm:cohom}, we have the following:
\begin{cor}\label{cor:cohom}
We have a natural bijection
\[
H^1_{f,S}(G_k;\Pi) \cong Z^1(U(\Oc_{k,S},E)^{\Gr};\Pi^{\Gr\dR})^{\glt}.
\]
\end{cor}
\begin{proof}
This follows from Theorem \ref{thm:cohom}, with $G = \pi_1^{\ME}(\Oc_{k,S}, E)^{\Gr}$, $\Gb=\glt$, $U=U(\Oc_{k,S}, E)^{\Gr}$, and $\Pi=\Pi^{\Gr\dR}$ as in the previous sections.
\end{proof}

Note that by Proposition \ref{prop:lie_cocycles}, we also have
\[
H^1_{f,S}(G_k;\Pi) \cong Z^1(\nf(\Oc_{k,S},E)^{\Gr};\Lie{\Pi}^{\Gr\dR})^{\glt}.
\]

\subsection{Graded vs. Ungraded}\label{sec:graded_ungraded}

We similarly have a projection 
\[
\pi_1^{\ME}(\Oc_{k,S},E) \to \pi_1(\Rep_{\Qp}^{\rm ss}(G_k,E)) \cong \glt,
\]
with kernel denoted $U(\Oc_{k,S},E)$, but it is not canonically split. We denote the associated objects by $\nf(\Oc_{k,S},E)$, $\Uc(\Oc_{k,S},E)$, and $A(\Oc_{k,S},E)$, respectively.

The scheme
\[
\operatorname{\underline{Isom}}^{\otimes, \Gr W}(\dR, \Gr \dR)
\]
of isomorphisms from the de Rham fiber functor to the graded de Rham fiber functor inducing the identity on associated graded is a $U(\Oc_{k,S},E)$-$U(\Oc_{k,S},E)^{\Gr}$ bitorsor in the fpqc topology over $\Qp$ by \cite[Main Theorem 1.2]{ZieglerGrFil15}. As already mentioned, it is trivial by \cite[IV.2.2.2]{SaavedraRivanoBook}. In particular, the fundamental exact sequence
\[
1 \to U(\Oc_{k,S},E) \to \pi_1^{\ME}(\Oc_{k,S},E) \to \pi_1(\Rep_{\Qp}^{\rm ss}(G_k,E)) \to 1
\]
splits, although not canonically. In fact, we have

\begin{prop}\label{prop:splittings_isoms}
The set
\[
\operatorname{\underline{Isom}}^{\otimes, \Gr W}(\dR, \Gr \dR)
\]
is naturally in bijection with the set
\[
\operatorname{Sec}(\pi_1^{\ME}(\Oc_{k,S},E))
\]
of sections of the fundamental exact sequence. The $U(\Oc_{k,S},E)$-torsor structure on $\operatorname{\underline{Isom}}^{\otimes, \Gr W}(\dR, \Gr \dR)$ corresponds to the action of $U(\Oc_{k,S},E)$ on $\operatorname{Sec}(\pi_1^{\ME}(\Oc_{k,S},E))$ by conjugation.
\end{prop}

\begin{proof}
Any $\alpha \in \operatorname{\underline{Isom}}^{\otimes, \Gr W}(\dR, \Gr \dR)$ induces an isomorphism $\beta \colon \pi_1^{\ME}(\Oc_{k,S},E)^{\Gr} \to \pi_1^{\ME}(\Oc_{k,S},E)$ over $\pi_1(\Rep_{\Qp}^{\rm ss}(G_k,E))$ defined by $\beta(g) = \alpha^{-1} g \alpha$. The composition
\[
\beta \circ s \colon \pi_1(\Rep_{\Qp}^{\rm ss}(G_k,E)) \to \pi_1^{\ME}(\Oc_{k,S},E),
\]
with $s$ as in (\ref{eqn:canonical_section}), is an element of $\operatorname{Sec}(\pi_1^{\ME}(\Oc_{k,S},E))$. In other words, we have a map:
\[
\operatorname{\underline{Isom}}^{\otimes, \Gr W}(\dR, \Gr \dR) \to \operatorname{Sec}(\pi_1^{\ME}(\Oc_{k,S},E)).
\]

If we replace $\alpha$ by $\alpha \circ u^{-1}$ for $u \in U(\Oc_{k,S},E)$, we replace $\beta(g)$ by $u \beta(g) u^{-1}$. In particular, this map intertwines the (left) action of $U(\Oc_{k,S},E)$ on the bitorsor with its (left) conjugation action on $\operatorname{Sec}(\pi_1^{\ME}(\Oc_{k,S},E))$.

It therefore suffices to show that the latter is a torsor under $U(\Oc_{k,S},E)$. This is true because $\glt$ has trivial higher cohomology and because $U(\Oc_{k,S},E)^{\glt}=0$ (for any choice of splitting).
\end{proof}

\begin{rem}\label{rem:splittings_isom}
Given a section $\glt \cong \pi_1(\Rep_{\Qp}^{\rm ss}(G_k,E)) \to \pi_1^{\ME}(\Oc_{k,S},E)$, we may associate a point of $\operatorname{\underline{Isom}}^{\otimes, \Gr W}(\dR, \Gr \dR)$ as follows. For any object $M$ of $\Rep_{\Qp}^{\rm f, S}(G_k,E)$, we get from the section a $\glt$-action on $M^{\dR}$, which then induces an isomorphism
\[
M^{\Gr\dR} \xrightarrow{\sim} M^{\dR}
\]
sending $\operatorname{D}_{\dR}{\Gr^W_{w} V}$ to 
\[
\bigoplus_{a+2b=-w} (M^{\dR})^{a,b},
\]
where $(M^{\dR})^{a,b}$ denotes the $M_{a,b}$-isotypic component of $M^{\dR}$.
\end{rem}

\begin{rem}\label{rem:change_of_isom}
A choice of $\alpha \in \operatorname{\underline{Isom}}^{\otimes, \Gr W}(\dR, \Gr \dR)$ induces an isomorphism
\[
H^1(\pi_1^{\ME}(\Oc_{k,S},E);\Pi^{\dR}) \cong Z^1(U(\Oc_{k,S},E);\Pi^{\dR})^{\glt}.
\]

By Theorem \ref{thm:cohom}, this determines a section of the surjection
\[
Z^1(\pi_1^{\ME}(\Oc_{k,S},E);\Pi^{\dR}) \twoheadrightarrow H^1(\pi_1^{\ME}(\Oc_{k,S},E);\Pi^{\dR}),
\]
whose image is \[\ker(Z^1(\pi_1^{\ME}(\Oc_{k,S},E);\Pi^{\dR}) \to Z^1(\glt;\Pi^{\dR}))\] and whose projection to $Z^1(U(\Oc_{k,S},E);\Pi^{\dR})$ is $Z^1(U(\Oc_{k,S},E);\Pi^{\dR})^{\glt}$.

Let's see what happens if we change $\alpha$. If $c \in Z^1(U(\Oc_{k,S},E);\Pi^{\dR})$ is $\glt$-equivariant with respect to $\alpha$, one may check that
\[
v \mapsto u(c(u^{-1} v u))
\]
is $\glt$-equivariant with respect to $\alpha \circ u^{-1}$.
\end{rem}

%% I checked that that really is a cocycle. You don't have to twist the action on Pi or anything like that.

\begin{rem}\label{rem:extension_class}
Given $M \in \Rep_{\Qp}^{\rm f, S}(G_k, E)$ and an extension
\[
1 \to M \to E \to \Qp(0) \to 1
\]
representing an element of $\Ext^1_{\Rep_{\Qp}^{\rm f, S}(G_k, E)}(\Qp(0),M) =. H^1(\pi_1^{\ME}(\Oc_{k,S}, E);{M_{a,b}}^{\dR})$, we may write down a cocycle representing this cohomology class by choosing a lift $1_E \in E$ of $1 \in \Qp(0)$ and then considering the cocycle
\[
u \mapsto u(1_E)-1_E.
\]

It is easy to see that this cocycle is $\glt$-equivariant if and only if $1_E$ is $\glt$-invariant. Assuming that $M$ contains no subquotient isomorphic to $\Qp(0)$, there is a unique lift $1_E$ of $1$ that is $\glt$-equivariant. The notion of $\glt$-equivariance of course depends on the choice of $\alpha$, and this $1_E$ is precisely the lift corresponding to the splitting of the weight filtration determined by $\alpha$.
\end{rem}

We fix once and for all a point of 
\[
\operatorname{\underline{Isom}}^{\otimes, \Gr W}(\dR, \Gr \dR).
\]

%We assume furthermore that this point is compatible with the Hodge filtration, in that for every $M \in \Rep_{\Qp}^{\rm f, S}(G_k, E)$, the isomorphism
%\[
%M^{\dR} \cong M^{\Gr\dR}
%\]
%respects the Hodge filtration. As a consequence, any $\glt$-equivariant map must respect the Hodge filtration, a fact we use in \S \ref{sec:p-adic_periods}-\ref{sec:univ_cocyc}.
%% I finally got rid of this because I'm both 1) not sure this is necessary if I just arbitrarily extend the p-adic period map and 2) might conflict with my choice of explicit basis of Pi^{dR}

% ****
%%% maybe need to prove you can do this. See if Saedra-Rivano might help.
%%% this should mean that the Hodge filtration is the same as the once induced by the direct sum. For a simple extension of Q(0), I think this is the same as saying that the extension piece lives in F^0
%%% I think you can show that this implies that any GL_2-equivariant cocycle preserves Hodge filtration. Might want to state/prove this here or in the next section.
%%% Maybe I could consider a filtration on the path torsor, and maybe preserving the filtration is something like being in F^0 of the path torsor??
%%% Also wonder how to expression "preserves filtration" condition in terms of the lift

We thus fix an identification $U(\Oc_{k,S},E) \cong U(\Oc_{k,S},E)^{\Gr}$ and thus an identification
\[
H^1_{f,S}(G_k;\Pi) \cong Z^1(U(\Oc_{k,S},E);\Lie{\Pi}^{\dR})^{\glt}.
\]
We therefore may mostly ignore the distinction between $\dR$ and $\Gr\dR$, although we revisit it briefly in \S \ref{sec:p-adic_periods}.

\begin{rem}\label{rem:semisimple}
In fact, if $\Pi$ is semisimple, then $U(\Oc_{k,S},E)$ acts trivially on $\Pi^{\dR}$, so we get a $\glt$-action on $\Pi^{\dR}$. This induces an isomorphism
\[
\Pi^{\dR} \cong \Pi^{\Gr\dR},
\]
independent of choice of point of $\operatorname{\underline{Isom}}^{\otimes, \Gr W}(\dR, \Gr \dR)$. In the notation of Remark \ref{rem:change_of_isom}, $u(c(u^{-1}vu)) = c(v)$ in this case.
\end{rem}

\subsection{Structure of the Unipotent Radical}\label{sec:structure_unip}

We have
\[
\nf(\Oc_{k,S},E)^{ab} = U(\Oc_{k,S},E)^{ab} \cong \prod_{a,b}  H^1_{f,S}(G_k;M_{a,b})^{\vee} \otimes_{\Qp} {M_{a,b}}^{\dR} 
\]
as an object of $\Pro \Rep_{\Qp}^{\rm ss}(G_k,E)$.

Dually, we have a canonical isomorphism
\[
\ker(\Delta' \colon A(\Oc_{k,S},E) \to A(\Oc_{k,S},E) \otimes A(\Oc_{k,S},E)) \cong \bigoplus_{a,b}  H^1_{f,S}(G_k;M_{a,b}) \otimes_{\Qp} {M_{a,b}^{\vee}}^{\dR}
\]
in $\Ind \Rep_{\Qp}^{\rm ss}(G_k,E)$.

Let us briefly describe this isomorphism explicitly. Let $M_{a,b} \in \Irr(\glt)$, and let $c \in \Ext^1(\Qp,M_{a,b})$. Then $c$ is described by an extension
\[
0 \to M_{a,b} \to E_c \to \Qp \to 0.
\]
Choose a lift $1_E \in E_c^{\dR}$ of $1 \in \Qp$. Given $v \in (M_{a,b}^{\vee})^{\dR}$, let $p_v \colon {E_c}^{\dR} \to \Qp$ be the functional given by the projection defined by $1_E$ followed by $v \colon M_{a,b}^{\dR} \to \Qp$. Then the element $c \otimes v$ of $A(\Oc_{k,S},E)$ is the Tannakian matrix coefficient (\cite[\S 2.2]{BrownNotes17})
\[
[E_c,1_E,p_v].
\]

Letting $A(\Oc_{k,S},E)_{>0}$ denote the augmentation ideal, $A(\Oc_{k,S},E)_{>0} \cdot A(\Oc_{k,S},E)_{>0}$ is the space of \emph{decomposables}, and
\[
\nf(\Oc_{k,S},E)^{\vee} \colonequals A(\Oc_{k,S},E)_{>0}/A(\Oc_{k,S},E)_{>0} \cdot A(\Oc_{k,S},E)_{>0}
\]
is the Lie coalgebra. Then $\Delta'$ induces the cobracket on $\nf(\Oc_{k,S},E)^{\vee}$, and \[\ker(\Delta' \colon A(\Oc_{k,S},E) \to A(\Oc_{k,S},E) \otimes A(\Oc_{k,S},E)) \cong \bigoplus_{a,b}  H^1_{f,S}(G_k;M_{a,b}) \otimes_{\Qp} {M_{a,b}^{\vee}}^{\dR}\] is the kernel of the cobracket.

\subsection{A Free Unipotent Group}\label{sec:free_unip}

As $\glt$ is reductive, we may choose a $\glt$-equivariant splitting of the projection
\begin{equation}\label{eqn:splitting}
\nf(\Oc_{k,S},E) \twoheadrightarrow \nf(\Oc_{k,S},E)^{ab}.
\end{equation}
Let $W$ denote the image of $\nf(\Oc_{k,S},E)^{ab}$ under this splitting, and set
\[
\nf(W) \colonequals \FreeLie{W},
\]
and let $U(W)$, $\Uc W$, and $A(W)$ denote the associated pro-unipotent group, universal enveloping algebra, and coordinate ring, respectively.

We have a map
\[
\theta \colon \nf(W) \to \nf(\Oc_{k,S},E)
\]
of Lie algebra objects in $\Pro \Rep_{\Qp}^{\rm ss}(G_k,E)$ that is an isomorphism on abelianizations. In particular, $\theta$ is surjective.

If we knew the following conjecture, then we would know that $\theta$ is an isomorphism:

\begin{conj}[{\cite[Conjecture II.3.2.2]{FPR91}}]\label{conj:FPR3}
If $V \in \Rep_{\Qp}^{\rm g}(G_k)$, and $V''$ is a quotient of $V$, then
\[
H^1_{f,S}(G_k;V) \to H^1_{f,S}(G_k;V'')
\]
is surjective.
\end{conj}

Indeed, Conjecture \ref{conj:FPR3} would imply that $\Ext^i_{\Rep_{\Qp}^{\rm f, S}(G_k,E)}$ vanishes for $i \ge 2$, and hence that $U(\Oc_{k,S},E)$ is a free pro-unipotent group over $\Qp$. Then it follows that $\theta$ is an isomorphism.

We nonetheless have an embedding
\[
H^1_{f,S}(G_k;\Pi) \cong Z^1(U(\Oc_{k,S},E);\Pi)^{\glt} \hookrightarrow Z^1(U(W);\Pi)^{\glt},
\]
conjectured to be an isomorphism.

% Doesn't seem to belong here anymore. Was probably relevant when the sections were ordered differently.
%%We will spend the rest of \S \ref{sec:tannakian_selmer} describing $Z^1(U(W);\Pi)^{\glt}$ and $U(W)$ in more detail. We will discuss their relationship to $H^1_{f,S}(G_k;\Pi)$ and to non-abelian Chabauty's method in \S \ref{sec:selmer_localization}

As an associative algebra in $\Pro \Rep_{\Qp}^{\rm ss}(G_k,E)$, the universal enveloping algebra $\Uc W$ is the tensor algebra on $W$.

Dually, $A(W)$ becomes the free shuffle algebra on the dual object 
\[
W^{\vee} =
\bigoplus_{M \in \Irr{\glt}}  H^1_{f,S}(G_k;M) \otimes_{\Qp} {M^{\vee}}^{\dR}
=
\bigoplus_{a,b}  H^1_{f,S}(G_k;M_{a,b}) \otimes_{\Qp} {M_{a,b}^{\vee}}^{\dR}
\]
of $\Ind \Rep_{\Qp}^{\rm ss}(G_k,E)$, with coproduct given by deconcatenation and $\glt$-action given by its action on this vector space. Note that it also has the structure of a non-semisimple motive by \S \ref{sec:background_red_unip}.

\begin{rem}\label{rem:general_J_2}
The constructions and results of \S \ref{sec:structure_tann}-\ref{sec:free_unip} apply with $E$ replaced by any abelian variety, as in Remark \ref{rem:general_J_1}.
\end{rem}

\section{Localization Maps for Selmer Varieties}\label{sec:selmer_localization}

The goal of this section is to explicitly understand the map
\[
\log_{\mathrm{BK}} \circ \loc_{\Pi} \colon H^1_{f,S}(G_k;\Pi) \to \Pi/F^0\Pi
\]
via the description of $H^1_{f,S}(G_k;\Pi)$ in \S \ref{sec:tannakian_selmer}. The results of this section will allow us to find an element of $\Oc(\Pi/F^0)$ vanishing on the image of $\loc_{\Pi}$ by computing a function vanishing on the image of a more explicit map
\[
\ev_{\Pi,W}/F^0 \colon Z^1(U(W);\Pi)^{\glt} \times U(W) \to \Pi/F^0 \times U(W).
\]

The material of this section corresponds to \cite[\S 2.3-2.4]{PolGonI}.

\subsection{Universal Cocycle Evaluation Maps}\label{sec:univ_cocyc}

For any $\Pi$, we have a \emph{universal cocycle evaluation map} (c.f. \cite[Definition 2.20]{PolGonI})
\[
\ev_{\Pi} \colon Z^1(U(\Oc_{k,S},E);\Pi)^{\glt} \times U(\Oc_{k,S},E) \to \Pi \times U(\Oc_{k,S},E).
\]
defined by
\[
(c,u) \mapsto (c(u), u).
\]

%Its projection to $\Pi/F^0 \times U(\Oc_{k,S},E)$ descends to a map
%\[
%\ev_{\Pi} \colon Z^1(U(\Oc_{k,S},E);\Pi)^{\glt} \times U(\Oc_{k,S},E)/F^0 \to \Pi/F^0 \times %U(\Oc_{k,S},E)/F^0
%\]
%because $c(u) \hspace{-3mm} \mod{F^0}$ depends only on $u \hspace{-3mm} \mod{F^0}$.

Its pullback along $U(W) \twoheadrightarrow U(\Oc_{k,S},E)$ factors through a map
\[
Z^1(U(\Oc_{k,S},E);\Pi)^{\glt} \times U(W) \to Z^1(U(W);\Pi)^{\glt} \times U(W) \xrightarrow{\ev_{\Pi,W}} \Pi \times U(W).
\]

Our goal in \S \ref{sec:geometric_step} will be to compute a function on $\Pi/F^0 \times U(W)$ vanishing on the image of the composition $\ev_{\Pi,W}/F^0$ of $\ev_{\Pi,W}$ with projection from $\Pi$ to $\Pi/F^0$:
\[
\ev_{\Pi,W}/F^0 \colon Z^1(U(W);\Pi)^{\glt} \times U(W) \xrightarrow{\ev_{\Pi,W}} \Pi \times U(W) \to \Pi/F^0 \times U(W).
\]

In \S \ref{sec:p-adic_periods}-\ref{sec:use_p-adic_periods}, we show how such a function on $\Pi/F^0 \times U(W)$ specializes to an element of the Chabauty-Kim ideal. First, we show how to descend $\ev_{\Pi,W}$ from $U(W)$ to a scheme of finite type over $\Qp$.

\subsection{Quotients of the Unipotent Radical}\label{sec:quot_unip_radical}

The algebra $A(W)$ has a weight filtration. However, it is not necessarily finite-dimensional in each degree, even assuming Conjecture \ref{conj:BK}. That's because given $w$, there are infinitely many pairs $(a,b) \in \Zb_{\ge 0} \times \Zb$ with $w = -a -2b$.

Let $W^{\aeff}$ the quotient of $W$ corresponding to
\[
\prod_{M \in \Irr^{\aeff}{\glt}}  H^1_{f,S}(G_k;M)^{\vee} \otimes_{\Qp} {M}^{\dR}.
\]
Then the quotient $\nf(W)^{\aeff} \colonequals \FreeLie{W^{\aeff}}$ of $\nf(W)$ is finite-dimensional in each degree. More generally, for a subset
\[
I \subseteq \Irr{\glt},
\]
we define
\[
W^I \colonequals \prod_{M \in I}  H^1_{f,S}(G_k;M)^{\vee} \otimes_{\Qp} {M}^{\dR},
\]
\[
\nf(W)^{I} \colonequals \FreeLie{W^I}
\]
as a quotient of $\nf(W)$, and
\begin{align*}
U(W)^I\\
A(W)^I
\end{align*}
the corresponding pro-unipotent group and coordinate ring, respectively. Assuming $I \subseteq \Irr^{\aeff}{\glt}$,\footnote{More generally, as long as $I$ has only representations of negative weight and contains finitely many of each given weight.} $\nf(W)^{I}$ and $A(W)^I$ are finite-dimensional in each degree and are strictly-negatively and positively graded, respectively.

\begin{prop}\label{prop:c_kernel}
If $I \subseteq \Irr{\glt}$ contains all graded pieces of $\Pi$, and the action of $U(W)$ on $\Pi$ factors through $U(W)^I$, then
\[
Z^1(U(W)^I;\Pi)^{\glt} = Z^1(U(W);\Pi)^{\glt}
\]
\end{prop}
\begin{proof}
Set
\[
W' \colonequals \Ker(W \twoheadrightarrow W^I).
\]

Then the kernel of
\[
\nf(W) \to \nf(W)^{I}
\]
corresponds under $\theta$ to the Lie ideal $\nf(W)'$ generated by $W'$. This Lie ideal corresponds to the normal subgroup scheme
\[
U(W)' = \Ker(U(W) \to U(W)').
\]

Let $c \in Z^1(\nf(W);\Lie{\Pi})^{\glt} = Z^1(U(W);\Pi)^{\glt}$. Note that $c$ vanishes on $W'$, since $\Hom_{\glt}(W',\Lie{\Pi}) = 0$.

It then suffices to show that
\[
\Ker{c} \cap \nf(W)'
\]
is a Lie ideal in $\nf(W)$. For this, suppose $u \in \nf(W)$ and $w \in \Ker{c} \cap \nf(W)'$. Then
\[
c([u,w]) = [c(u),c(w)] + u(c(w)) - w(c(v)) = [c(u),0] + u(0) - w(c(v)) = 0
\]
because $w \in \nf(W)'$, which acts trivially on $\Lie{\Pi}$.

It follows that
\[
\Ker{c} = \nf(W)',
\]
so that $c \in Z^1(U(W)^I;\Pi)^{\glt} \subseteq Z^1(U(W);\Pi)^{\glt}$. As $c$ is arbitrary, we are done.
\end{proof}

\begin{rem}
Given $\Pi$, one may ensure the hypothesis of Proposition \ref{prop:c_kernel} by taking $I$ to include the set of irreducible components of $\End(\Lie{\Pi})$ as a $\glt$-representation.
\end{rem}

Suppose that the graded pieces of $\Pi$ are contained in $I \subseteq \Irr{\glt}$ and that the action of $U(W)$ on $\Pi$ factors through $U(W)^I$, as in the hypotheses of Proposition \ref{prop:c_kernel}. Then $\ev_{\Pi,W}$ is just the pullback along $U(W) \to U(W)^I$ of a map
\[
\ev_{\Pi,W}^I \colon Z^1(U(W)^I;\Pi)^{\glt} \times U(W)^I \to  \Pi \times U(W)^I.
\]

In particular, an element of $\Oc(\Pi/F^0 \times U(W)^I) = \Oc(\Pi/F^0) \otimes A(W)^{I}$ vanishing on the image of $\ev_{\Pi,W}^I/F^0$ also vanishes on the image of $\ev_{\Pi,W}/F^0$. In \S \ref{sec:geometric_step}, we will be computing a function vanishing on the image of $\ev_{\Pi,W}^I/F^0$ for $X=E'$, $\Pi=U_3$, and appropriately chosen $I$.

\subsection{\texorpdfstring{$p$}{p}-adic Periods and Localization}\label{sec:p-adic_periods}

In this section, we describe a $p$-adic period map contained in forthcoming work of the author and I. Dan-Cohen. This is a $\Qp$-algebra homomorphism $A(\Oc_{k,S},E) \to \Qp$ for $\pf \in \Spec{\Oc_{k,S}}$ with residue characteristic $p$ and $\Oc_{\pf} \cong \Zp$, compatible with the non-abelian Bloch-Kato map of \cite{kim09}.

Such a map for mixed Tate motives was defined by \cite{ChatUnv13} and used in \cite[2.4.1]{PolGonI}. A different $p$-adic period map for mixed Tate motives was defined in \cite[5.28]{DelGon05} (c.f. also \cite[\S 3.2]{YamashitaBounds} and \cite[3.4.3]{BrownIntegral}). %The latter definition has an obvious extension to general mixed motivic periods, but as shown in \cite[\S 3.1]{BrownKimHodge}, it is not compatible with the Bloch-Kato map, even in the case of mixed Tate motives.

\begin{rem}\label{rem:period_not_necessary}
This map is not logically necessary for the definition of the universal cocycle evaluation map in \S \ref{sec:univ_cocyc} or the calculations in \S \ref{sec:geometric_step}. However, the period map motivates these constructions and calculations because it implies that the function derived at the end of \S \ref{sec:geometric_step} specializes to an element of the Chabauty-Kim ideal.\end{rem}

In forthcoming work with I. Dan-Cohen, we prove the following:

\begin{thm}\label{thm:p-adic period}

For $\pf \in \Spec{\Oc_{k,S}}$ with residue characteristic $p$ and $\Oc_{\pf} \cong \Zp$, there is a point
\[
\per_{\pf} \colon \Spec{\Qp} \to U(\Oc_{k,S},E),
\]
satisfying the following property:

Let $\Pi$ be a unipotent group with negative-weight action of $\pi_1^{\ME}(\Oc_{k,S}, E)$ and
\[
c \in Z^1(U(\Oc_{k,S},E),\Pi)^{\glt}.
\]

%Note that by $\glt$-equivariance, $c$ respects the Hodge filtration and therefore induces a map
%\[
%c/F^0 \colon U(\Oc_{k,S},E)/F^0 \to \Pi/F^0.
%\]
Then the composition
\[
c/F^0 \circ \per_{\pf} \colon \Spec{\Qp} \xrightarrow{\per_{\pf}} U(\Oc_{k,S},E) \xrightarrow{c} \Pi \twoheadrightarrow \Pi/F^0
\]
is $\log_{\mathrm{BK}}(\loc_{\Pi}(c)) \in \Pi/F^0(\Qp)$.
\end{thm}

\begin{rem}\label{rem:depends_on_alpha}
The map $\mathrm{per}_{\pf}$ depends on the choice of element of $\operatorname{\underline{Isom}}^{\otimes, \Gr W}(\dR, \Gr \dR)$ in \S \ref{sec:graded_ungraded}, although we omit this from the notation. C.f. Remark \ref{rem:change_of_isom}.
\end{rem}

\begin{rem}\label{rem:general_J_3}
The result Theorem \ref{thm:p-adic period} (as well as the constructions of \S \ref{sec:selmer_localization}) in fact apply to more general categories of Galois representations as in Remarks \ref{rem:general_J_1} and \ref{rem:general_J_2}.
\end{rem}

%\begin{notat}\label{notat:F0_coord}
%We set $A(\Oc_{k,S},E)^{F^0} \colonequals \Oc(U(\Oc_{k,S},E)/F^0)$ and $A(W)^{F^0} \colonequals \Oc(U(W)/F^0)$. When we define $U(W)^I$ in \S \ref{sec:quot_unip_radical}, we we use $A(W)^{I F^0}$ to refer to $\Oc(U(W)^I/F^0)$.
%\end{notat}

\subsection{\texorpdfstring{$p$}{p}-adic Periods and Universal Cocycle Evaluation}\label{sec:use_p-adic_periods}

Since the map $U(W) \twoheadrightarrow U(\Oc_{k,S},E)$ is a surjection of vector spaces, we may lift $\per_{\pf} \in U(\Oc_{k,S},E)$ arbitrarily to an element of $U(W)$. The choice of lift will not matter, so we denote it, by abuse of notation, by $\per_{\pf}$ as well.

We then have the following diagram of schemes over $\Qp$:
\[
\xymatrix{
H^1_{f,S}(G_k;\Pi) \ar[r]_-{\log_{\mathrm{BK}} \circ \loc_{\Pi}} \ar@{=}[d] & \Pi/F^0 \ar@{=}[d]\\
Z^1(U(\Oc_{k,S};E);\Pi)^{\glt} \ar[r]_-{\log_{\mathrm{BK}} \circ \loc_{\Pi}} \ar@{=}[d] & \Pi/F^0 \ar@{=}[d]\\
Z^1(U(\Oc_{k,S};E);\Pi)^{\glt} \times \Spec{\Qp} \ar[r] \ar[d]_-{\per_{\pf}} & \Pi/F^0 \times \Spec{\Qp} \ar[d]_-{\per_{\pf}}\\
Z^1(U(\Oc_{k,S},E);\Pi)^{\glt} \times U(W)  \ar[r] \ar@{->>}[d] & \Pi/F^0 \times U(W) \ar@{->>}[d]
\\
Z^1(U(\Oc_{k,S};E);\Pi)^{\glt} \times U(\Oc_{k,S},E) \ar[r]^-{\ev_{\Pi}} & \Pi/F^0 \times U(\Oc_{k,S},E)
}
\]

The commutativity of the diagram follows by Theorem \ref{thm:p-adic period}. More precisely, Theorem \ref{thm:p-adic period} ensures commutativity of
\[
\xymatrix{
Z^1(U(\Oc_{k,S};E);\Pi)^{\glt} \times \Spec{\Qp} \ar[r] \ar[d]_-{\per_{\pf}} & \Pi/F^0 \times \Spec{\Qp} \ar[d]_-{\per_{\pf}}\\
Z^1(U(\Oc_{k,S};E);\Pi)^{\glt} \times U(\Oc_{k,S},E) \ar[r]^-{\ev_{\Pi}} & \Pi/F^0 \times U(\Oc_{k,S},E)
},
\]
and the former diagram is commutative because its bottom square is Cartesian.

%\ar@{=}[]
%\ar@{^{(}->}[]
%\ar@{->>}[]

It follows from the definition of $\ev_{\Pi,W}$ that if $f \in \Oc(\Pi/F^0 \times U(W)) = \Oc(\Pi/F^0) \otimes A(W)$ vanishes on the image of $\ev_{\Pi,W}$, then $f$ also vanishes on the image of $\ev_{\Pi} \times_{U(\Oc_{k,S},E)} U(W)$. This in turn implies, by the diagram above, that
\[
\id_{\Oc(\Pi/F^0)} \otimes \per_{\pf}(f) \in \Oc(\Pi/F^0) \otimes_{\Qp} \Qp = \Oc(\Pi/F^0)
\]
vanishes on the image of $\log_{\mathrm{BK}} \circ \loc_{\Pi}$.

\section{The Level \texorpdfstring{$3$}{3} Quotient for a Punctured Elliptic Curve}\label{sec:level_3}

We use the notation of \S \ref{sec:punctured_ell_curve}, where $E$ is an elliptic curve, $X=E' = E \setminus \{O\}$, and $U=U(X)$. We set $\Ec=\overline{\Xc}$ and $\Ec'=\Xc$ to be the affine and projective curves, respectively, over $\Oc_{k,S}$ given by the minimal Weierstrass model of $E$. We suppose $E$ is also given by a Weierstrass equation $y^2=x^3+ax+b$. We also set $\Pi=U_3(E')$, and we suppose from now on that $|S|=1$.

Finally, we set
\begin{align*}
    e_0 \colonequals \frac{dx}{y}\\
    e_1 \colonequals \frac{xdx}{y},
\end{align*}

and we view words in the set $\{e_0,e_1\}$ as elements of $\Oc(\Pi)$.

\subsection{Coordinates for \texorpdfstring{$\Pi$}{Pi}}\label{sec:fund_grp_coords}

We now write down coordinates for $\Oc(U)$ and $\Oc(\Pi)$. The coordinate ring $\Oc(U)$ is filtered by weight, and $\Oc(\Pi)$ is the subalgebra generated by $W_3 \Oc(U)$. Furthermore, the choice of differential forms $e_0,e_1$ splits the (motivic) weight filtration. We find bases in each degree. In line with the notation of \cite[\S 3.1]{PolGonI}, we view $\Oc(U)$ as the shuffle algebra in words $e_0$ and $e_1$.

A basis of $\Oc(U)$ in degree $k \ge 0$ is given by the set of words of length $k$ the letters $e_0,e_1$. In degree $0$, a basis is \[\{1\},\] corresponding to the empty word. In degree $1$, a basis is
\[
\{e_0,e_1\}.
\]
In degree $2$, a basis is
\[
\{e_0^2,e_0e_1,e_1e_0,e_1^2\}.
\]
In degree $3$, a basis is 
\[
\{e_0^3, e_0^2 e_1, e_0 e_1 e_0, e_0 e_1^2, e_1 e_0^2, e_1 e_0 e_1, e_1^2 e_0, e_1^3\}.
\]

\subsection{Local Selmer Variety}\label{sec:local_selmer_coords}

We wish to explicitly describe $\Oc(\Pi/F^0 \Pi)$ in the bases of \S \ref{sec:fund_grp_coords}.

We have $[\Pi]=[U/U^4] = [M_{1,0}]+[M_{0,1}]+[M_{1,1}]$. Therefore,
\[
l(\Pi) = l_{1,0}+l_{0,1}+l_{1,1} = 1 + 1 + 2 = 4.
\]
Thus $\Oc(\Pi/F^0 \Pi)$ is a polynomial ring in $4$ variables.

We refer to results of \cite{Beacom} for the coordinate ring. In the notation of loc.cit., we have $\alpha_0 = e_0$, $\alpha_1 = e_1$, $F=\frac{y}{x}$, and $\lambda = -2$. Then by \cite[Proposition 5.4]{Beacom}, $\Oc(\Pi/F^0 \Pi)$ is generated as a polynomial ring by the four elements
\begin{align*}
J_1 \coloneqq e_0\\
J_2 \coloneqq e_0 e_1 \\
J_3 \coloneqq e_0 e_1 e_0\\
J_4 \coloneqq e_0 e_1 e_1 + 2 e_1
\end{align*}

Upon choosing a place $\pf$ with $k_{\pf} \cong \Qp$ and basepoint $b \in \Ec'(\Oc_{\pf})$, each $J_i$ determines a function
\[
J_i \colon \Ec'(\Oc_{\pf}) \to \Qp.
\]

At the same time, we have $d(\Pi)=d_{1,0}+d_{0,1}+d_{1,1}=r+|S|+1=3$. We therefore expect to have a nontrivial element of the Chabauty-Kim ideal in degree $4$.

We will prove the following in \S \ref{sec:geometric_step}:

\begin{thm}\label{thm:CK_ideal_element_form}
Let $\alpha_1,\cdots,\alpha_N$ be as in \S \ref{sec:bad_outside_S}. If $E$ is an elliptic curve over $\Qb$ of $p$-Selmer rank $1$ for which Conjecture \ref{conj:BK} holds for $h^1(E)$, then the ideal of functions vanishing on the image of $\loc_{\pf} \colon H^1_{f,S}(G_k;\Pi)_{\alpha_i} \to \Pi/F^0 \Pi$ contains an element of the form
\[
c_1 J_4 + c_2 J_3 + c_3 J_1 J_2 + c_4 J_1^3 + c_5 J_1,
\]
with $c_i \in \Qp$ arising as periods of elements of $\Oc(W)$, not all of which are zero.
\end{thm}

From this, if one has at least five elements of $\Ec'(\Oc_k[1/S])$ mapping to $H^1_{f,S}(G_k;\Pi)_{\alpha_i}$, then one may in theory determine the coefficients $c_i$ by computing Coleman integrals. In \S \ref{sec:explicit_examples}, we review two examples where this is the case.

Theorem \ref{thm:CK_ideal_element_form} follows from a procedure analogous to the \emph{geometric step} of \cite[\S 4.2]{PolGonI}. As in loc.cit., we prove more, writing the coefficients $c_i$ as $p$-adic periods of specific elements of $\Oc(W)$.

\begin{rem}\label{rem:explicit_periods}
By applying Theorem \ref{thm:p-adic period} to appropriately chosen $\Pi$ developing the theory of motivic periods more carefully (e.g. as discussed in \S \ref{sec:future_work}), one will likely be able to determine the $c_i$ even when $\Ec'(\Oc_k[1/S])$ is not large enough. This is analogous to \cite[\S 4.3]{PolGonI}.
\end{rem}

\begin{rem}\label{rem:betts}
Work in preparation by A. Betts on weight filtrations of Selmer varieties is expected to prove the weaker statement that this ideal contains a nonzero element of the form
\[
c_1 J_4 + c_2 J_3 + c_3 J_1 J_2 + c_4 J_1^3 + c_5 J_1 + c_6 J_2 + c_7 J_1^2 + c_8.
\]
\end{rem}

In fact, we prove a more refined statement, in that we may write the coefficients in terms of the periods of certain elements of $U(\Oc_{k,S},E)$.

In particular, we may find such an equation as long as we have at least $5$ sufficiently independent $\Zb[1/2]$-points in each fiber of the map
\[
\Ec'(\Zb[1/2]) \to \prod_{v \in T_0 \setminus S} \kappa_v(\Xc(\Oc_{v})).
\]

For the elliptic curve with Cremona label ``128a2'', this map is trivial (i.e., $N=1$), and there are thirteen $\Zb[1/2]$-points. For the elliptic curve with Cremona label ``102a1'', we have $N=2$, and the two fibers contain nine and ten $\Zb[1/2]$-points, respectively.

\subsection{The Level \texorpdfstring{$3$}{3} Quotient of \texorpdfstring{$U(E')$}{U(E')}}\label{sec:U3_semisimple}

We show the motive
\[
\Lie{\Pi}
\]
is semisimple. Equivalently, the action of $U(\Oc_{k,S},E)$ on $\Pi$ is trivial. By Corollary \ref{cor:cohom}, this tells us that
\[
H^1_{f,S}(G_k;\Pi) \cong Z^1(U(\Oc_{k,S},E);\Pi)^{\glt} = \Hom(U(\Oc_{k,S},E),\Pi)^{\glt}.
\]

Furthermore, since $\Pi$ has graded pieces contained in $I=\{M_{1,0},M_{0,1},M_{1,1}\}$, this tells us that
\[
H^1_{f,S}(G_k;\Pi) = \Hom(U(\Oc_{k,S},E)^I,\Pi)^{\glt}.
\]

This is the specialization of a universal mixed elliptic motive in the sense of \cite{HainMatsumoto}. It is the weight $\ge -3$ quotient of the extension
\[
0 \to (\Sym \Hb(1))(1) \to \mathbf{Pol}_2^{\rm ell} \to \Hb(1) \to 0
\]
of \cite[\S 14]{HainMatsumoto}, where $\Hb$ is the first cohomology of the universal elliptic curve.

In the notation of \cite{HainMatsumoto}, the weight $\ge -3$ part of the extension shows up in 
\[\Ext^1(\Hb(1),M_{0,1}) = \Ext^1(\Qb(0), M_{0,1} \otimes \Hb)\]
and
\[\Ext^1(\Hb(1),M_{1,1}) = \Ext^1(\Qb(0), M_{1,1} \otimes \Hb),\]
where we take $\Ext$ in the category $\mathbf{MEM}_*$ for $*=2,\vec{1},1$.
%(This is because $\Ext^1(V,W) = \Ext^1(1, \Hom(V,W))$, a fact used in \cite[(14.2)]{HainMatsumoto}.)

Now $M_{0,1} \otimes \Hb$ is $M_{1,0}$, so the class vanishes for $*=1,\vec{1}$ by \cite[Theorem 15.1]{HainMatsumoto}, because it corresponds to $m=1, r=1$.

Then $M_{1,1} \otimes \Hb = M_{1,0} \otimes M_{1,0} = M_{2,0} + M_{0,1}$. That class also vanishes by \cite[Theorem 15.1]{HainMatsumoto} (for $M_{2,0}$ it's $(m,r)=(2,2)$ and thus vanishes for all $*$, and for $M_{0,1}$ it's $(m,r) = (0,1)$ so vanishes for $*=1, 2$).

\section{Geometric Step}\label{sec:geometric_step}

Our goal for this section is to find an element of $\Oc(\Pi/F^0) \otimes A(W)^{IF^0}$ that vanishes on the image of $\ev_{\Pi,W}^I/F^0$ (and therefore of $\ev_{\Pi,W}/F^0$, as described in \S \ref{sec:quot_unip_radical}) for $\Pi = U_3(X)$ and $X=E'$. This is analogous to the ``geometric step'' of \cite[\S 1.3.4, \S 4.2]{PolGonI}. For such $\Pi$, we may take $I=\{M_{1,0},M_{0,1},M_{1,1}\}$ by \S \ref{sec:U3_semisimple}. The result will prove Theorem \ref{thm:CK_ideal_element_form}.

%In fact, we compute with the base-change of $\ev_{\Pi,W}^I$ along the projection $U(W)^I \twoheadrightarrow U(W)^I/F^0$. The fact that it is given by base-change will allow us to find a function whose coefficients are in $A(W)^{IF^0} \subseteq A(W)^I$, and we may verify at the end that this is the case.

When base-changed to the fraction field $\Kc(W)^I$ of $U(W)^I$, the map $\ev_{\Pi,W}^I/F^0$ becomes a map of finite-dimensional varieties over a field. Since the left side is three-dimensional, and the right side is four-dimensional, we may use elimination theory to find a nonzero function vanishing on the image of $\ev_{\Pi,W}^I/F^0$. One may then clear denominators to ensure that the coefficients are in $A(W)^I$ rather than $\Kc(W)^I$.

We perform this elimination theory in \S \ref{sec:elimination}. Before that, we define coordinates on $U(W)^I$ in \S \ref{sec:basis_A}, and we write $\ev_{\Pi,W}^I$ in coordinates in \S \ref{sec:coord_cocyc}.

\subsection{Basis for \texorpdfstring{$A$}{A}}\label{sec:basis_A}

For simplicity of notation, we omit the superscript $\dR$ when talking about the realization of a motive. So we may think of $M_{1,0}^{\vee}$ as a $2$-dimensional $\Qp$-vector space with basis $\{e_0,e_1\}$.

We let $A=A(W)^I$. For an integer $n$, we let $A_n$ denote the degree-$n$ part of $A$ for the weight-grading coming from the $\glt$-action. We let
\[
\pr_n \colon A \to A
\]
denote projection onto $A_n$, and
\[
\pr_{m,n} = \pr_m \otimes \pr_n \colon A \otimes A \to A \otimes A
\]
denote projection onto $A_m \otimes A_n$.

We recall from \S \ref{sec:free_unip} that the choice of splitting (\ref{eqn:splitting}) determines an isomorphism of Hopf algebras with $\glt$-action over $\Qp$ between $A$ and the free shuffle algebra on the vector space
\[
(W^I)^{\vee} = \mathrm{Ext}^1(\Qp,M_{1,0}) \otimes M_{1,0}^{\vee}
\oplus
\mathrm{Ext}^1(\Qp,M_{0,1}) \otimes M_{0,1}^{\vee}
\oplus
\mathrm{Ext}^1(\Qp,M_{1,1}) \otimes M_{1,1}^{\vee}.
\]

Equivalently, $\Uc W^I$ is the tensor algebra on $W^I$, and every element of $W^I$ is primitive for the coproduct. We first describe a basis of $W^I$, which is $2+1+2=5$-dimensional. We let $\pi$ denote a generator of $\mathrm{Ext}^1(\Qp,M_{1,0})^{\vee}$ corresponding to a chosen generator of $E(\Qb)$ (to be fixed in each case), $\tau_2$ a generator of $\mathrm{Ext}^1(\Qp,M_{0,1})^{\vee}$ corresponding to $\log{2}$, and $\sigma$ a generator of $\mathrm{Ext}^1(\Qp,M_{1,1})^{\vee}$. We also let $\{e_0^{\vee},e_1^{\vee}\}$ denote a basis of $M_{1,0}$ dual to $\{e_0,e_1\}$. We set:
\begin{align*}
    \pi_i \colonequals \pi \otimes e_i^{\vee}\\
    \tau = \tau_{2} \otimes (e_0^{\vee} e_1^{\vee} - e_1^{\vee} e_0^{\vee})\\
    \sigma_i = \sigma \otimes e_i^{\vee}(e_0^{\vee} e_1^{\vee} - e_1^{\vee} e_0^{\vee}),
\end{align*}
so that $\{\pi_0,\pi_1,\tau,\sigma_0,\sigma_1\}$ is a basis of $W^I$. Note that $\pi_0,\pi_1$ have weight $1$, $\tau$ has weight $2$, and $\sigma_0,\sigma_1$ have weight $3$.

This basis determines a basis of $\Uc W^I$, which is simply the set of words in the set $\{\pi_0,\pi_1,\tau,\sigma_0,\sigma_1\}$. We consider a basis of $A$ dual to this basis of $\Uc W^I$. For a word $w$, we let $f_w$ denote the dual basis element of $A$. While the notation of subscripted $f$'s may be cumbersome, it helps distinguish between the shuffle product and concatenation product in $A$ (compare \cite[\S 4.1]{PolGonI}).

%%**slightly confused why it couldn't be  (e_0 \wedge e_1) e_i. Probably it doesn't actually matter, since the GL_2-action is the same either way

We may take $\{1\}$ as a basis for $A_0$ and $\{f_{\pi_0},f_{\pi_1}\}$ as a basis for $A_1$.

A basis for $A_2$ is given by $\{f_{\pi_0^2}, f_{\pi_0 \pi_1}, f_{\pi_1 \pi_0}, f_{\pi_1^2}, f_{\tau}\}$.

A basis for $A_3$ is given by
\[
\{f_{\sigma_0}, f_{\sigma_1}, f_{\pi_0^3}, f_{\pi_0^2 \pi_1}, f_{\pi_0 \pi_1 \pi_0}, f_{\pi_0 \pi_1^2}, f_{\pi_1 \pi_0^2}, f_{\pi_1 \pi_0 \pi_1}, f_{\pi_1^2 \pi_0}, f_{\pi_1^3}, f_{\pi_0 \tau}, f_{\pi_1 \tau}, f_{\tau \pi_0}, f_{\tau \pi_1}\},
\]
and it is $2+8+4 = 14$-dimensional.

\subsection{Coordinates on the Space of Cocycles}\label{sec:coord_cocyc}

Let $c \in Z^1(U(W)^I,U_3)^{\mathrm{GL}_2}$. For a word $\lambda$ in $\{e_0,e_1\}$ and $w$ in $\{\pi_0,\pi_1,\tau,\sigma_0,\sigma_1\}$, we define $\phi_{\lambda}^w(c)$ so that
\[
c^{\sharp}(\lambda) = \sum_{w} \phi_{\lambda}^w(c) f_{w}.
\]

Analogous to \cite[4.1.1]{PolGonI}, we set
\[
w_1 \colonequals \phi^{\pi_0}_{e_0}
\]
\[
w_2 \colonequals \phi^{\tau}_{e_0 e_1}
\]
\[
w_3 \colonequals  \phi^{\sigma_0}_{e_0 e_1 e_0}.
\]

Note that $c$ respects the $\glt$-action and therefore the weight grading. It follows that $\phi^w_{\lambda}=0$ unless $w$ and $\lambda$ have the same weight. We let $\Sigma_n$ denote the set of words of weight $n$.

Then the $\phi$'s generate $\Oc(Z^1(U(W)^I;U_3)^{\glt})$. We explicitly show that $\Oc(Z^1(U(W)^I;U_3)^{\glt}) = \Qp[w_1,w_2,w_3]$; i.e., for each word $\lambda$ of length at most $3$ in $e_0$ and $e_1$, we want to write
\[
c^{\#}(\lambda) \in A
\]
in terms of $w_1(c)$, $w_2(c)$, and $w_3(c)$. From now on, we write $w_1,w_2,w_3$ instead of $w_1(c),w_2(c),w_3(c)$, with the understanding that they are scalar functions of $c$.

This also determines the universal cocycle evaluation map
\[
\ev_{\Pi,W} \colon Z^1(U(W);\Pi)^{\glt} \times U(W) \to \Pi \times U(W)
\]
in that for $\lambda \in \Oc(\Pi)$, we have
\[
\ev_{\Pi,W}^{\#}(\lambda) = \sum_w \phi_{\lambda}^w f_w \in \Oc(Z^1(U(W);\Pi)^{\glt} \times U(W)).
\]

\begin{defn}
Under the basis of $A$ chosen in \S \ref{sec:basis_A}, we let
\[
\pr_{\pi}
\]
denote projection onto the vector subspace spanned by words in the set $\{\pi_0,\pi_1\}$.
\end{defn}

We now prove an analogue of \cite[Proposition 3.10]{PolGonI}.

%%%% might be able to remove condition on w

\begin{prop}\label{prop:phi_description}
For $\lambda$ and $w$ as above, we have
\begin{align*}
\phi_{\lambda e_i}^{w \pi_j}  = \phi_{e_i \lambda}^{\pi_j w} = \delta_{ij} w_1 \phi_{\lambda}^{w}.
\end{align*}
\end{prop}
\begin{proof}

Suppose $w$ and $\lambda$ have weight $n$ (otherwise $\phi_{e_i \lambda}^{\pi_j w}=0$). Because $c$ respects the coproduct (a result of the semisimplicity of $\Pi$) and commutes with weight projections $\pr_m$, we have
\begin{eqnarray*}
c^{\#}(e_i) \otimes c^{\#}(\lambda) &=& (c^{\#} \otimes c^{\#})(e_i \otimes \lambda)\\
&=& (c^{\#} \otimes c^{\#})(\pr_{1,n} \Delta' e_i \lambda)\\
&=& \pr_{1,n} \Delta' c^{\#} (e_i \lambda)\\
&=& \pr_{1,n} \Delta' \sum_{w' \in \Sigma_{n+1}} \phi_{e_i \lambda}^{w'} f_{w'}\\
&=& \sum_{w' \in \Sigma_{n+1}} \phi_{e_i \lambda}^{w'} \pr_{1,n} \Delta' f_{w'}.
\end{eqnarray*}

By $\glt$-equivariance and Schur's lemma, we have
\[
\phi^{\pi_i}_{e_j} = \delta_{ij} w_1,
\]
so this becomes
\[
(w_1 f_{\pi_i}) \otimes (\sum_{w'' \in \Sigma_n} \phi_{\lambda}^{w''} f_{w''})
=
\sum_{w'' \in \Sigma_n} w_1 \phi_{\lambda}^{w''} f_{\pi_i} \otimes f_{w''}
=
\sum_{w' \in \Sigma_{n+1}} \phi_{e_i \lambda}^{w'} \pr_{1,n} \Delta' f_{w'}.
\]

If $w'$ begins with a letter other than $\pi_0$ or $\pi_1$, then $\pr_{1,n} \Delta' f_{w'}$ is zero. Otherwise, if $w' = \pi_j w$, we have $\pr_{1,n} \Delta' f_{w'} = f_{\pi_j} \otimes f_{w}$, so that the equation becomes
\[
\sum_{w'' \in \Sigma_n} w_1 \phi_{\lambda}^{w''} f_{\pi_i} \otimes f_{w''}
=
\sum_{j=0}^1 \sum_{w \in \Sigma_n} \phi_{e_i \lambda}^{\pi_j w} f_{\pi_j} \otimes f_{w}.
\]

This implies that if $j \neq i$, then $\phi_{e_i \lambda}^{\pi_j w}=0$. It also implies that if $j = i$, then
\[
\phi_{e_i \lambda}^{\pi_j w} = w_1 \phi_{\lambda}^{w}.
\]

For $\phi_{\lambda e_i}^{w \pi_j}$, we simply apply the same argument with $\pr_{n,1}$ in place of $\pr_{1,n}$.
\end{proof}

\begin{cor}\label{cor:phis_and_pis}
For a word $\lambda$ in $\{e_0,e_1\}$, let $\pi(\lambda)$ denote the word in $\{\pi_0,\pi_1\}$ obtained by replacing $e_i$ with $\pi_i$ for $i=0,1$. Let
\[
\pr_{\pi} \colon A \to A
\]
denote projection onto the subspace generated by words in $\{\pi_0,\pi_1\}$. Then for any word $\lambda$ in $\{e_0,e_1\}$,
\[
\pr_{\pi} c^{\#}(\lambda) = f_{\pi(\lambda)}.
\]

In other words,
\[
\phi_{\lambda}^w
\]
is $1$ if $w = \pi(\lambda)$ and $0$ if $w$ is any other word in $\{\pi_0,\pi_1\}$.
\end{cor}
\begin{proof}
This follows by repeated application of Proposition \ref{prop:phi_description}.
\end{proof}

\subsection{Computing the Universal Cocycle Evaluation Map}\label{sec:compute_univ_cocycle_eval}

We compute $c^{\#}(\lambda)$ in terms of $W_1(c),w_2(c),w_3(3)$ for $\lambda = e_0, e_1, e_0 e_1, e_1 e_0, e_0 e_1 e_0, e_0 e_1^2$. This will allow us to compute $c^{\#}(J_i)$ for $i=1,2,3,4$.

\subsubsection{Degree $1$}

As in the proof of Proposition \ref{prop:phi_description}, $\glt$-equivariance and Schur's Lemma imply that
\[
\phi^{\pi_i}_{e_j} = \delta_{ij} w_1.
\]

Thus
\[
c^{\#}(e_0) = w_1 f_{\pi_0}
\]
\[
c^{\#}(e_1) = w_1 f_{\pi_1}
\]

\subsubsection{Degree $2$ Powers}

By Corollary \ref{cor:phis_and_pis} and the definition of $w_2$, we have
\[
c^{\#}(e_0 e_1) = w_1^2 f_{\pi_0 \pi_1} + w_2 f_{\tau}.
\]

Applying the matrix
\[
s \colonequals
\begin{pmatrix}
0 & 1\\
1 & 0
\end{pmatrix}
\in \glt
\]
that switches $e_0$ and $e_1$, we find
\[
c^{\#}(e_1 e_0) = w_1^2 f_{\pi_1 \pi_0} - w_2 f_{\tau}.
\]

\subsubsection{Degree $3$}

Let's compute $c^{\#}(e_0e_1e_0)$. Using Proposition \ref{prop:phi_description} applied to $\lambda = e_0 e_1$ and $\lambda = e_1 e_0$, we may check that
\[
\phi_{e_0 e_1 e_0}^{\pi_1 \tau} = \phi_{e_0 e_1 e_0}^{\tau \pi_1} = 0,
\]
while
\[
\phi_{e_0 e_1 e_0}^{\pi_0 \tau} = w_1 \phi_{e_1 e_0}^{\tau} = - w_1 w_2,
\]
and
\[
\phi_{e_0 e_1 e_0}^{\tau \pi_0} = w_1 \phi_{e_0 e_1}^{\tau} = w_1 w_2.
\]
It follows from this and Corollary \ref{cor:phis_and_pis} that
\[
c^{\#}(e_0 e_1 e_0) = w_1^3 f_{\pi_0 \pi_1 \pi_0} + w_1 w_2(f_{\tau \pi_0} - f_{\pi_0 \tau}) + w_3 f_{\sigma_0} + a f_{\sigma_1}
\]
for some function $a$ of $c$. Let $N \colonequals \begin{pmatrix}
1 & 1\\
0 & 1
\end{pmatrix}$, and let $\pr_{\sigma} \colon A_3 \to A_3$ denote projection onto the subspace spanned by $f_{\sigma_1}$ and $f_{\sigma_3}$. Then $N(e_0 e_1 e_0) = e_0 e_1 e_0 + e_0^3$, and $\pr_{\sigma}c^{\#}(e_0^3)=\pr_{\sigma} w_1^3 f_{\pi_0^3}=0$, so
\[
w_3 f_{\sigma_0} + a f_{\sigma_1} = \pr_{\sigma} c^{\#}(e_0 e_1 e_0) = \pr_{\sigma}(N(c^{\#}(e_0 e_1 e_0))
=
N(w_3 f_{\sigma_0} + a f_{\sigma_1})
=
(w_3 + a) f_{\sigma_0} + a f_{\sigma_1},
\]
which implies that $a=0$.

Applying $s$ to $c^{\#}(e_0 e_1 e_0)$ and noting that $s(f_{\sigma_0}) = - f_{\sigma_1}$, we find that
\begin{eqnarray*}
c^{\#}(e_1 e_0 e_1)
&=&
c^{\#}(s(e_0 e_1 e_0))\\
&=&
s(w_1^3 f_{\pi_0 \pi_1 \pi_0} + w_1 w_2(f_{\tau \pi_0} - f_{\pi_0 \tau}) + w_3 f_{\sigma_0})\\
&=&
w_1^3 f_{\pi_1 \pi_0 \pi_1} + w_1 w_2(-f_{\tau \pi_1} + f_{\pi_1 \tau}) - w_3 f_{\sigma_1}.
\end{eqnarray*}

We now compute $c^{\#}(e_0 e_1^2)$. Notice that $e_1 \Sha e_0 e_1 = e_1 e_0 e_1 + 2 e_0 e_1^2$, so that
\begin{eqnarray*}
2 c^{\#}(e_0 e_1)^2
&=&
c^{\#}(e_1 \Sha e_0 e_1) - c^{\#}(e_1 e_0 e_1)\\
&=&
c^{\#}(e_1) c^{\#}(e_0 e_1) - \left[w_1^3 f_{\pi_1 \pi_0 \pi_1} + w_1 w_2(-f_{\tau \pi_1} + f_{\pi_1 \tau}) - w_3 f_{\sigma_1}\right]\\
&=&
(w_1 f_{\pi_1})(w_1^2 f_{\pi_0 \pi_1} + w_2 f_{\tau}) - \left[w_1^3 f_{\pi_1 \pi_0 \pi_1} + w_1 w_2(-f_{\tau \pi_1} + f_{\pi_1 \tau}) - w_3 f_{\sigma_1}\right]\\
&=&
w_1^3(f_{\pi_1} f_{\pi_0 \pi_1} - f_{\pi_1 \pi_0 \pi_1})
+
w_1 w_2 (f_{\pi_1} f_{\tau} + f_{\pi_1 \tau} -f_{\tau \pi_1}) + w_3 f_{\sigma_1}\\
&=&
2 w_1^3 f_{\pi_0 \pi_1^2} + 2 w_1 w_2 f_{\pi_1 \tau} + w_3 f_{\sigma_1}.
\end{eqnarray*}

\subsection{The Geometric Step}\label{sec:elimination}

Let us recall that, as elements of $\Oc(U_3)$, we have
\begin{flalign*}
J_1 = e_0\\
J_2 = e_0 e_1 \\
J_3 = e_0 e_1 e_0\\
J_4 = e_0 e_1^2 + 2 e_1
\end{flalign*}

By the calculations above, we find
\begin{flalign*}
c^{\#}(J_1) = w_1 f_{\pi_0}\\
c^{\#}(J_2) = w_1^2 f_{\pi_0 \pi_1} + w_2 f_{\tau}\\
c^{\#}(J_3) = w_1^3 f_{\pi_0 \pi_1 \pi_0} + w_1 w_2(f_{\tau \pi_0} - f_{\pi_0 \tau}) + w_3 f_{\sigma_0}
\end{flalign*}

Finally, we have
\[
c^{\#}(J_4)
=
c^{\#}(e_0 e_1^2)
+
2c^{\#}(e_1)\\
=
w_1^3 f_{\pi_0 \pi_1^2} + w_1 w_2 f_{\pi_1 \tau} + w_3 f_{\sigma_1}/2 + 2 w_1 f_{\pi_1}.
\]

We now use elimination theory to find a polynomial in $J_1,J_2,J_3,J_4$ with coefficients in $A$ that maps to $0$ under $c^{\#}$.

To eliminate $w_3$, we try
\[
K \colonequals f_{\sigma_1} J_3 - 2 f_{\sigma_0} J_4
\]

Applying $c^{\#}$ to this, we get
\begin{eqnarray*}
c^{\#}(K) &=&
f_{\sigma_1}(w_1^3 f_{\pi_0 \pi_1 \pi_0} + w_1 w_2(f_{\tau \pi_0} - f_{\pi_0 \tau})) 
- 2 f_{\sigma_0}(w_1^3 f_{\pi_0 \pi_1^2} + w_1 w_2 f_{\pi_1 \tau} + 2 w_1 f_{\pi_1})\\
&=&
w_1^3(f_{\sigma_1} f_{\pi_0 \pi_1 \pi_0} - 2 f_{\sigma_0} f_{\pi_0 \pi_1^2})
+
w_1 w_2
(
f_{\sigma_1} (f_{\tau \pi_0} - f_{\pi_0 \tau})
-
2 f_{\sigma_0} f_{\pi_1 \tau}
)
- 4
w_1 (f_{\sigma_0} f_{\pi_1})
\end{eqnarray*}

Now, to eliminate $w_1 w_2$, we note
\[
c^{\#}(J_1 J_2) = (w_1 f_{\pi_0})(w_1^2 f_{\pi_0 \pi_1} + w_2 f_{\tau})
=
w_1^3 f_{\pi_0} f_{\pi_0 \pi_1} + w_1 w_2 f_{\pi_0} f_{\tau}
\]
and then set
\[
L \colonequals f_{\pi_0} f_{\tau} K - (f_{\sigma_1}(f_{\tau \pi_0} - f_{\pi_0 \tau}) - 2 f_{\sigma_0} f_{\pi_1 \tau}) J_1 J_2.
\]

Then
\begin{eqnarray*}
c^{\#}(L)
&=&
f_{\pi_0} f_{\tau}(w_1^3(f_{\sigma_1} f_{\pi_0 \pi_1 \pi_0} - 2 f_{\sigma_0} f_{\pi_0 \pi_1^2})
+
w_1 w_2
(
f_{\sigma_1} (f_{\tau \pi_0} - f_{\pi_0 \tau})
-
2 f_{\sigma_0} f_{\pi_1 \tau}
)
- 4
w_1 (f_{\sigma_0} f_{\pi_1}))\\
& & -
(f_{\sigma_1}(f_{\tau \pi_0} - f_{\pi_0 \tau}) - 2 f_{\sigma_0} f_{\pi_1 \tau})(w_1^3 f_{\pi_0} f_{\pi_0 \pi_1} + w_1 w_2 f_{\pi_0} f_{\tau})\\
&=&
f_{\pi_0} f_{\tau}(w_1^3(f_{\sigma_1} f_{\pi_0 \pi_1 \pi_0} - 2 f_{\sigma_0} f_{\pi_0 \pi_1^2})
- 4
w_1 (f_{\sigma_0} f_{\pi_1}))
-
(f_{\sigma_1}(f_{\tau \pi_0} - f_{\pi_0 \tau}) - 2 f_{\sigma_0} f_{\pi_1 \tau})(w_1^3 f_{\pi_0} f_{\pi_0 \pi_1})\\
&=&
w_1^3 (f_{\pi_0} f_{\tau}(f_{\sigma_1} f_{\pi_0 \pi_1 \pi_0} - 2 f_{\sigma_0} f_{\pi_0 \pi_1^2}) -
(f_{\sigma_1}(f_{\tau \pi_0} - f_{\pi_0 \tau}) - 2 f_{\sigma_0} f_{\pi_1 \tau})( f_{\pi_0} f_{\pi_0 \pi_1})) - 4 w_1 f_{\pi_0} f_{\pi_1} f_{\tau} f_{\sigma_0}
\end{eqnarray*}

So
\[
L
-
\frac{f_{\pi_0} f_{\tau}(f_{\sigma_1} f_{\pi_0 \pi_1 \pi_0} - 2 f_{\sigma_0} f_{\pi_0 \pi_1^2}) -
(f_{\sigma_1}(f_{\tau \pi_0} - f_{\pi_0 \tau}) - 2 f_{\sigma_0} f_{\pi_1 \tau})( f_{\pi_0} f_{\pi_0 \pi_1})}{f_{\pi_0}^3} J_1^3
+ 4 f_{\pi_1} f_{\tau} f_{\sigma_0} J_1
\]
is in the Chabauty-Kim ideal.
%\section{Coordinates on the Motivic Galois Group}

In particular, it is a linear combination over $\Oc(W)$ of $J_4$, $J_3$, $J_1 J_2$, $J_1^3$,  and $J_1$. By the discussion in \S \ref{sec:use_p-adic_periods}, this proves Theorem \ref{thm:CK_ideal_element_form}.

\section{Explicit Examples}\label{sec:explicit_examples}

We describe some examples to be carried out in future work, once we have appropriate code for computing triple integrals on punctured elliptic curves.

\subsection{128a2}\label{sec:CK_function_128a2}

We set $\Xc = \Ec'$ to be the affine curve over $\Zb[1/2]$ given by minimal Weierstrass equation
\[
y'^2 = x'^{3} + x'^{2} - 9 x' + 7,
\]
and $\Ec$ the corresponding projective curve. We let $E'$ and $E$ denote their generic fibers, respectively.

A $2$-descent shows that $E(\Qb)$ has rank $1$, with generator 
\[
(-1,-4).
\]

A simple search reveals $13$ elements of $\Ec'(\Zb[1/2])$:
\begin{align*}
(x',y') = (-3,-4), (-3,4), (-1,-4), (-1,4), (1,0), (2,-1), (2,1),\\ (3,-4), (3,4), (29/4, -155/8), (29/4, 155/8), (19, -84), (19, 84)
\end{align*}

We let $P_i$ denote the $i$th element of this list, starting with $P_0 = (-3,-4)$ and ending with $P_{12} = (19,84)$. We let $F = \{P_i\}_{0 \le i \le 12}$. Our goal is to show that $F = \Ec'(\Zb[1/2])$ using the Chabauty-Kim method. More precisely, we set $\Pi = U_3$.

\subsubsection{Coleman Integration}

We choose $p=\pf=5$ and perform the necessary Coleman integration to determine the coefficients $c_i$.

For the purposes of Coleman integration, it is easier to use a simple Weierstrass model. A simple Weierstrass model for $E$ is given by
\[
y^2 = x^3 - 12096 x + 470016
\]
We thus have $e_0 = dx/y$ and $e_1 = xdx/y$, as in \S \ref{sec:local_selmer_coords}.

The conversion between the models is given by the transformation $(x',y') \mapsto (x,y) = (u^2 x' + r,u^3y' + s u^2 x' + t)$ for
\[
(u,r,s,t) = (6, 12, 0, 0)
\]
The inverse is given by
\[
(u,r,s,t) = (1/6, -1/3, 0, 0).
\]
In particular, these models are isomorphic over $\mathbb{Z}[1/6]$ and thus $\mathbb{Z}_5$. Therefore, we may use it for all local computations at $p=5$.

Converting the $\Zb[1/2]$-points to the simple Weierstrass model, we get
\begin{align*}
(-96,-864), (-96,864), (-24,-864), (-24,864), (48,0), (84,-216), (84,216),\\ (120,-864), (120,864), (273,-4185), (273,4185), (696,-18144), (696,18144)
\end{align*}

We choose basepoint
\[
b = P_0 = (-96,-864),
\]
thus fixing $J_1$, $J_2$, $J_3$, and $J_4$ as functions on $\Ec'(\Zb_5)$.

\subsection{102a1}\label{sec:CK_function_102a1}

We set $\Xc = \Ec'$ to be the affine curve over $\Zb[1/2]$ given by minimal Weierstrass equation
\[
y'^2 + x' y' = x'^{3} + x'^{2} - 2 x',
\]
and $\Ec$ the corresponding projective curve. We let $E'$ and $E$ denote their generic fibers, respectively.

A $2$-descent shows that $E(\Qb)$ has rank $1$, with generator 
\[
(-1,-1).
\]

A simple search reveals $19$ elements of $\Ec'(\Zb[1/2])$:
\begin{align*}
(x',y') = (-2, 0), (-2, 2), (-1, -1), (-1, 2), (-\frac{1}{4}, -\frac{5}{8}), (-\frac{1}{4}, \frac{7}{8}), (0, 0), (1, -1), (1, 0), (\frac{121}{64}, -\frac{1881}{512}),\\ (\frac{121}{64}, \frac{913}{512}), (2, -4), (2, 2), (8, -28), (8, 20), (9, -33), (9, 24), (2738, -144670), (2738, 141932)
\end{align*}

We let $P_i$ denote the $i$th element of this list, starting with $P_0 = (-2,0)$ and ending with $P_{18} = (2738,141932)$. We let $F = \{P_i\}_{0 \le i \le 18}$. Our goal is to show that $F = \Ec'(\Zb[1/2])$ using the Chabauty-Kim method.

More precisely, we set $\Pi = U_3$.

\subsubsection{Local Computation at Bad Primes}\label{sec:local_comp_bad}

We make regular use of the results and notation of \S \ref{sec:bad_outside_S}-\ref{sec:local_bad}.

We have $j(E) = \frac{1771561}{612} = 2^{-2} \cdot 3^{-2} \cdot 11^6 \cdot 17^{-1}$. Therefore, the Tamagawa numbers of $E$ at $3$ and $17$ are $2$ and $1$, respectively. The condutor is $102 = 2 \cdot 3 \cdot 17$, so the curve is semistable at all primes. We have $T_0 = \{2,3,17\}$, so that $T_0 \setminus S = \{3,17\}$.

It follows that $\Ec$ is already a regular semistable model at $17$. Its special fiber has one component, so that $\kappa_{17}(\Ec'(\Zb_{17})) = *$.

The special fiber $\Ec'_{\Fb_3}$ contains one singular point, given by $(x',y') = (2,2)$. Let $\Ec''$ denote the blowup of $\Ec'$ at this point. Then $\Ec''$ is the minimal regular model of $E'$. Let $\{\alpha_1,\alpha_2\} = E(\Gamma_3(\Ec''))$, with $\alpha_1$ corresponding to the smooth locus of $\Ec'$ and $\alpha_2$ to the exceptional divisor. Then
\[
\kappa_3(\Ec'(\Zb_{3})) \subseteq \{\alpha_1,\alpha_2\},
\]
where by abuse of notation we use $\alpha_1,\alpha_2$ to refer to the corresponding elements of $H^1(G_3;\Pi)$.

For $z \in \Ec'(\Zb_3)$, we have
\begin{itemize}
    \item $\kappa_3(z) = \alpha_2$ if $z \mod 3 = (2,2)$
    \item $\kappa_3(z) = \alpha_1$ otherwise
\end{itemize}

Notice that $P_0 \mod{3} = (1,0) \neq (2,2)$, while $P_2 \mod{3} = (2,2)$, so that in fact $\kappa_3(\Ec'(\Zb_{3})) = \{\alpha_1,\alpha_2\}$ (in the language of \cite[Proposition 1.2.1(1)]{BettsDogra20}, this is saying that each component contains the reduction of a point in $\Xc(\Zb_p)$).

We set $F_i \colonequals \{z \in F \, \mid \, \kappa_3(z) = \alpha_i\}$. Then
\[
F_1 = \{P_0, P_1, P_6, P_7, P_8, P_9, P_{10}, P_{15}, P_{16}\}
\]
\[
F_2 = \{P_2, P_3, P_4, P_5, P_{11}, P_{12}, P_{13}, P_{14}, P_{17}, P_{18}\}
\]

\subsubsection{Coleman Integration}

We choose $p=\pf=5$ and perform the necessary Coleman integration to determine the coefficients $c_i$.

For the purposes of Coleman integration, it is easier to use a simple Weierstrass model. A simple Weierstrass model for $E$ is given by
\[
y^2 = x^{3} - 3267 x + 45630 
\]
We thus have $e_0 = dx/y$ and $e_1 = xdx/y$, as in \S \ref{sec:local_selmer_coords}.

The conversion between the models is given by the transformation $(x',y') \mapsto (x,y) = (u^2 x' + r,u^3y' + s u^2 x' + t)$ for
\[
(u,r,s,t) = (6, 15, 3, 0)
\]
The inverse is given by
\[
(u,r,s,t) = (1/6, -5/12, -1/2, 5/24).
\]
In particular, these models are isomorphic over $\mathbb{Z}[1/6]$ and thus $\mathbb{Z}_5$. Therefore, we may use it for all local computations at $p=5$.

Converting the $\Zb[1/2]$-points to the simple Weierstrass model, we get
\begin{align*}
(x,y) = (-57, -216), (-57, 216), (-21, -324), (-21, 324), (6, -162), (6, 162), (15, 0), (51, -108),\\ (51, 108), (1329/16, -37719/64), (1329/16, 37719/64), (87, -648), (87, 648), (303, -5184),\\ (303, 5184), (339, -6156), (339, 6156), (98583, -30953016), (98583, 30953016)
\end{align*}

We first work with $F_2$. We choose basepoint
\[
b_2 = P_2 = (-21,-324),
\]
thus fixing $J_1$, $J_2$, $J_3$, and $J_4$ as functions on $\Ec'(\Zb_5)$.

To be continued.

\subsection{Estimation of Zeroes of \texorpdfstring{$p$}{p}-adic Power Series}\label{sec:newton_polygons}

We describe how the known $\Zb[1/2]$-points fit into $5$-adic residue discs.

\subsubsection{128a2}

The set $\Ec'(\Fb_5)$ is the $7$-element set given by
\[
(x,y) = (0,1), (0,4), (1,1), (1,4), (3,0), (4,1), (4,4)
\]
We let $R_i$ denote the residue disc of the $i$th point in the list, starting with $R_1$ the residue disc of $(0,0)$. We note the intersection of $F$ with each residue disc:

\begin{eqnarray*}
F \cap R_1 &=& \{ P_7\} \\
F \cap R_2 &=& \{ P_8 \} \\
F \cap R_3 &=& \{ P_2, P_{11}\} \\
F \cap R_4 &=& \{P_3, P_{12}\} \\
F \cap R_5 &=& \{P_4, P_9, P_{10}\} \\
F \cap R_6 &=& \{P_0, P_6\} \\
F \cap R_7 &=& \{ P_1, P_5\} \\
\end{eqnarray*}

For $i \neq 1$, we may choose a constant plus $x$ as an analytic uniformizer on $R_i$.

\subsubsection{102a1}

The set $\Ec'(\Fb_5)$ is the $10$-element set given by
\[
(x,y) = (0,0), (0,1), (1,2), (1,3), (2,2), (2,3), (3,1), (3,4), (4,1), (4,4)
\]
We let $R_i$ denote the residue disc of the $i$th point in the list, starting with $R_1$ the residue disc of $(0,0)$. We note the intersection of $F$ with each residue disc:

\begin{eqnarray*}
F \cap R_1 &=& \{P_6 \} \\
F \cap R_2 &=& \emptyset \\
F \cap R_3 &=& \{P_5, P_7 \} \\
F \cap R_4 &=& \{P_4, P_8\} \\
F \cap R_5 &=& \{P_{11} \} \\
F \cap R_6 &=& \{P_{12}\} \\
F \cap R_7 &=& \{P_1, P_{13}, P_{18} \} \\
F \cap R_8 &=& \{P_0, P_{14}, P_{17}  \} \\
F \cap R_9 &=& \{P_2,P_{10}, P_{16}  \} \\
F \cap R_{10} &=& \{P_3,P_9, P_{15} \}
\end{eqnarray*}

For $i \neq 1$, we may choose a constant plus $x$ as an analytic uniformizer on $R_i$.

\appendix
\renewcommand{\thesection}{A}

\section{Unipotent Groups and Cohomology}

\subsection{Unipotent Groups}\label{sec:unipotent_background}

Let $U$ be a pro-unipotent group over a field $K$. Then we have the Lie algebra $\Lie{U}$ and its universal enveloping algebra $\Uc U$, along with the coordinate ring $\Oc(U)$. The first two are naturally a Lie algebra object and cocommutative Hopf algebra object, respectively, of $\Pro \Vect^{\fin}_K$, while the latter is a commutative Hopf algebra object of $\Ind \Vect^{\fin}_K$. The natural duality between $\Pro \Vect^{\fin}_K$ and $\Ind \Vect^{\fin}_K$ induces a natural isomorphism
\[
\Oc(U) \cong \Uc U^{\vee}.
\]

If $\Pi$ is a unipotent group with an action of $U$ (equivalently, $\Pi$ is a unipotent group in the Tannakian category $\Rep_K(U)$ in the sense of \cite[\S 5]{Deligne89}), then
$M \colonequals \Uc \Pi$ has the structure of a module over
\[
A \colonequals \Uc U.
\]
If $\rho \colon A \otimes M \to M$ denotes the multiplication map, then
\[
\rho \circ (\id_A \otimes \mult_M) = \mult_M \circ (\rho \otimes \rho) \circ (\Delta_A \otimes \id_{M \otimes M}) \colon A \otimes M \otimes M \to M,
\]
where $\rho \otimes \rho$ sends $u_1 \otimes u_2 \otimes \pi_1 \otimes \pi_2$ to $\rho(u_1 \otimes \pi_1) \otimes \rho(u_2 \otimes \pi_2)$. This implies that
\[
\Lie{U} \subseteq \Uc U = A
\]
acts via derivations on $M$, while
\[
U \subseteq \Uc U = A
\]
acts via automorphisms on $M$. We also have
\[
\Delta_M \circ \rho = (\rho \otimes \rho) \circ (\Delta_A \otimes \Delta_M) \colon A \otimes M \to M \otimes M,
\]
which implies that the action of $\Lie{U} \subseteq A$ preserves $\Lie{\Pi} \subseteq M$, and the action of $U \subseteq A$ preserves both $\Lie{\Pi}$ and $\Pi$ in $M$, acting via derivations and automorphisms,\footnote{In fact, they preserve the coalgebra structure, in that elements of $\Lie{U}$ act as coderivations, and elements of $U$ act as automorphisms of the Hopf algebra.} respectively.

For an object $W$ of $\Pro \Vect^{\fin}_K$, we denote by
\[
\FreeLie{W}
\]
the free pro-nilpotent Lie algebra on $W$.

\begin{prop}\label{prop:lie_cocycles}
We have
\[
Z^1(U;\Pi) = Z^1(\nf(U);\Lie{\Pi}),
\]
where
\[
Z^1(\nf(U);\Lie{\Pi})
\]
denotes the space of $K$-linear maps
\[
c \colon \nf(U) \to \Lie{\Pi}
\]
satisfying the cocycle condition
\[
c([g,h]) = [c(g),c(h)] + g(c(h)) - h(c(g)),
\]
where $\nf(U)$ acts on $\Lie{\Pi}$ by derivations.
\end{prop}
\begin{proof}
Either can easily be seen to be in bijection with the space of $K$-linear maps
\[
c \colon A \to M
\]
satisfying the two conditions
\begin{itemize}
    \item $(c \otimes c) \circ \Delta_A = \Delta_M \circ c \colon A \to M \otimes M$ (i.e., $c$ is a homomorphism of coalgebras)
    \item $c \circ \mult_{A} = \mult_{M} \circ (\id_{M} \otimes \rho) \circ (c \otimes \id_{A} \otimes c) \circ (\Delta_A \otimes \id_{A}) \colon A \otimes A \to M$.
\end{itemize}

\end{proof}

\subsection{Extension of a Reductive Group by a Unipotent Group}\label{sec:background_red_unip}

Let $\Gb$ be a reductive group acting on $U$, and let
\[
G \colonequals \Gb \ltimes U.
\]

We suppose that the action of $U$ on $\Pi$ extends to an action of $G$ on $\Pi$. The subgroup $\Gb \subseteq G$ acts on $U$ by conjugation and on $\Pi$ via the latter's action of $G$.

The action of $\Gb$ on $U$ induces an action on $\Lie{U}$ and on $\Uc(U)$ by Lie and Hopf algebra automorphisms, respectively. We also let $U$ act on $\Uc(U)$ by left-translation. This is compatible with the action of $\Gb$ on $\Uc(U)$ in that it extends to an action of
\[
G
\]
on
\[
\Uc(U).
\]

\begin{rem}\label{rem:glt_equiv}
For $g \in \glt$, $u \in U(\Oc_{k,S},E)$, and $\pi \in \Pi$, we have
\[
(g(u)) (g \pi) = (g u g^{-1}) (g \pi) = g(u \pi),
\]
implying that the action map
\[
U(\Oc_{k,S},E) \times \Pi \to \Pi
\]
is $\glt$-equivariant. The same is true of the associated Lie algebra and universal enveloping algebra actions as in \S \ref{sec:unipotent_background}.
\end{rem}

\subsection{Cohomology and Cocycles}\label{sec:cohom_cocycles}
 The main goal of this section is to prove a generalization of \cite[Proposition 5.2]{MTMUE}.

Some computations below use the fact (which follows from the definition of semidirect product) that for $g \in \Gb$, $u \in U$, and $\pi \in \Pi$, we have
\[
g(u)(\pi) = g(u(g^{-1}(\pi))).
\]

We recall some facts about nonabelian cocycles. Some of our discussion borrows from \cite[\S 6.1]{BrownIntegral}. Although we do not denote it in the notation, all points are relative to a $K$-algebra $R$.

For a cocycle $c \colon G \to \Pi$ and $g \in G$, we denote $c(g)$ by $c_g$. We recall that for a $K$-algebra $R$, we have
\[
Z^1(G;\Pi)(R) \colonequals \{ c \colon G_R \to \Pi_R \, \mid \, c_{g_1 g_2} = c_{g_1} g_1(c_{g_2}) \, \forall \, g_1, g_2 \in G\},
\]
where the cocycle condition is imposed on the functor of points. We similarly define $Z^1(U;\Pi)$. The trivial cocycle $1 \in Z^1(G;\Pi)$ (resp. $Z^1(U;\Pi)$) sends all of $G$ (resp. $U$) to the identity of $\Pi$. %We note

Recall also that $\Pi$ acts on $Z^1(G;\Pi)$ (resp. $Z^1(U;\Pi)$) by sending $\pi \in \Pi$ and $c \in Z^1(G;\Pi)$ (resp. $c \in Z^1(U;\Pi)$) to
\[
g \mapsto \pi c_g g(\pi^{-1})
\]
(resp. $u \mapsto \pi c_u u(\pi^{-1})$) and that
\[
H^1(G;\Pi) \colonequals Z^1(G;\Pi)/\Pi.
\]

\begin{lemma}\label{lemma:cocycles_homs}
We have
\[
Z^1(G;\Pi) \cong \ker(\Hom(G,\Pi \rtimes G) \to \Hom(G,G)).
\]
given by
\[
c \mapsto \rho_c,
\]
where $\rho_c(g) \colonequals c_g g \in \Pi \rtimes G$. The action of $\Pi$ is given by conjugation on $\Pi \rtimes G$. % (resp. $\Pi \rtimes \Gb$).

%If $U$ acts trivially on $\Pi$, we have
%\[
%Z^1(G;\Pi) \cong \ker(\Hom(G,\Pi \rtimes \Gb) \to \Hom(G,\Gb)),
%\]
%compatible with the above via a quotient map
%\[
%\Pi \rtimes G \to \Pi \rtimes \Gb.
%\]

\end{lemma}
\begin{proof}
Since
\[
\rho_c(g_1) \rho_c(g_2) =  (c_{g_1} g_1)(c_{g_2} g_2) =  c_{g_1} g_1(c_{g_2}) g_1 g_2 = c_{g_1} g_1(c_{g_2}) g_1 g_2,
\]
while
\[
\rho_c(g_1 g_2) = c_{g_1 g_2} g_1 g_2,
\]
the homomorphism condition for $c$ is equivalent to the cocycle condition for $\rho$.

%The second equality follows by a similar argument. Note that if $U$ acts trivially on $\Pi$, then the subgroup $\Pi \rtimes U \subseteq \Pi \rtimes G$ is isomorphic to $\Pi \times U$, so that $U$ is normal in $\Pi \rtimes G$ with $\Pi \rtimes G/U \cong \Pi \rtimes \Gb$.

To check the action of $\Pi$, note that for $\pi \in \Pi$ and $g \in G$, we have
\begin{eqnarray*}
\pi \rho_c(g) \pi^{-1}
&=&
\pi (c_g g) \pi^{-1}\\
&=&
\pi c_g g \pi^{-1}\\
&=& 
\pi c_g g(\pi^{-1}) g\\
&=&
\pi(c)_g g\\
&=&
\rho_{\pi(c)}(g),
\end{eqnarray*}
as desired.% The same computation holds if $g \in U$.
\end{proof}

% True but not necessary:
%The stabilizer of $1 \in Z^1(G;\Pi)$ (resp. $1 \in Z^1(U;\Pi)$) is $\Pi^{G}$ (resp. $\Pi^U$). In particular, when $\Pi$ is abelian, we have an exact sequence
%\[
%0 \to \Pi^U \to \Pi \to  Z^1(U;\Pi) \to H^1(U;\Pi) \to 0
%\]
%and similarly for $G$ in place of $U$.

Recall that $\Gb$ acts on $Z^1(U;\Pi)$
by
\[
g(c)_u = g(c_{g^{-1}(u)}),\]
and the $\Pi$- and $\Gb$-actions on $Z^1(U;\Pi)$ are compatible in that
\[
g(\pi(c)) = g(\pi)(g(c))),
\]
so that the $\Gb$-action induces an action on
\[
H^1(U; \Pi).
\]

The main result of this section is the following generalization of \cite[Proposition 5.2]{MTMUE}:

\begin{thm}\label{thm:cohom}
Suppose that $\Pi^{\Gb}=1$. Then each equivalence class of cocycles $[c]$ in $H^1(G; \Pi)$ contains a unique representative $c_0$ such that
\[
c_0(g) = 1 \mbox{ for each } g \in \Gb,
\]
and its restriction to $U$ is a $\Gb$-equivariant cocycle
\[
\restr{c_0}{U} \colon  U \to \Pi.
\]
The map
\[
[c] \mapsto \restr{c_0}{U}
\]
defines a bijection
\[
H^1(G, \Pi) \cong Z^1(U; \Pi)^{\Gb}.
\]
\end{thm}

We prove this result via a series of lemmas.

\begin{lemma}\label{lemma:g-equiv_cocycle}
The restriction map
\[
Z^1(G;\Pi) \to Z^1(U;\Pi)
\]
induces a bijection
\[
\ker(Z^1(G;\Pi) \to Z^1(\Gb;\Pi)) \cong Z^1(U;\Pi)^{\Gb}
\]
\end{lemma}
\begin{proof}
Suppose $c \in \ker(Z^1(G;\Pi) \to Z^1(\Gb;\Pi))$. Then for $u \in U$, we have
\begin{eqnarray*}
g(c)_u
&=&
g(c_{g^{-1}(u)})\\
&=&
g(c_{g^{-1} u g})\\
&=&
g(c_{g^{-1}u} g^{-1}(u(c_g)))\\
&=&
g(c_{g^{-1}u})\\
&=&
g(c_{g^{-1}} g^{-1}(c_u))\\
&=&
g(g^{-1}(c_u))\\
&=&
c_u
\end{eqnarray*}

Conversely, suppose $c \in Z^1(U;\Pi)^{\Gb}$. We extend $c$ to an element of $Z^1(G;\Pi)$ as follows. For $ug \in G$, we set
\[
c_{ug} = c_u.
\]
It is clear by construction that it is trivial when restricted to $\Gb$. To show that it is a cocycle, suppose we have $g_1,g_2 \in \Gb$ and $u_1,u_2 \in U$. Then
\begin{eqnarray*}
c_{u_1 g_1 u_2 g_2}
&=&
c_{u_1 g_1 u_2 g_1^{-1} g_1 g_2}\\
&=&
c_{u_1 g_1(u_2) g_1 g_2}\\
&=&
c_{u_1 g_1(u_2)}\\
&=&
c_{u_1} u_1(c_{g_1(u_2)})\\
&=&
c_{u_1} u_1(g_1(c_{u_2}))\\
&=&
c_{u_1 g_2} u_1(g_1(c_{u_2 g_2})),
\end{eqnarray*}
as desired.

Given $c \in Z^1(U;\Pi)^{\Gb}$, it is clear that the maps are mutual inverses. Given $c \in \ker(Z^1(G;\Pi) \to Z^1(\Gb;\Pi))$, note that
\[
c_{ug} = c_u,
\]
so the maps are indeed mutual inverses.

\end{proof}

\begin{lemma}\label{lemma:hss}
The natural map
\[
H^1(G;\Pi) \to H^1(U;\Pi)
\]
induces an isomorphism
\[
H^1(G;\Pi) \cong H^1(U;\Pi)^{\Gb}
\]
\end{lemma}
\begin{proof}
As $\Gb$ is reductive, all higher group cohomology of $\Gb$ vanishes, so the Hochschild-Serre spectral sequence for the inclusion $U \hookrightarrow G$ degenerates. In particular, we get
\[
H^1(G;\Pi) = H^0(\Gb;H^1(U;\Pi)) = H^1(U;\Pi)^{\Gb}.
\]
\end{proof}

By Lemma \ref{lemma:g-equiv_cocycle} and Lemma \ref{lemma:hss}, we have a diagram
\[
\xymatrix{
\ker(Z^1(G;\Pi) \to Z^1(\Gb;\Pi)) \ar[r] \ar[d]_-{\cong} & H^1(G;\Pi) \ar[d]^{\cong}\\
Z^1(U;\Pi)^{\Gb} \ar[r] & H^1(U;\Pi)^{\Gb}
}
\]

It now suffices to prove that either horizontal arrow of the diagram is an isomorphism. We prove this for the top horizontal arrow whenever $\Pi^{\Gb}=0$.

\begin{lemma}\label{lemma:splittings}
Suppose $\Pi^{\Gb}=0$. Then the set of splittings of $\Pi \rtimes \Gb \to \Gb$ forms a pointed torsor under $\Pi$ acting by conjugation.
\end{lemma}

\begin{proof}
To check transitivity, note that the set of splittings modulo $\Pi$ is the cohomology set $H^1(\Gb;\Pi)$. But this vanishes because $\Gb$ is reductive. The action is simply transitive because the stabilizer of a section is $\Pi^{\Gb}$, which is trivial by assumption. The point of the torsor is the trivial splitting $\Gb \hookrightarrow \Pi \rtimes \Gb$.
\end{proof}

\begin{lemma}
The set
\[
\ker(Z^1(G;\Pi) \to Z^1(\Gb;\Pi)) \subseteq Z^1(G;\Pi)
\]
maps bijectively onto $H^1(G;\Pi)$.
\end{lemma}

\begin{proof}
By Lemma \ref{lemma:cocycles_homs}, under the identification
\[
Z^1(G;\Pi) = \ker(\Hom(G,\Pi \rtimes G) \to \Hom(G,G)),
\]
the action of $\Pi$ is by conjugation. Via the inclusion $\Gb \subseteq G$, we may consider $\Pi \rtimes \Gb$ as a subgroup of $\Pi \rtimes G$ (it is not normal, but this does not matter). For any $\rho \in \ker(\Hom(G,\Pi \rtimes G) \to \Hom(G,G))$, $\restr{\rho_c}{\Gb}$ is a splitting of $\Pi \rtimes \Gb \to \Gb$.

For $\pi \in \Pi$, we have $\restr{\rho_{\pi(c)}}{\Gb} = \pi \restr{\rho_c}{\Gb} \pi^{-1}$. Therefore, by Lemma \ref{lemma:splittings}, there is a unique $\pi \in \Pi$ for which $\restr{\rho_{\pi(c)}}{\Gb}$ is the trivial splitting.

Now $\restr{\rho_{\pi(c)}}{\Gb}$ is the trivial splitting iff $\restr{c}{\Gb}$ is trivial, i.e. if $c \in \ker(Z^1(G;\Pi) \to Z^1(\Gb;\Pi))$. Therefore, each class $[c] \in H^1(G;U)$ contains a unique representative in $\ker(Z^1(G;\Pi) \to Z^1(\Gb;\Pi))$.
\end{proof}

\begin{rem}
One may alternatively show that the bottom horizontal arrow is an isomorphism, as follows. If $\Pi$ is abelian, this amounts to considering the short exact sequence
\[
0 \to \Pi/\Pi^U \to Z^1(U,\Pi) \to H^1(U,\Pi) \to 0
\]
and then taking $\Gb$-fixed points and using the fact that $\Gb$ is reductive and $\Pi^{\Gb}=0$ to get
\[
Z^1(U,\Pi)^{\Gb} \cong H^1(U,\Pi)^{\Gb}.
\]

One may replicate this argument for general $\Pi$ by showing that if $\Pi$ is a nonabelian group acting on a pointed set (or variety) $Y$, and both have a compatible action of $\Gb$, then there is a long exact sequence of pointed sets
\[
\Pi^{\Gb} \to Y^{\Gb} \to (Y/\Pi)^{\Gb} \to H^1(\Gb;\Pi).
\]
\end{rem}

\bibliography{references,etale_homotopy}%%%%%%
%%%%%%%%%%%%%%%%%%
\bibliographystyle{alpha}%%%
%%%%%%%%%%%%%%%%%
%%%%%%%%%%%%%%%%%

\vfill

 \Small\textsc{D.C. Department of Mathematics, UC Berkeley}. {Email address:} \texttt{corwind@alum.mit.edu}

\end{document}